\documentclass[12pt]{amsart}

\usepackage[english]{babel}
\usepackage[T1]{fontenc}
\usepackage[utf8]{inputenc} 
\usepackage{amssymb,latexsym}
\usepackage{mathtools}
\usepackage{enumerate}  
\usepackage{a4wide}    

\usepackage[colorlinks=true,citecolor=blue,linkcolor=blue]{hyperref}
\usepackage{cleveref}
\usepackage{latexsym,amsmath,amssymb}
\usepackage{accents}
\usepackage{color}  

\usepackage[colorinlistoftodos,prependcaption,textsize=tiny]{todonotes}

\title[Calderon-Zygmund type estimates for nonlocal PDE]{Calderon-Zygmund type estimates for nonlocal PDE with H\"older continuous kernel}

\author{Tadele Mengesha}
\address[Tadele Mengesha]{Department of Mathematics, The University of Tennessee, Knoxville, 204 Ayres Hall, 1403 Circle Drive
Knoxville, TN, 37996.
}
\email{mengesha@utk.edu}

\author{Armin Schikorra}
\address[Armin Schikorra]{Department of Mathematics,
University of Pittsburgh,
301 Thackeray Hall,
Pittsburgh, PA 15260, USA}
\email{armin@pitt.edu}

\author{Sasikarn Yeepo}
\address[Sasikarn Yeepo]{Department of Mathematics and Computer Science, Faculty of Science, Chulalongkorn University, Bangkok 10330, Thailand}
\email{6072857023@student.chula.ac.th}

plus 6pt minus 12pt
\def\eps{\varepsilon}

\def\N{{\mathbb N}}

\renewcommand{\div}{\operatorname{div}}

\newcommand{\subsubset}{\subset\subset}

\newtheorem{theorem}{Theorem}

\newtheorem{lemma}[theorem]{Lemma}
\newtheorem{corollary}[theorem]{Corollary}
\newtheorem{proposition}[theorem]{Proposition}

\theoremstyle{definition}

\newtheorem{definition}[theorem]{Definition}

\def\diam{{\rm diam\,}}
\def\dist{{\rm dist\,}}

\def\lip{{\rm Lip\,}}

\def\supp{{\rm supp\,}}

\newcommand{\R}{\mathbb{R}}

\newcommand{\brac}[1]{\left (#1 \right )}
\newcommand{\abs}[1]{\left |#1 \right |}

\newcommand{\barint}{
\rule[.036in]{.12in}{.009in}\kern-.16in \displaystyle\int }

\newcommand{\barcal}{\mbox{$ \rule[.036in]{.11in}{.007in}\kern-.128in\int $}}
\def\mvint_#1{\mathchoice
          {\mathop{\vrule width 6pt height 3 pt depth -2.5pt
                  \kern -8pt \intop}\nolimits_{\kern -3pt #1}}%
          {\mathop{\vrule width 5pt height 3 pt depth -2.6pt
                  \kern -6pt \intop}\nolimits_{#1}}%
          {\mathop{\vrule width 5pt height 3 pt depth -2.6pt
                  \kern -6pt \intop}\nolimits_{#1}}%
          {\mathop{\vrule width 5pt height 3 pt depth -2.6pt
                  \kern -6pt \intop}\nolimits_{#1}}}

\numberwithin{theorem}{section} \numberwithin{equation}{section}

\newcommand{\lap}{\Delta }

\newcommand{\aleq}{\lesssim}
\newcommand{\ageq}{\gtrsim}
\newcommand{\aeq}{\approx}
\newcommand{\Rz}{\mathcal{R}}
\newcommand{\laps}[1]{(-\lap) ^{\frac{#1}{2}}}

\newcommand{\lapms}[1]{I^{#1}}

\begin{document}

\begin{abstract}
We study interior $L^p$-regularity theory, also known as Calderon-Zygmund theory, of the equation
\[
 \int_{\R^n} \int_{\R^n} \frac{K(x,y)\ (u(x)-u(y))\, (\varphi(x)-\varphi(y))}{|x-y|^{n+2s}}\, dx\, dy = \langle f, \varphi \rangle \quad \varphi \in C_c^\infty(\R^n).
\]
For $s \in (0,1)$, $t \in  [s,2s]$, $p \in [2,\infty)$, $K$ an elliptic, symmetric, \emph{H\"older continuous} kernel, if $f \in \brac{H^{t,p'}_{00}(\Omega)}^\ast$, then the solution $u$ belongs to $H^{2s-t,p}_{loc}(\Omega)$ as long as $2s-t < 1$.

The increase in differentiability and integrability is independent of the H\"older coefficient of $K$. For example, our result shows that if $f\in L^{p}_{loc}$ then $u\in H^{2s-\delta,p}_{loc}$ for any $\delta\in (0, s]$ as long as $2s-\delta < 1$. This is different than the classical analogue of divergence-form equations $\div(\bar{K} \nabla u) = f$ (i.e. $s=1$) where a $C^\gamma$-H\"older continuous coefficient $\bar{K}$ only allows for estimates of order $H^{1+\gamma}$. In fact, it is another appearance of the differential stability effect observed in many forms by many authors for this kind of nonlocal equations -- only that in our case we do not get a ``small'' differentiability improvement, but all the way up to $\min\{2s-t,1\}$.

The proof argues by comparison with the (much simpler) equation
\[
 \int_{\R^n} K(z,z) \laps{t} u(z) \, \laps{2s-t} \varphi(z)\, dz = \langle g,\varphi\rangle \quad \varphi \in C_c^\infty(\R^n).
\]
and showing that as long as $K$ is H\"older continuous and $s,t, 2s-t \in (0,1)$ then the ``commutator''
\[
 \int_{\R^n} K(z,z) \laps{t} u(z) \, \laps{2s-t} \varphi(z)\, dz  - c\int_{\R^n} \int_{\R^n} \frac{K(x,y)\ (u(x)-u(y))\, (\varphi(x)-\varphi(y))}{|x-y|^{n+2s}}\, dx\, dy
\]
behaves like a lower order operator.

\end{abstract}

\maketitle
{\tiny 
\tableofcontents
}

% \section{Introduction}
\section{Introduction and statement of main results}
In this article, we develop the Calderon-Zygmund theory for a popular nonlocal equation 
% \begin{equation}\label{eq:nonloc}
% \langle\mathcal{L}^{s}_{\R^n}u, \varphi\rangle = \langle f\,, \varphi \rangle
% \end{equation}
\begin{equation}\label{eq:nonloc}
\mathcal{L}^{s}_{\Omega}u = f, 
\end{equation}
where $\Omega\subseteq \R^{n}$ is an open set, $s \in (0,1)$, and the operator $\mathcal{L}^{s}_\Omega$ is formally given by
\[
\mathcal{L}^{s}_{\Omega}u(x) :=  P.V. \int_{\Omega} K(x,y)\frac{u(x)-u(y)}{|x-y|^{n+2s}}dy.
\]
The ``coefficient of $\mathcal{L}^{s}_{\Omega}$'' is $K:\R^{n}\times\R^{n}\to \R$, and it is assumed to be measurable, symmetric, and bounded. Moreover we assume $K$ to be bounded from below on the diagonal by a positive number, $\inf_x K(x,x) > 0$, which corresponds to ellipticity.

In the event that $K=1$ and $\Omega= \R^{n}$, the operator $\mathcal{L}^{s}_{\Omega}$ corresponds to the well-known fractional Laplacian operator $(-\Delta)^{s}$. 

% % % % Repeated below in definition
% For $u$ in $H^{s,2}(\R^{n})$, which is the fractional Sobolev space of $L^2(\R^n)$-functions $u$ such that $\laps{s}  u\in L^{2}(\R^{n})$, the operator $\mathcal{L}^{s}_{\R^n}$ is defined via the action on $\varphi\in C_c^{\infty}(\R^{n})$ as  
% \[
% \langle\mathcal{L}^{s}_{\R^n}u, \varphi\rangle := \int_{\R^n} \int_{\R^n} K(x,y)\frac{(u(x)-u(y))\, (\varphi(x)-\varphi(y))}{|x-y|^{n+2s}}\, dx\, dy. 
% \]
% In this case, $\mathcal{L}^{s}_{\R^n}u\in (H^{s,2}(\R^{n}))^\ast$, the set of bounded linear functions on $H^{s,2}(\R^{n})$,
% and we understand \eqref{main-eqn} as an equation in $(H^{s,2}(\R^{n}))^\ast$ in the sense that 
% for $f\in  (H^{s,2}(\R^{n}))^\ast$, we say that $u\in H^{s,2}(\R^{n})$ is a solution to \eqref{main-eqn} if
% \begin{equation}\label{eq:nonloc}
% \langle\mathcal{L}^{s}_{\R^n}u, \varphi\rangle = \langle f\,, \varphi \rangle
% \end{equation}
% for all $\varphi\in C_c^{\infty}(\R^{n})$. 
The main objective of this paper is to address the question of regularity of such a solution $u$ relative to the data $f$.

Before we state our main theorem, \Cref{th:main}, we need some definitions. We say that $K$ satisfies a uniform H\"older continuity assumption if there exists $\alpha\in (0, 1)$, $\Lambda>0$ such that
\begin{equation}\label{H-Continuity}  
\sup_{z\in \R^n} |K(z,y)-K(z,x)| \leq \Lambda\, |x-y|^\alpha, \quad \text{for $x, y\in \R^{n}$}. 
\end{equation}
For given positive numbers  $\lambda, \Lambda$ and $\alpha\in(0, 1)$, define the class of elliptic coefficients  
\[
\mathcal{K}(\alpha,\lambda, \Lambda) = \left\{K: K(x,y) = K(y,x), \inf_{x \in \R^n} K(x,x) > \lambda, \|K\|_{L^{\infty}} <\frac{1}{\lambda}\,\text{and satisfies \eqref{H-Continuity}} \right\}.
\]
% % where in the above by $\lambda<K|_{\{x=y\}}$ we mean that $K$ is bounded from below on the diagonal, 
% % \[
% % K(x, x) > \lambda,\quad \text{for almost all $x\in \R^{n}$}. 
% % \]
% It is one of the remarkable features of our results that we only need the positivity of coefficients along the diagonal -- which corresponds more to a classical ellipticity assumption that $K(x,x)  
We also need to introduce relevant  
differential operators as  well  as function spaces. 
Let $\mathcal{F}$ denote the Fourier transform. For $s > 0$ the fractional Laplacian $\laps{s}$  is defined as the operator that for $f$ in the Schwartz class acts as multiplier with symbol $c|\xi|^s$
\begin{equation}\label{eq:fraclapFT}
 \mathcal{F}(\laps{s} f)(\xi) = c\, |\xi|^s \mathcal{F} f(\xi). 
\end{equation}
The constant $c$ depends on $n$ and $s$ and plays no deeper role in the theory that we consider.  Next we will introduce two types of fractional Sobolev spaces that we need to state the main result: Bessel potential spaces $H^{s,p}$ and Besov spaces $W^{s,p}$. 
For $1<p<\infty,$ the Bessel potential spaces $H^{s,p}(\R^{n})$ are defined as follows: $f \in H^{s,p}(\R^n)$ if $f \in L^p(\R^n)$ and $\laps{s} f \in L^p(\R^n)$. The associated norm is
\[
 \|f\|_{H^{s,p}(\R^n)} := \|f\|_{L^p(\R^n)} + \|\laps{s} f\|_{L^p(\R^n)}.
\]
The Besov spaces $W^{s,p}(\Omega)$, for $s \in (0,1)$, are induced by the semi-norm (called Sobolev-Slobodeckij or Gagliardo norm)
\[
 [f]_{W^{s,p}(\Omega)} = \brac{\int_{\Omega} \int_{\Omega} \frac{|f(x)-f(y)|^{p}}{|x-y|^{n+sp}}\, dx\, dy}^{\frac{1}{p}},
\]
and $\|\cdot\|_{W^{s,p}(\Omega)} = \| \cdot\|_{L^{p}(\Omega)} +  [\cdot]_{W^{s,p}(\Omega)}$ serves as a norm.  For $p=2$, $W^{s,2}(\R^{n}) = H^{s,2}(\R^{n})$, for $p < 2$ we have $W^{s,p}(\R^n) \subsetneq H^{s,p}(\R^{n})$ and for $p>2$ we have $H^{s,p}(\R^n) \subsetneq W^{s,p}(\R^{n})$. These spaces are particular examples of the more general Triebel-Lizorkin spaces and $F^{s}_{pp}(\R^{n}) = W^{s,p}(\R^{n})$ and $F^s_{p,2}(\R^{n}) = H^{s,p}(\R^{n})$, see \cite{RS96}. 

% However, the spaces are not too far away from each other. If $f \in W^{s,p}(\R^{n})$ then $f \in H^{s-\eps,p}_{loc}(\R^{n})$ for any $\eps > 0$, and if $f \in H^{s,p}(\R^{n})$ then $f \in W^{s-\eps,p}_{loc}(\R^{n})$ for any $\eps > 0$. 

For $u\in W^{s, 2}(\Omega),$  we define the map $\mathcal{L}^{s}_{\Omega}$ by 
\begin{equation}\label{def-op-Omega}
\langle\mathcal{L}^{s}_{\Omega}u, \varphi\rangle  :=  \int_{\Omega} \int_{\Omega} K(x,y)\frac{(u(x)-u(y))\, (\varphi(x)-\varphi(y))}{|x-y|^{n+2s}}\, dx\, dy. 
\end{equation}
for any $\varphi\in W^{s,2}(\Omega)$. 
It  is not difficult to show that if $K\in L^{\infty}(\Omega\times \Omega)$, then for any $u\in W^{s,2}(\Omega)$,
$\mathcal{L}^{s}_{\Omega}u\in (W^{s,2}(\Omega))^{*}. $ 

We now define precisely what we mean by a solution to our equation of interest, \eqref{eq:nonloc}. 
\begin{definition}\label{Def-loc-sol}
Let  $s\in(0, 1)$ and $t\in [s, 2s)$. Suppose that $f_1, f_2\in L^{2}(\R^n)$.   We  say $u\in W^{s, 2}(\Omega)$ is a distributional solution of 
\begin{equation}\label{Def-loc-sol-eqn}
 \mathcal{L}^{s}_{\Omega}u = \laps{2s-t} f_1 + f_2  \quad \text{in $\Omega_1$}
\end{equation}
for some $\Omega_1 \subseteq \Omega$ if for any $\varphi\in C_c^{\infty}(\Omega_1)$, it holds that 
\[
\langle\mathcal{L}^{s}_{\Omega}u, \varphi\rangle=\int_{\R^{n}}  f_1\laps{2s-t} \varphi\, dx  + \int_{\R^{n}} f_2 \varphi\, dx. 
\]
\end{definition}
If $\Omega$ is bounded or $\Omega=\R^{n}$, the notion of solution introduced in Definition \ref{Def-loc-sol} coincides with the classical notion of weak solution. Moreover, for $\Omega=\R^{n}$ and for any  bounded open subset $\Omega_1$, given $f_1, f_2\in L^{2}(\R^n)$, a solution to \eqref{Def-loc-sol-eqn} exists with additional assumption on $u$. For example, a minimizer of the energy 
\[
\mathcal{E}(u) := \frac{1}{2} \langle\mathcal{L}^{s}_{\R^{n}}u,u\rangle + \int_{\R^{n}}  f_1\, \laps{2s-t} u dx  + \int_{\R^{n}} f_2\, u dx
\]
 over $\{u\in H^{s,2}(\R^{n}): u=0 \quad \text{on $\R^{n}\setminus \Omega_1$}\}$ exists and is a solution to \eqref{Def-loc-sol-eqn} in the sense of \Cref{Def-loc-sol-eqn}. 
 
We also notice that \eqref{Def-loc-sol-eqn} is often thought as the nonlocal (fractional) analogue of the weak formulation of the elliptic differential equation 
\begin{equation}\label{classical-eqn}
\div(A(\cdot) \nabla u) = \div{{h} } + g.
\end{equation}
The question of regularity of weak solutions $u$ to \eqref{classical-eqn} in relation to the regularity of data (the coefficient $A$, the right hand sides ${h}$ and $g$) is decades old.  One line of regularity theory is the Calderon-Zygmund regularity theory where higher integrability of the gradient $\nabla u$ of the solution $u$ is sought in relation to higher integrability of ${h}$ and $g$. The now well-known $W^{1,p}$-theory proves that  for a possibly rough coefficient $A(x)$ but with small mean oscillation, for any $1<p<\infty$, if ${h}\in L^{p}_{loc}$ and $g$ is, say, smooth, then $\nabla u\in L^{p}_{loc}(\R^{n})$ \cite{IS98}. Another line of regularity focuses on the differentiability of $\nabla u$ and this is intimately related to the smoothness of the coefficient $A(x)$ in \eqref{classical-eqn}. In fact, the $W^{2,p}$-theory states that if $A$ is Lipschitz continuous, and $g\in L^p_{loc}(\R^{n})$, say ${h}$ is smooth, then the weak solution  $u$ of \eqref{classical-eqn} is twice differentiable and $D^{2}u\in L^{p}_{loc}$, \cite[Theorem 9.11]{GT01}\footnote{observe that the statement of \cite[Theorem 9.11]{GT01} is in non-divergence form with bounded order coefficients. To transform a divergence form equation to an nondivergence form equation with bounded coefficients, the original coefficients should be Lipschitz}.

The main objective of this paper is to prove regularity results of the above type for distributional solutions ${u}$ of nonlocal equations such as \eqref{Def-loc-sol-eqn}. Although the conditions we put are different,  the spirit of the results is similar in the sense that we are looking for higher differentiability in the fractional Sobolev scale and higher integrability of the solution $u$ as a function of data $f_1$ and $f_2$ in \eqref{Def-loc-sol-eqn}.  The following theorem states the main result of the paper.

\begin{theorem}\label{th:main}
Let  $s\in(0, 1)$ and $s \leq t<\min\{ 2s, 1\}$. 
% Then the following  holds true.  
If for $2\leq q<\infty$, $f_1,f_2 \in L^q(\Omega)\cap L^2(\R^n)$, and  $u\in  W^{s, 2}(\Omega)$  is  a distributional solution of
\[
\langle\mathcal{L}^{s}_{\Omega}u, \varphi\rangle=\int_{\R^{n}}  f_1\laps{2s-t} \varphi\, dx  + \int_{\R^{n}} f_2 \varphi\, dx \quad \forall \varphi \in C_c^\infty(\Omega_1), 
\]
for some $\Omega_1 \subseteq \Omega$ in  the sense of Definition \ref{Def-loc-sol} with $\mathcal{L}^{s}_{\Omega}$ corresponding to $K\in \mathcal{K}(\alpha, \lambda, \Lambda)$ for some given $\alpha\in (0, 1)$ and $\lambda, \Lambda>0$, then for any $W^{s,2}$-extension $\tilde{u}$ of $u$ to $\R^n$ we have $\laps{t} \tilde{u} \in L^q_{loc}(\Omega_1)$ and for any $\Omega' \subsubset \Omega_1$  we have 
\[
 \|\laps{t} u\|_{L^q(\Omega')} \leq C \brac{\|u\|_{W^{s,2}(\Omega)} + \sum_{i=1}^2 \|f_i\|_{L^q(\Omega_1)}+\|f_i\|_{L^2(\R^n)} }.
\]
The constant  $C$ depends only $s$,$t$,$q$,$\alpha$, $\lambda$, $\Lambda$,$\Omega$, and $\Omega'$.  
\end{theorem}
We used the notation $A \subsubset B$ when $A$ and $B$ are open and the closure of $A$  is compact and is a subset of $B$.

Let us highlight some corollaries of \Cref{th:main} that might appear in applications. For the proofs we refer to \Cref{s:corollaries}.
\begin{corollary}\label{co:main2}
Let  $s\in(0, 1)$ and $s\leq t <\min\{1, 2s\})$.  If for  $q\geq 2$,  $u\in  W^{s, 2}(\Omega)$  is  a distributional solution of 
 \[
 %\int_{\Omega}\int_{\Omega} \frac{K(x,y) (u(x)-u(y))\, (\varphi(x)-\varphi(y))}{|x-y|^{n+2s}}\, dx\, dy = 
 \langle \mathcal{L}^{s}_{\Omega} u, \varphi\rangle = 
 \int_{\Omega}f\, \varphi\,dx,  \quad \forall \varphi \in C_c^\infty(\Omega),
 \]
with $\mathcal{L}^{s}_{\Omega}$ corresponding to $K\in \mathcal{K}(\alpha, \lambda, \Lambda)$ for some given $\alpha\in (0, 1)$ and $\lambda, \Lambda>0$. Then for any $W^{s,2}$-extension $\tilde{u}$ of $u$ to $\R^n$,  $\laps{t} \tilde{u} \in L^q_{loc}(\Omega)$, and for any $\Omega' \subsubset \Omega$ we have 
\[
 \|\laps{t} \tilde{u}\|_{L^q(\Omega')} \leq C \brac{\|f\|_{L^q(\Omega)} + \|\tilde{u}\|_{W^{s,2}(\R^n)}}.
\]
\end{corollary}

\begin{corollary}\label{co:maindual}
Let  $s\in(0, 1)$ and $s \leq t<\min\{ 1, 2s\}$. For any open set $\Omega \subset \R^n$, $2 \leq q < \infty$ the following holds.

If $f \in (H^{2s-t,q'}(\Omega))^\ast$ and $u\in  W^{s, 2}(\Omega)$  is  a distributional solution of
\[
\langle\mathcal{L}^{s}_{\Omega}u, \varphi\rangle=\langle f, \varphi \rangle \quad \forall \varphi \in C_c^\infty(\Omega)
\]
in  the sense of Definition \ref{Def-loc-sol} with $\mathcal{L}^{s}_{\Omega}$ corresponding to $K\in \mathcal{K}(\alpha, \lambda, \Lambda)$ for some given $\alpha\in (0, 1)$ and $\lambda, \Lambda>0$. Then for any $W^{s,2}$-extension $\tilde{u}$ of $u$ to $\R^n$ we have $\laps{t} \tilde{u} \in L^q_{loc}(\Omega)$ and for any $\Omega' \subsubset \Omega$  we have 
\[
 \|\laps{t} u\|_{L^q(\Omega')} \leq C \brac{\|u\|_{W^{s,2}(\Omega)} + \|f\|_{(H^{2s-t,q'}(\Omega))^\ast}}
\]
The constant  $C$ depends only $s$,$t$,$q$,$\alpha$, $\lambda$, $\Lambda$,$\Omega$, and $\Omega'$.  
\end{corollary}
\begin{corollary}\label{co:main3}
Let  $s\in(0, 1)$ and $p\geq 2$. Let $\Omega \subsubset \R^n$ be a smoothly bounded set, and let $\Omega_1 \subsubset \Omega$ be open. Assume that $u \in W^{s,2}(\Omega)$ satisfies 
\begin{equation}\label{eq:Lsudivs}
 \langle \mathcal{L}^{s}_{\Omega} u, \varphi\rangle = 
 \int_{\Omega}\int_{\Omega} \frac{(f(x)-f(y))\, (\varphi(x)-\varphi(y))}{|x-y|^{n+2s}}\, dx\, dy 
 \end{equation}
 for any $\varphi\in C_c^{\infty}(\Omega_{1})$, where $\mathcal{L}^{s}_{\Omega}$ corresponds to $K\in \mathcal{K}(\alpha, \lambda, \Lambda)$ for some given $\alpha\in (0, 1)$ and $\lambda, \Lambda>0$. Then if for $s < t_0 < \min\{2s,1\}$, $f\in W^{t_0, p}(\Omega)$
then for any $s \leq t < t_0$, $u\in W^{t, p}_{loc}(\Omega)$, and for any $\Omega_1 \subset \Omega$ we have the estimate
\[
 [u]_{W^{t,p}(\Omega_1)} \leq C\, \brac{[f]_{W^{t_0,p}(\Omega)} + [u]_{W^{s,2}(\Omega)}}+\|u\|_{L^2(\Omega)}.
\]
\end{corollary}
We observe that \Cref{co:main2} is to some extent an analogue of the local $W^{2,p}$-theory for divergence-form equations such as \eqref{classical-eqn}. However, there is one major difference: while the higher fractional differentiability of solutions for local equations of the form \eqref{classical-eqn} is \emph{closely related} to the smoothness of the coefficient $A$,  for nonlocal equations of the form \eqref{Def-loc-sol-eqn} it is only \emph{loosely related} to the smoothness of the coefficient $K$. 

Namely, for local equations, if $\div(A \nabla u) \in L^p$ and $A \in C^\alpha$, then $u \in W^{s,q}_{loc}$ for any $s < 1+\alpha$ and $2-s -\frac{n}{p} > -\frac{n}{q}$.
That is, the increase in differentiability of the solution depends on the relative smoothness of the coefficient, the $\alpha$-H\"older continuity of $A$. In the case of solutions to the  nonlocal equation in Definition \ref{Def-loc-sol}, the increase on differentiability of $u$ is independent of the measure of H\"older continuity of the coefficient $K$. In other words, as long as $K$ is H\"older continuous of any order $\alpha\in (0, 1)$, the solution can be proved to be differentiable up to the order of $\min\{1, 2s\}$. 

This presents one of the distinctions of our work from that of the regularity result obtained in \cite{gale_ofa485955384} (which considers $L^2$-regularity). In \cite{gale_ofa485955384}, the almost optimal regularity of solution to \eqref{Def-loc-sol} corresponding to $f_1=0,$ and $f_2\in L^{2}_{loc}(\R^n)$ is obtained under the assumption that $K\in C^{s} (\R^{n}\times\R^{n})$. Using this smoothness assumption on $K$, which allows the application of the ``difference quotient'' method of proving higher differentiability,  in \cite{gale_ofa485955384} the solution $u$ is shown to belong to $H^{2s-\epsilon,2}_{loc}(\R^{n})$ for any $\epsilon>0$.  

For right-hand sides in $L^2$ we get similar \emph{differentiability} results to \cite{gale_ofa485955384}, but at most up to differential order $1$. However, we merely assume $K$ to be $C^\alpha$-H\"older continuous for some $\alpha > 0$ possibly much smaller than $s$, and $K$ only needs to be be positive on the diagonal. An example for a 
kernel that belongs to $\mathcal{K}(\alpha, \lambda, \Lambda)$ but does not fit the framework given in \cite{gale_ofa485955384} is $K(x, y) = \frac{2\lambda + |x|^{\alpha}+ |y|^{\alpha}}{\lambda + |x|^{\alpha}+ |y|^{\alpha}}+ 10^{6}{(\sin x + \sin y)}\frac{|x-y|^{\alpha}}{(1 + |x-y|^{\alpha})}$. Observe that for small $\lambda >0$, $K$ could be negative off the diagonal $\{x=y\}$. 

% In addition, Corollary \ref{co:main2} is applicable for higher integrable data as well with the expected result of the solution being in highly integrable fractional Sobolev space. Optimal 

Local elliptic regularity results for weak solutions to the Dirichlet problem associated
with the fractional Laplacian is also studied in \cite{Warma}. 

Let us also mention the recent work \cite{Nowak19}, where nonlocal equations of the type \eqref{Def-loc-sol-eqn} are studied for translation invariant coefficients, $K(x, y) = K(x-y)$. In this work, without imposing any smoothness assumption on $K(x-y)$,  and using a real-analytic perturbation argument pioneered in \cite{alma9925334574502311} and expanded in \cite{alma992559020102311} to obtain $W^{1,p}$-estimates, it was shown that if $f_1\in L^{p}_{loc}(\R^{n})$, and $f_2\in L^{\frac{pn}{n+sp}}_{loc}(\R^{n})$, then any weak solution $u$ to \eqref{Def-loc-sol-eqn} is in $H^{s,p}_{loc}(\R^{n})$. This result in \cite{Nowak19} concerns only the higher integrability of $\laps{s}u$, whereas, in  comparison, our work presents results on both higher differentiability and higher integrability of $\laps{s}u$ for solutions of nonlocal equations corresponding to coefficients that are not necessarily translation invariant. 

We should also mention that for ``strong solutions'' of nonlocal equations of the type $(\mathcal{L}^{s}_{\R^{n}} + \gamma \mathcal{I})u = f$ corresponding to translation invariant coefficients, $K(x, y) = K(x-y)$, and $\gamma >0$ the optimal regularity theory of $f\in L^{p}(\R^{n}) \implies u\in H^{2s, p}(\R^{n})$ is obtained in \cite{DONG20121166}. Similar to the previous paper discussed, the result in \cite{DONG20121166} requires no smoothness assumption on $K(x-y)$ and relies on a priori mean-oscillation estimates and maximal function theorem.  

Other types of improved regularity results have also been observed for weak solution of nonlocal equations of type \eqref{Def-loc-sol-eqn} with coefficients $K(x, y)$ that are just measurable, elliptic and bounded from above. What is called a self-improvement property of such solutions, which was first obtained in \cite{KMS15} via a generalized Gehring lemma, states that  for $f_1\in H^{s+\epsilon}$ and $f_2\in L^{\frac{2n}{n + 2s}}_{loc},$ a weak solution $u\in H^{s, 2}(\R^{n})$ is in fact in $W^{s+\delta, 2+\delta}_{loc}(\R^{n})$. While the improvement in integrability of the solutions is expected, the incremental improvement in differentiability without requiring any smoothness assumption on the coefficient $K$ is unique to nonlocal equations of this type.  
%\cite{S15,S16} (but they have an $\eps$-stability, we get it all the way down to $1$, $2s$.). 
%One can interpret the self improvement property as a certain stability in differential direction. 
Intuitively, one can see why such improvement can be possible. In fact, that for any $s_1,s_2 \in (0,1)$ with $s_1+s_2 = 2s$ we have
\[
\begin{split}
 %&\int_{\R^n} \int_{\R^n} \frac{K(x,y)\ (u(x)-u(y))\, (\varphi(x)-\varphi(y))}{|x-y|^{n+2s}}\, dx\, dy \\
\langle\mathcal{L}^{s}_{\R^{n}}u, \varphi \rangle \leq& \|K\|_{L^\infty} 
\brac{ \int_{\R^n} \int_{\R^n} \frac{|u(x)-u(y)|^p}{|x-y|^{n+s_1p}}\, dx\, dy}^{\frac{1}{p}}\, \brac{\int_{\R^n} \int_{\R^n} \frac{|\varphi(x)-\varphi(y)|^{p'}\, }{|x-y|^{n+s_2p'}}\, dx\, dy}^{\frac{1}{p'}}
\end{split}
\]
That is, there is a possibility that one can distribute derivatives freely on test functions or the solution. This is clearly false for the local case unless $s_1=s_2 = 1$.
\[
 \int_{\R^n} A(x) \partial_\alpha u\, \partial_\alpha \varphi \not \aleq \|A\|_{L^\infty}\, \|u\|_{\dot{H}^{s_1,p}}\, \|\varphi\|_{\dot{H}^{s_2,p'}}.
\]
The self-improving property of nonlocal equations have also been demonstrated via other approaches: via functional analytic approach in \cite{euclideuclid.tunis/1543854680} and via comparison and commutator estimates in \cite{S16}. This kind of $\delta$-differential flexibility of nonlocal equations has also been observed and crucially used in the regularity theory of geometric equations \cite{S15,BRS2}.

Finally, we comment on our strategy of proving \Cref{th:main}. Our argument relies  on comparing the leading order operator in \eqref{eq:nonloc}, $\mathcal{L}^{s}_{\R^{n}}$ with that of the simpler operator $ L^{s,t}_{diag}$ defined as
\begin{equation}\label{defn-Ldiag}
\langle L^{s,t}_{diag}u, \varphi\rangle :=\int_{\R^n} K(z,z) \laps{t} u(z)\, \laps{2s-t}\varphi(z)\, dz, 
\end{equation}
for all $\varphi\in C_c^{\infty}(\R^{n})$ and $s\leq t < 1.$ To facilitate comparison of the operators, let us define the difference function 
\[
 D_{s,t}(u,\varphi) :=  \langle \mathcal{L}^{s}_{\R^{n}}u, \varphi\rangle - \Gamma \langle L^{s,t}_{diag}u, \varphi\rangle.
\]
Here $\Gamma$ is the constant (depending on $s$, $t$, and $n$) such that $D_{s,t}(u,\varphi) \equiv 0$ for all $u$ and $\varphi$ admissible whenever the coefficient $K$ is a constant map. In this sense, $D_{s,t}(u,\varphi)$ can be seen as a commutator $\int [T,K]u\, \varphi$ which is the main intuition in what follows. Indeed we obtain in \Cref{th:reformulationcommie} a quantitative estimate for $D_{s,t}$ that shows that in the case of H\"older continuous $K$, the commutator is \emph{of lower order}. Intuitively, the operator  $D_{s,t}(u,\varphi)$ gives us the mechanism to 'transfer derivatives' to $K$ which along the way reduces the number of derivatives on $u$ and $\varphi$. The commutator estimate we state in \Cref{th:reformulationcommie} is similar in spirit to the Coifman-Rochberg-Weiss commutator $[T,K](f)$ where $T$ is a Calderon-Zygmund operator. If $K$ is H\"older continuous of order $\gamma$,   then $[T,K](f)$ can be estimated by a Riesz potential $\lapms{\gamma} f$ of $f$ (i.e. a fractional antiderivative) -- this is exactly what we obtain for our commutator $D_{s,t}$ in \Cref{th:reformulationcommie}. 
While such a quantitative estimate is almost obvious for the Coifman-Rochberg-Weiss commutator it is already involved for our situation. Observe, however that a consequence of the famous work \cite{CRW} Coifman, Rochberg, Weiss is that the operator $f\mapsto [T,K](f)$ is a compact operator for $K$ in $VMO$, \cite{Uchiyama}. This suggests that with some work there could be a version of our theorem for $K$ in $VMO$ (in a suitable sense yet to be defined). We will investigate this direction in a future work.

Once we identify $D_{s,t}(u,\varphi)$  as a lower-order operator, we can essentially read of the regularity theory for the operator in \Cref{th:main} from the regularity theory of equations of the type
\begin{equation}\label{diagonal-coef-eqn}
\langle L^{s,t}_{diag}u, \varphi\rangle= \int g \varphi \quad \forall \varphi \in C_c^\infty(\R^n),
\end{equation}
which is relatively easy to handle. Notice that \eqref{diagonal-coef-eqn} is a distributional formulation of the elliptic equation $\laps{2s-t}(K(z,z) \laps{t} u) = g$. Thus, formally, \eqref{diagonal-coef-eqn} is equivalent to
\[
 \laps{t} u(x) = \frac{1}{K(x,x)}\lapms{2s-t} g(x),
\]
and thus one expects the estimate 
\[
 \|\laps{t} u\|_{L^q(\R^n)} \leq \frac{1}{\inf_{x} K(x,x)}\, \|\lapms{2s-t} g\|_{L^q(\R^n)}.
\]
In particular, if $g \in L^q(\R^n)$, then by Sobolev embedding $\lapms{2s-t} g \in L^{\frac{nq}{n-(2s-t)q}}(\R^n)$; that is, if $u$ solves \eqref{diagonal-coef-eqn} and $g \in L^q(\R^n)$, then $u \in H^{s,\frac{nq}{n-(2s-t)q}}_{loc}(\R^n)$. 

The precise argument is based on a duality argument and a bit tedious, but in the end we obtain the following result in Section~\ref{s:fraclap}.
\begin{theorem}\label{th:regularityKlapls:1}
Let $s>0$ and $t \in (0,2s)$. Assume that for some $q \in (1,\infty)$, $\laps{t} u \in L^q(\R^n)$ is a distributional solution to
\[
 \int_{\R^n} \bar{K}(z) \laps{t} u\, \laps{2s-t} \varphi = \int_{\R^n} f_1\, \laps{2s-t} \varphi + \int_{\R^n} f_2\, \varphi\quad \forall \varphi \in C_c^\infty(\Omega).
\]
Here $\bar{K}: \R^n \to \R$ is a positive, measurable, and bounded from above and below, i.e.
\[
\Lambda^{-1} \leq \bar{K}(z) \leq  \Lambda \quad \text{a.e. }x \in \R^n.
\]
Then for any $\Omega' \subsubset \Omega \subsubset \R^n$ if $f_1,f_2 \in  L^q(\R^n) \cap L^p(\Omega)$  then $\laps{t} u \in L^p(\Omega')$ with
\[
 \|\laps{t} u\|_{L^p(\Omega')} \aleq \|f_1\|_{L^p(\Omega)} + \|f_2\|_{L^p(\Omega)}+ \|f_1\|_{L^q(\R^n)} + \|\laps{t} u\|_{L^q(\R^n)}.
\]
\end{theorem}
%\begin{remark}
It might seem surprising at first that in \Cref{th:regularityKlapls:1} there is no assumption on the kernel being continuous or belonging to $VMO$ -- and still we are able to obtain $L^p$-estimates for any $p > 2$ if the right-hand side of the equation is good enough.
For classical divergence form equations, 
\begin{equation}\label{eq:divknu}
 \div(\bar{K} \nabla u) = f
\end{equation}
if $\bar{K}$ is only bounded measurable, the best one can hope for is an $W^{1,2+\eps}$-type estimate (if $f$ is nice enough) -- this is known as a Meyers-type estimate, \cite{M63,EM75}. The reason that we get a (seemingly) better result in \Cref{th:regularityKlapls:1} is not because of the fractional order, but rather of the fact that $\nabla$ and $\div$ are non-elliptic operators, while $\laps{2s-t}$ is invertible. 

% Indeed, a formal argument for \Cref{th:regularityKlapls:1} is
% \[
%  \laps{s}(\bar{K} \laps{s} u) = f \quad \Leftrightarrow \quad \bar{K} \laps{s} u = \lapms{s} f \quad \Leftrightarrow \quad \laps{s} u = \frac{1}{\bar{K}} \lapms{s} f. 
% \]
% If $f \in L^p(\R^n)$ and $\bar{K}$ is bounded from below, then by Sobolev embedding, $\frac{1}{\bar{K}} \lapms{s} f \in L^{\frac{np}{n-sp}}_{loc}$. That is, if $u$ solves $\laps{s}(\bar{K} \laps{s} u) = f$ and $f \in L^p(\R^n)$, then $u \in H^{s,\frac{np}{n-sp}}_{loc}(\R^n)$. 
An argument such as the one described before \Cref{th:regularityKlapls:1} does not work for solutions to \eqref{eq:divknu}, because we cannot invert the $\div$-operator (and indeed for merely bounded measurable kernels only Meyer's $2+\eps$-estimate remains true). So in \Cref{th:regularityKlapls:1} we make crucial use of the fact that the equation involved is structurally substantially different from (and for our purposes: simpler than) \eqref{eq:divknu} -- even if $s = 1$. 

Let us mention in passing, that a more proper `nonlocal analogue' of the equation \eqref{eq:divknu} (in the sense that it has generally comparable regularity properties as \eqref{eq:divknu}) is 
\begin{equation}\label{eq:Rieszfracgradient}
  D^{2s-t}[\bar{K}D^{t} u] = f
\end{equation}
where $D^t$ denotes the Riesz-fractional gradient $\Rz \laps{t} \equiv \nabla \lapms{1-t}$. Indeed, if $\bar{K}$ is merely bounded, measurable then for solutions to \eqref{eq:Rieszfracgradient} only Meyer-type estimates are known, \cite[Section 9]{ABES17}; and one needs $\bar{K}$ in VMO to conclude $L^p$-estimates, \cite{SSS15}. See also \cite{SS15,SS18}.  
%\end{remark}

% 
% Another question we do not address here, but which seems to be even open for the $K \equiv const$ case is boundary Calderon-Zygmund theory.

Let us remark on previous arguments that inspired this work: for regularity theory via an harmonic analysis approach in the local case with an elliptic matrix $A_{\alpha,\beta}$ instead of the scalar $A$ see \cite{IS98}. This was applied to nonlocal equations different from \eqref{eq:nonloc} in \cite{SSS15}. Commutator operator similar to $D_{s,t}$ have also been proved to be very useful in the harmonic analysis of harmonic-type maps between manifolds \cite{S15} and nonlocal equations arising in topological calculus of variations, \cite{BRS2}.

The remainder of this work is as follows: in \Cref{s:commutator} we prove the commutator estimate for $D_{s,t}$. This essentially reduces the desired Calderon-Zygmund theory to that of the theory of a weighted fractional Laplacian which we treat in \Cref{s:fraclap} where the proof of \Cref{th:regularityKlapls:1}. Since we only obtain local estimates, we will repeatedly employ cutoff arguments that are obtained in \Cref{s:localglobal}. In \Cref{s:proofmain}, the proof of the main result \Cref{th:main} is presented. And finally, the corollaries of \Cref{th:main} are proved Section~\ref{s:corollaries}.

\textbf{Acknowledgments.}
The authors acknowledge funding as follows
\begin{itemize}
\item TM: National Science Foundation (NSF), grant no 1910180. 
\item AS: Simons foundation, grant no 579261.
\item SY: Science Achievement Scholarship of Thailand (SAST)
\end{itemize}
The research that lead to this work was partially carried out while AS was visiting Chulalongkorn University, and while SY was visiting the University of Pittsburgh.

\section{Preliminaries and notation}
Some notation and convention we will use throughout the paper. Domains of integrals are always open sets. We use the symbol $\subsubset$ to say compactly contained, e.g. $\Omega_1 \subsubset \Omega_2$ if $\overline{\Omega_1}$ is compact and $\overline{\Omega_1} \subset \Omega_2$.

Constants change from line to line, and generally depend on the dimension. We will make frequent use of $\aleq$, $\ageq$ and $\aeq$, which denotes inequalities with multiplicative constants (depending on non-essential data). For example we say $A \aleq B$ if for some constant $C > 0$ we have $A \aleq C B$.

We work with fractional Laplacians, Sobolev spaces, and related operators. Below we introduce the notation but refer the interested reader to surveys, e.g. \cite{DNPV12,G19}, or monographs \cite{S02}. We will use many techniques from harmonic analysis, such as Sobolev inequalities, embeddings etc. -- these are all well-known in the abstract framework of Triebel-Lizorkin or Besov-space theory -- see e.g. in \cite{GMF}. Generally we like to refer to \cite{RS96} for the identification of Triebel-Lizorkin and Besov-spaces with the ``usual'' function spaces. While we try to make as little as possible use of such abstract arguments sometimes they are unavoidable. 

%Let $\mathcal{F}$ denote the Fourier transform. 
For $s > 0$ we have defined the fractional Laplacian $\laps{s}$ in the introduction as the operator that, for $f$ in the Schwartz class, acts as multiplier with symbol $c|\xi|^s$ (cf. \eqref{eq:fraclapFT})
\[% \begin{equation}\label{eq:fraclapFT}
 \mathcal{F}(\laps{s} f)(\xi) = c\, |\xi|^s \mathcal{F} f(\xi).
\]
%The constant $c$ depends on $n$ and $s$ and plays no deeper role in the theory that we consider.
The Riesz potential $\lapms{s} = (-\lap)^{-\frac{s}{2}}$ is the inverse of the fractional Laplacian, i.e. the multiplier operator with symbol $(c|\xi|^s)^{-1}$,
\[
 \mathcal{F}(\lapms{s} f)(\xi) := \frac{1}{c}\, |\xi|^{-s} \mathcal{F} f(\xi).
\]
This operator makes sense (for $f$ a function in the Schwartz class) if $0 \leq s < n$, because $|\xi|^{-s}$ is then locally integrable.

For $s \in (0,2)$ the fractional Laplacian $\laps{s}$ has a useful integral representation. Namely, for a function $f$ in the Schwartz class
\[
 \laps{s} f(x) = c \int_{\R^n} \frac{f(x)-f(y)}{|x-y|^{n+s}}\, dy, 
\]
where the integral is defined in the \emph{principal value} sense, although we do not explicitly state it. 
%For $s \in (0,2)$, a popular integral representation is
%\begin{equation}\label{eq:fraclapintegro2}
% \laps{s} f(x) = \frac{c}{2} \int_{\R^n} \frac{f(x+h)+f(x-h)-2f(x)}{|x-y|^{n+s}}\, dy.
%\end{equation}
%For higher orders see \cite{S02}, although we will not need this here (or ever).
For the Riesz potential, for $s \in (0,n)$, we have the representation
\[
 (-\lap)^{-\frac{s}{2}} f(x) \equiv \lapms{s} f(x) = c\int_{\R^n} |x-y|^{s-n} f(y)\, dy
\]
for a function $f$ in the Schwartz class. 

For functions $f$ and $g$ in the Schwartz class, the $L^2$-inner product of $ \laps{s} f(x)$ and  $g(x)$  can be represented as, for $s \in (0,2)$,
 \begin{equation}\label{eq:fraclap}
 \int_{\R^n} \laps{s} f(x) \, g(x) dx = \int_{\R^n} \int_{\R^n} \frac{(f(y)-f(x))(g(y)-g(x))}{|x-y|^{n+s}}\, dx\, dy
 \end{equation}
 see e.g. \cite[Proposition 2.36.]{ArminPhD} or \cite{DNPV12}.
 
 Fractional Laplacians and gradients are related via Riesz transforms and Riesz potentials. The Riesz transform, $\Rz = (\Rz_1,\ldots,\Rz_n) := \nabla \lapms{1}$, has the Fourier symbol $c i \frac{\xi}{|\xi|}$, and a potential representation
 \[
  \Rz f(x) = \int_{\R^n} \frac{\frac{x-y}{|x-y|} }{|x-y|^{n}}\, f(y)\, dy.
 \]
Riesz transforms are most prominent examples of Calderon-Zygmund operators and are $L^p$-bounded. That is, for $1<p<\infty$, there exists a constant $C=C(n,p)>0$ such that 
\[
\|  \Rz f\|_{L^{p}} \leq C\|f\|_{L^{p}}, \quad\text{for all $f\in L^{p}$}. 
\]
%\itad{Remove:}
The now classical $L^p$-regularity theory for linear second-order PDEs is called Calderon-Zygmund theory because it (secretly or explicitly) relies on estimates of Calderon-Zygmund-operators (in most of the cases: the Riesz transforms).
%We will work with both Besov and Bessel potential spaces. The Bessel potential spaces $H^{s,p}$ are characterized as follows: $f \in H^{s,p}(\R^n)$ if $f \in L^p(\R^n)$ and $\laps{s} f \in L^p(\R^n)$, and the associated norm is
%\[
% \|f\|_{H^{s,p}(\R^n)} := \|f\|_{L^p(\R^n)} + \|\laps{s} f\|_{L^p(\R^n)}.
%\]
%%A word of warning, these are different (and notation varies) to the 
%Whereas the Besov spaces $W^{s,p}$, which for $s \in (0,1)$ are induced by the semi-norm
%\[
% [f]_{W^{s,p}(\R^n)} = \brac{\int_{\R^n} \int_{\R^n} \frac{|f(x)-f(y)|^{p}}{|x-y|^{n+sp}}\, dx\, dy}^{\frac{1}{p}}.
%\]
%For $p=2$, $W^{s,2} = H^{s,2}$, for $p < 2$ we have $W^{s,p} \subsetneq H^{s,p}$ and for $p>2$ we have $H^{s,p} \subsetneq W^{s,p}$. These spaces are particular examples of the more general Triebel-Lizorkin spaces and $F^{s}_{pp} = W^{s,p}$ and $F^s_{p,2} = H^{s,p}$, see \cite{RS96}. However, the spaces are not too far away from each other. If $f \in W^{s,p}$ then $f \in H^{s-\eps,p}_{loc}$ for any $\eps > 0$, and if $f \in H^{s,p}$ then $f \in W^{s-\eps,p}_{loc}$ for any $\eps > 0$. 
%
%such as ts in \Cref{th:main} we also obtain estimates in $W^{s,p}$-space by embedding.

We will frequently use Sobolev inequalities for Riesz potential. 
%\subsection{Some standard theorems}

\begin{proposition}[Sobolev inequalities]\label{pr:sob}
Suppose that $s \in (0,n)$ and $p \in (1,\infty)$. Then, 
\begin{enumerate}
\item[(a)] if $sp < n$, then there exists a constant $C=C(s,p,n)>0$ such that 
\begin{equation}\label{eq:sob:glob}
 \|\lapms{s} g\|_{L^{\frac{np}{n-sp}}(\R^n)} \leq C\, \|g\|_{L^p(\R^n)}\quad \text{for any $g \in L^p(\R^n)$}.
\end{equation}
In addition, if $\Omega\subset\R^{n}$ is bounded, then corresponding to any $q \in [1,\frac{np}{n-sp}]$, there is a constant $C=C(s,p,n, \Omega)>0$ such that 
\begin{equation}\label{eq:sob:loc1}
 \|\lapms{s} g\|_{L^q(\Omega)} \leq C\, \|g\|_{L^p(\R^n)} \quad \text{for any $g \in L^p(\R^n)$}.
\end{equation}
\item [(b)]If $sp \geq n$ and $\Omega\subset\R^{n}$ is bounded domain,  then for any $q \in [1,\infty)$, and $r \in [1,\frac{n}{s})$, there exists a constant $C=C(s,p,n, \Omega)>0$ such that 
\begin{equation}\label{eq:sob:loc2}
 \|\lapms{s} g\|_{L^q(\Omega)} \leq C\, \brac{\|g\|_{L^p(\R^n)}+\|g\|_{L^r(\R^n)}}.
\end{equation}
\end{enumerate}
\end{proposition}
\begin{proof}
The proof of \eqref{eq:sob:glob} can be found in \cite{alma996965690102311}. \eqref{eq:sob:loc1} follows easily from \eqref{eq:sob:glob}.  As for \eqref{eq:sob:loc2}, observe that for any $q \in (1,\infty)$ there exists some $\theta \in (r,\frac{n}{s})$ such that $\frac{\theta n}{n-s\theta} > q$. Observe that $\theta < p$ so that we have the interpolation inequality
\[
\|g\|_{L^\theta(\R^n)} \aleq \|g\|_{L^p(\R^n)} + \|g\|_{L^r(\R^n)}
\]
By \eqref{eq:sob:loc1} we have
\[
 \|\lapms{s} g\|_{L^q(\Omega)} \aleq \|g\|_{L^\theta(\R^n)} \aleq {\|g\|_{L^p(\R^n)}+\|g\|_{L^r(\R^n)}}.
\]
\end{proof}
We also need the following characterization of the dual space of the function spaces $H^{s,p}(\R^n)$ and $\dot{H}^{s,p}(\R^n).$ The homogeneous space $\dot{H}^{s,p}(\R^n)$ is the set of $u$ a tempered distributions such that $(-\Delta)^{s/2}u\in L^{p}(\R^{n})$, with the semi-norm $\| (-\Delta)^{s/2}u\|_{L^{p}}$.      

By definition,  $T \in \brac{H^{s,p}(\R^n)}^\ast$, the dual space of $H^{s,p}(\R^n)$, if $T$ is linear on $\varphi \in C_c^\infty(\R^n)$ and
\[
 |T[\varphi]|\leq \Lambda \brac{\|\varphi\|_{L^p(\R^n)} + \|\laps{s} \varphi\|_{L^p(\R^n)}} \quad \forall \varphi \in C_c^\infty(\R^n).
\]
The operator norm of $T$, $\|T\|$, is defined to be the infimum of all such $\Lambda$. 
Similarly, $T \in \brac{\dot{H}^{s,p}(\R^n)}^\ast$, if $T$ is linear on $\varphi \in C_c^\infty(\R^n)$ and
\[
 |T[\varphi]|\leq \Lambda \|\laps{s} \varphi\|_{L^p(\R^n)} \quad \forall \varphi \in C_c^\infty(\R^n).
\]

\begin{proposition}(Dual Spaces)\label{pr:dualclass}
\begin{enumerate}
\item If $T \in \brac{H^{s,p}(\R^n)}^\ast$, then there exists $g_1,g_2 \in L^{p'}(\R^n)$, 
\[
\|g_1\|_{L^{p'}(\R^n)} + \|g_2\|_{L^{p'}(\R^n)} \aeq \|T\|
\]
such that 
\[
 T[\varphi] = \int g_1 \laps{s} \varphi + \int g_2 \varphi \quad \forall \varphi \in C_c^\infty(\R^n).
\]
\item  If $T \in \brac{\dot{H}^{s,p}(\R^n)}^\ast$, then there exists $g \in L^{p'}(\R^n)$, 
\[
\|g\|_{L^{p'}(\R^n)} \aeq \|T\|
\]
such that 
\[
 T[\varphi] = \int g \laps{s} \varphi \quad \forall \varphi \in C_c^\infty(\R^n).
\]
\end{enumerate}
\end{proposition}
\begin{proof}
Let $T \in \brac{{H}^{s,p}(\R^n)}^\ast$. Denoting $\langle{\xi} \rangle := \sqrt{1+|\xi|^2}$, using the equivalence of norms we have
\[
|T(\varphi)|\leq \|T\|\|\mathcal{F}^{-1}( \langle{\xi}\rangle^{s}\mathcal{F}(\varphi))\|_{L^{p}},\quad \text{for all $\varphi \in C_c^\infty(\R^n)$.}
\]
We then introduce the linear function $\tilde{T}: L^{p}(\R^{n})\to \R$ defined by 
\[
\tilde{T}(v) = T(\mathcal{F}^{-1}( \langle{\xi}\rangle^{-s}\mathcal{F}(v))).
\]
Then from the estimate for $T$, we have that $|\tilde{T}(v)| \leq \|T\|\|v\|_{L^{p}}$ for all $v\in L^{p}(\R^{n})$. By the characterization of the dual of $L^p$ spaces we have  $u_0\in L^{p'}(\R^{n})$ such that 
\[
\tilde{T}(v) = \int_{\R^{n}}u_0(x)v(x)dx,\quad \text{for all $v\in L^{p}(\R^{n})$.}\]
Define now $g = \mathcal{F}^{-1}( \langle{\xi}\rangle^{s}\mathcal{F}(u_0))$. Then $g\in H^{-s, p'}(\R^n)$ and  for any $\varphi\in \mathcal{S}$, the Schwartz space, we have by applying Plancherel's theorem repeatedly that 
\[
\begin{split}
\langle g, \varphi\rangle = \langle  \mathcal{F}^{-1}( \langle{\xi}\rangle^{s}\mathcal{F}(u_0)), \mathcal{F}^{1}(\mathcal{F} \varphi)\rangle &=  \langle  \langle{\xi}\rangle^{s}\mathcal{F}(u_0), \mathcal{F} \varphi\rangle\\
& =  \langle  \mathcal{F}(u_0), \langle{\xi}\rangle^{s}(\mathcal{F} \varphi)\rangle\\
 &=  \langle  u_0, \mathcal{F}^{-1}(\langle{\xi}\rangle^{s}(\mathcal{F} \varphi))\rangle = \tilde{T}(\mathcal{F}^{-1}(\langle{\xi}\rangle^{s}(\mathcal{F} \varphi))) = T(\varphi)
\end{split}
\]
We next characterize $g$ further. Using \cite[Lemma 2 of Chapter 3]{alma996965690102311}, that describes the relationship between Riesz and Bessel potentials, there exists a pair of finite measures $\nu_s$ and $\lambda_s$ so that 
\[
\begin{split}
g= \mathcal{F}^{-1}( \langle{\xi}\rangle^{s}\mathcal{F}(u_0)) 
&=  \nu_s\ast u_0 +  \mathcal{F}^{-1}( |\xi|^{s}\mathcal{F}(\lambda_s \ast u_0))\\
\end{split}
\]
Define $g_1 = \nu_s\ast u_0 $ and $g_2 = \lambda_s\ast u_0 $. Then both $g_1$ and $g_2$ are in $L^{p'}(\R^{n})$. Moreover, by applying Plancherel's theorem again
\[
T(\varphi) = \langle g, \varphi\rangle = \langle g_1, \varphi\rangle +\langle  \mathcal{F}^{-1}( |\xi|^{s}\mathcal{F}(g_2)), \varphi\rangle =  \langle g_1, \varphi\rangle +\langle g_2, \mathcal{F}^{-1}( |\xi|^{s}\mathcal{F}(\varphi))\rangle
\]
as desired.

As for the second part, observe that since $\dot{H}^{s,q}(\R^n) \aeq F^{s}_{q,2}(\R^n)$ we have that $(\dot{H}^{s,q}(\R^n))^\ast \aeq F^{-s}_{q',2}(\R^n)$ (\cite[Remark 5.14]{FJ90}). Since $\lapms{s}$ is an isomorphism from $F^{-s}_{q',2}(\R^n)$ to $F^{0}_{q',2} \aeq L^{q'}$, see \cite[\textsection 2.6, Proposition 2, p. 95]{RS96}, we find that for any $(\dot{H}^{s,q}(\R^n))^\ast$ there must be $g \in L^{q'}(\R^n)$ with $\|g\|_{L^{q'}(\R^n)} \aeq \|T\|_{(\dot{H}^{s,q}(\R^n))^\ast}$ such that $\laps{s} g[\varphi] = T[\varphi]$, that is
$
 T[\varphi] = \int_{\R^n} \laps{s} g\, \varphi.
$
\end{proof}

Let us mention also two technical results, that we will employ frequently. They fall under the notion of ``cutoff argument'', and the techniques are mainly based on estimating nonlocal quantities for functions with disjoint support.
\begin{lemma}\label{la:cutoffarg}
Let $\Omega_1 \subsubset \Omega_2 \subsubset \R^n$, and $u, v \in H^{s,2}(\R^n)$ with $u \equiv v$ in $\Omega_2$, $s \in [0,1)$.

Then for any $p \in (1,\infty)$ we have
\[
 \|\laps{s} u\|_{L^p(\Omega_1)} \aleq \|\laps{s} v\|_{L^p(\Omega_2)} + \|u\|_{L^p(\R^n)} + \|v\|_{L^p(\R^n)}.
\]
\end{lemma}
\begin{proof}
Let $\eta \in C_c^\infty(\Omega_2)$ with $\eta \equiv 1$ in a neighborhood of $\Omega_1$.

We have 
\[
 u = \eta v + (1-\eta) u.
\]
Then
\[
 \chi_{\Omega_1} \laps{s} u = \chi_{\Omega_1} \laps{s} (\eta v) + \chi_{\Omega_1} \brac{\laps{s} (1-\eta) u},
\]
and by the usual disjoint support argument
\[
 \left |\chi_{\Omega_1} \brac{\laps{s} (1-\eta) u}\right | (x) \aleq \int_{|x-y|\ageq 1} |x-y|^{-n-s} |u(y)|\, dy
\]
By Young's inequality for convolutions we conclude
\[
 \|\chi_{\Omega_1} \brac{\laps{s} (1-\eta) u}\|_{L^\infty} \aleq \|u\|_{L^p(\R^n)}.
\]
And thus in particular,
\[
  \|\chi_{\Omega_1} \brac{\laps{s} (1-\eta) u}\|_{L^p} \aleq \|u\|_{L^p(\R^n)}.
\]
Now we use commutator notation $[T,m](f) = T(mf)-mTf$,
\[
 \laps{s} (\eta v) = \eta \laps{s} v + [\laps{s},\eta](v).
\]
Since $s \in (0,1)$ we can use the Coifman–McIntosh–Meyer commutator estimate, e.g. in the formulation in \cite[Theorem 6.1.]{LS18} or the Leibniz rule, \cite[Theorem 7.1.]{LS18}, and conclude that 
\[
 \|[\laps{s},\eta](v)\|_{L^p(\Omega_1)} \aleq \|\eta \|_{\lip}\, \|v\|_{L^p(\R^n)}.
\]
This concludes the proof.
\end{proof}

\begin{proposition}\label{prop:disjoint-support}
Suppose that $\eta_1, \eta_2, \psi\in C_c^{\infty}(\R^{n})$, and that $\eta_1 = 1$ in a neighborhood of the support of $\psi$ and $\eta_1 =1$ in the neighborhood of the support of $\eta_2$. 
Suppose that $p\in (1, \infty)$, $\tau \in (0,2)$ and 
 \[\text{$r > \frac{n p}{n+\tau p}$ if $\tau \leq 1$\,\, and  $r > \frac{n p}{n+p}$ if $\tau \geq 1$} \]
 Then we have the following estimates:
 \begin{enumerate}
 \item[(a)] There exists a constant $C>0$ such that 
 \begin{equation}\label{eq:alskd2}
 \|(1-\eta_2)\laps{\tau} \brac{(1-\eta_1) \lapms{\tau} \psi }\|_{L^{r'}(\R^{n})} \leq C \|\psi\|_{L^{p'}(\R^n)}.
 \end{equation}
 \item[(b)] For any bounded set $\Sigma \subsubset \R^n$, there exists a constant $C=C(\Sigma)$ such that 
\begin{equation}\label{eq:alskd1}
 \|\laps{\tau} \brac{(1-\eta_1) \lapms{\tau} \psi }\|_{L^{r'}(\Sigma)} \leq C \|\psi\|_{L^{p'}(\R^n)}.
\end{equation}
\end{enumerate}
\end{proposition}
\begin{proof}
We prove part $(b)$  first. Fix a  large ball $B \subsubset \R^n$ that compactly  contains  $\Sigma$. Then it follows from Leibniz's rule for fractional Laplacian that for $x\in \R^{n}$, 
\[
\begin{split}
 &\laps{\tau} \brac{(1-\eta_1) \lapms{\tau} \psi}(x)\\
 =& [\laps{\tau} (1-\eta_1)]\, \lapms{\tau} \psi(x) + \underbrace{(1-\eta_1)\, \psi(x)}_{=0} + c\int_{\R^n} \frac{(\eta_1(x)-\eta_1(y))(\lapms{\tau}\psi(x)-\lapms{\tau}\psi(y))}{|x-y|^{n+\tau}}\, dy
  \end{split}
\]
since the support of $1-\eta_1$ and $\psi$ do not intersect.  
The right hand side can be rewritten as 
\[
\begin{split}
&\laps{\tau} \brac{(1-\eta_1) \lapms{\tau} \psi}(x)\\
 =& -[\laps{\tau} \eta_1\, ]\lapms{\tau} \psi(x) + c\int_{B} \frac{(\eta_1(x)-\eta_1(y))(\lapms{\tau}\psi(x)-\lapms{\tau}\psi(y))}{|x-y|^{n+\tau}}\, dy\\
 & + c\int_{\R^n \backslash B} \frac{(\eta_1(x)-\eta_1(y))(\lapms{\tau}\psi(x)-\lapms{\tau}\psi(y))}{|x-y|^{n+\tau}}\, dy. 
 \end{split}
\]
We will estimate each term in the right hand side. We begin with the first one. To that end, since $\eta_1 \in C_c^\infty(\R^n)$, $\laps{\tau}\eta_1 \in L^\infty(\R^n)$. Moreover, in view of \Cref{pr:sob}, for any $r' < \frac{np'}{n-\tau p'}$
\begin{equation}\label{eq:742}
 \|\lapms{\tau} \psi\|_{L^{r'}(\Sigma)} \aleq \|\psi\|_{L^{p'}(\R^n)} \aleq \|\psi\|_{L^1(\R^n)} + \|\psi\|_{L^{p'}(\R^n)} \aleq \|\psi\|_{L^{p'}(\R^n)}, 
\end{equation}
where  in the last step, we have used $\psi$ has compact support. That is,
\[
 \|\laps{\tau} \eta_1\, \lapms{\tau} \psi\|_{L^{r'}(\Sigma)} \aleq \|\psi\|_{L^{p'}(\R^n)}.
\]
For the second term, we use, see e.g. \cite[Proposition 6.6.]{S18Arma}, that for any $\alpha < 1$,
\[
 |u(x)-u(y)| \aleq |x-y|^{\alpha}\, \brac{\mathcal{M} \laps{\alpha} u(x)+\mathcal{M} \laps{\alpha} u(y)},
\]
where $\mathcal{M}$ denotes the Hardy-Littlewood maximal function. Then, applying this inequality for $u=\lapms{\tau}\psi$, by the Lipschitz continuity of $\eta_1$, for any $x \in \Sigma$ (observe that also $B$ is bounded) for any $\alpha \in (0,\tau)$ with $1+\alpha-\tau >0$, we get 
\[
\begin{split}
 \int_{B} &\frac{\big (\eta_1(x)-\eta_1(y) \big )(\lapms{\tau}\psi(x)-\lapms{\tau}\psi(y))}{|x-y|^{n+\tau}}\, dy\\
& \aleq C \int_{B} |x-y|^{1+\alpha-\tau-n}\, \brac{\mathcal{M} \lapms{\tau-\alpha} \psi(x)+\mathcal{M} \lapms{\tau-\alpha} \psi(y)}\, dy\\
 &\aleq C\, \brac{\mathcal{M} \lapms{\tau-\alpha} \psi(x) + \lapms{1+\alpha-\tau}\brac{\mathcal{M} \lapms{\tau-\alpha} \psi}(x)},
\end{split}
\]
where $C$ depends only on $\|\eta_1\|_{\lip},\alpha,2s,t,\diam(B),$ and $\diam(\Sigma)$. 
If $\tau \leq 1$, $1+\alpha-\tau >0$ is equivalent to $\alpha > 0$. In that case, whenever $r' < \frac{np'}{n-\tau p'}$ we can choose $\alpha$ above so that $r' < \frac{np'}{n-(\tau-\alpha)p'}$.
If $\tau \geq 1$, we need to choose $\alpha>\tau -1 >0$, so whenever $r' < \frac{np'}{n-p'}$ we can find an $\alpha$ satisfying this condition so that $r' < \frac{np'}{n-(\tau-\alpha)p'}$. Then by Sobolev embedding, maximal theorem
\[
\left \|x \mapsto \int_{B} \frac{\big (\eta_1(x)-\eta_1(y) \big )(\lapms{\tau}\psi(x)-\lapms{\tau}\psi(y))}{|x-y|^{n+\tau-1}}\, dy \right \|_{L^{r'}(\Sigma)} \aleq \|\psi\|_{L^{p'}(\R^n)}.
\]
It remains to estimate for $x \in \Sigma$,
\[
\begin{split}
\left |\int_{\R^n \backslash B} \frac{\big (\eta_1(x)-\eta_1(y) \big )(\lapms{\tau}\psi(x)-\lapms{\tau}\psi(y))}{|x-y|^{n+\tau}}\, dy\right|
&\aleq\int_{\R^n \backslash B} \frac{\lapms{\tau}|\psi|(x)+|\lapms{\tau}|\psi|(y)}{1+|y|^{n+\tau}}\, dy\\
&\aleq\lapms{\tau} |\psi|(x) + \|\lapms{\tau}|\psi|\|_{L^\infty(\R^n \backslash B)}.
\end{split}
\]
The first term we have already estimated, \eqref{eq:742}. For the second term, observe that by the integral representation of $\lapms{\tau}$ and the support of $\psi$, we have
\[
 \|\lapms{\tau}|\psi|\|_{L^\infty(\R^n \backslash B)} \aleq \|\psi\|_{L^1(\R^n)} \aleq \|\psi\|_{L^{p'}(\R^n)}.
\]
Thus, for any $r' < \frac{np'}{n-\tau p'}$,
\[
 \left \|x \mapsto \int_{\R^n \backslash B} \frac{\big (\eta_1(x)-\eta_1(y) \big )(\lapms{\tau}\psi(x)-\lapms{\tau}\psi(y))}{|x-y|^{n+\tau}}\, dy\right\|_{L^{r'}(\R^n)} \aleq \|\psi\|_{L^{p'}(\R^n)}.
\]
This concludes the proof of part $(b)$.

%For the remaining cases $III$ and $I_3$, 
Next we prove part (a). 
%We observe first that, since $\tau \in (0,2)$, for any $r' \geq 1$,
%\begin{equation}\label{eq:alskd2}
% \|(1-\eta_2)\laps{2s-t} \brac{(1-\eta_1) \lapms{\tau} \psi }\|_{L^{r'}(\R^n)} \aleq \|\psi\|_{L^{p'}(\R^n)}.
%\end{equation}
As before, we split as 
\[
\begin{split}
 &\laps{\tau} \brac{(1-\eta_1) \lapms{\tau} \psi}\\
 =& \brac{\laps{\tau} (1-\eta_1)}\, \lapms{\tau} \psi + \underbrace{(1-\eta_1)\, \psi}_{=0} + c\int_{\R^n} \frac{(\eta_1(\cdot)-\eta_1(y))(\lapms{\tau}\psi(\cdot)-\lapms{\tau}\psi(y))}{|\cdot-y|^{n+\tau}}\, dy\\
 \end{split}
\]
Observe that by the disjoint support of $\psi$ and $1-\eta_2$,
\[
 \|(1-\eta_2) \lapms{2s-t} \psi\|_{L^\infty} \aleq \|\psi\|_{L^1(\R^n)} \aleq \|\psi\|_{L^{p'}(\R^n)}.
\]
Moreover, $\laps{2s-t} (1-\eta_1) = \laps{2s-t} \eta_1 \in L^1 \cap L^\infty(\R^n)$ since $\eta_1 \in C_c^\infty(\R^n)$. Consequently, for any $r' \in [1,\infty]$,
\[
 \|(1-\eta_2)\brac{\laps{2s-t} (1-\eta_1)}\, \lapms{2s-t} \psi\|_{L^{r'}(\R^n)} \aleq \|\psi\|_{L^{p'}(\R^n)}.
\]
On the other hand, for $x \in \supp (1-\eta_2)$ and $y \in \supp(\eta_1)$ we have $\eta_1(x) = 0$ and $|x-y| \ageq 1+|x|$. Thus,
\[
\begin{split}
 &\left |(1-\eta_2(x))\int_{\R^n} \frac{\big (\eta_1(x)-\eta_1(y) \big )(\lapms{2s-t}\psi(x)-\lapms{2s-t}\psi(y))}{|x-y|^{n+2s-t}}\, dy\right|\\
 \aleq &\frac{1}{1+|x|^{n+2s-t}}\, |1-\eta_2(x)|\, \brac{|\lapms{2s-t}\psi(x)| + \int_{\Omega} |\lapms{2s-t}\psi(y)|dy}\\
 \aleq &\frac{1}{1+|x|^{n+2s-t}}\, |1-\eta_2(x)|\, \|\psi\|_{L^1(\R^n)}\\
 \aleq &\frac{1}{1+|x|^{n+2s-t}}\, \|\psi\|_{L^{p'}(\R^n)}.
 \end{split}
\]
The right-hand side is now integrable for any $r' \geq 1$, and \eqref{eq:alskd2} is established.

\end{proof}
\section{A commutator estimate}\label{s:commutator}
As we described in the introduction, the crucial idea of this work is to compare the two differential operators: $\mathcal{L}^{s}_{\R^n}$ defined by, for $s \in (0,1)$,
\begin{equation}\label{eq:op1}
\langle \mathcal{L}^{s}_{\R^n} u, \varphi \rangle=\int_{\R^n}\int_{\R^n} K(x,y)\frac{ (u(x)-u(y))\, (\varphi(x)-\varphi(y))}{|x-y|^{n+2s}}\, dx\, dy 
\end{equation}
and  $L^{s,t}_{diag}$ defined by, for $t\in[s, 2s)$,
to the operator
\begin{equation}\label{eq:op2}
\langle L^{s,t}_{diag}u,\varphi  \rangle=\int_{\R^n}K(z,z) \laps{t} u\, \laps{2s-t} \varphi \,dx,
\end{equation}
for $u$ and $\varphi$ in appropriate spaces. 
Observe that if $K$ is constant, then the two operators are the same up to a multiplicative constant. Indeed, for $t\in[s, 2s)$
\[
\begin{split}
 &\int_{\R^n}\int_{\R^n} \frac{(u(x)-u(y))\, (\varphi(x)-\varphi(y))}{|x-y|^{n+2s}}\, dx\, dy= C\int_{\R^n} \laps{2s} u\,  \varphi dx= C\int_{\R^n} \laps{t} u \laps{2s-t} \varphi dx.
 \end{split}
\]
The first equality follows from the characterization of the fractional Laplacian, \eqref{eq:fraclap} and Fubini's theorem. The second equality follows from the Fourier transform characterization of fractional Laplacian, \eqref{eq:fraclapFT} and Plancherel's theorem. 

In this section, we prove a fundamental estimate for $D_{s,t}(u,\varphi) $, introduced as,
\begin{equation}\label{defn-Lcommutator}
D_{s,t}(u,\varphi) = \langle L^{s,t}_{diag}u, \varphi\rangle  - \Gamma \langle \mathcal{L}^{s}_{\R^{n}}u, \varphi\rangle 
\end{equation}
that establishes the difference $\mathcal{L}^{s}_{\R^n}-L^{s,t}_{diag}u$ is a lower order differential operator when $K$ is bounded and  uniformly H\"older continuous. In \eqref{defn-Lcommutator}, $\Gamma$ is the universal constant that ensures that $D_{s,t}(u,\varphi) = 0$ whenever $K$ is a constant kernel, 

This allows us to obtain estimates for the operator in \eqref{eq:op1} from estimates for the operator \eqref{eq:op2}, for which corresponding estimates are relatively easy to obtain  as we will see in \Cref{s:fraclap}. The main theorem of this section is the following. 
\begin{theorem}\label{th:reformulationcommie}  
Let $s \in (0,1)$, $t \in (0,1)$ such that $2s-t \in (0,1)$. Suppose also that $\alpha\in (0, 1)$ and $\Lambda>0$ are given.  Then, there exist  constants $\sigma_0 \in (0, \alpha]$  and $\Gamma = \Gamma(n,s,t) \in \R$ such that the following holds.
Let $K = K(x,y) \in \mathcal{C}(\alpha, \Lambda)$, where
\[
\mathcal{C}(\alpha, \Lambda) = \left \{ K: \R^n \times \R^n \to \R: |K(x,y)| \leq \Lambda,\text{ and \eqref{H-Continuity} is satisfied}
\right \}
\]
% be a uniformly $\gamma$-H\"older continuous and bounded function, in the sense that
%\[
 %\sup_{a,b \in \R^n} |K(a,b)| < \Lambda < \infty
%\]
%and
%\[
%\sup_{c \in \R^n} |K(a,c)-K(b,c)| + \sup_{c} |K(c,a)-K(c,b)| \leq \Lambda\, |a-b|^\gamma.
%\]
%For $u, \varphi \in C_c^\infty(\R^n)$, 
% let
%\[
% D_{s,t}(u,\varphi) := \int_{\R^n}\int_{\R^n} \frac{K(x,y) (u(x)-u(y))\, (\varphi(x)-\varphi(y))}{|x-y|^{n+2s}}\, dx\, dy - \Gamma \int_{\R^n}K(z,z) \laps{t} u\, \laps{2s-t} \varphi.
%\]
 Then for all $\sigma \in (0,\sigma_0)$ and all $\eps \in (0,\frac{\sigma}{4})$  there exists  a constant $C= C(\Lambda,\sigma,\eps)$ such that, 
\[
|D_{s,t}(u,\varphi)| \leq C\, \int_{\R^n} \lapms{\sigma-\eps} |\laps{t-\eps} u|(x)\, |\laps{2s-t} \varphi|(x)\, dx
\]
and
\[
|D_{s,t}(u,\varphi)| \leq C\, \int_{\R^n} \lapms{\sigma-\eps} |\laps{t} u|(x)\, |\laps{2s-t-\eps} \varphi|(x)\, dx
\]
for all $u \in H^{t,p}(\R^n)$ and any $\varphi \in C_c^\infty(\R^n)$. The constant $\sigma_0 \in (0,\gamma]$ depends on $s$ and $t$ in the following way: for any $\theta > 0$, if 
\[
 s \in (\theta,1-\theta),\ t \in (\theta,1-\theta),\ 2s-t \in (\theta,1-\theta)
\]
then $\sigma_0$ can be chosen dependent only on $\theta$ and $\alpha$ (but not further depending on $s$ and $t$).
\end{theorem}
Observe that $\mathcal{K}(\alpha,\lambda,\Lambda) \subset \mathcal{C}(\alpha,\Lambda)$, so \Cref{th:reformulationcommie} is applicable in our situation.
%The proof of \Cref{th:reformulationcommie} is a consequence of \Cref{la:Kest} and \Cref{la:Kest2} below. \Cref{la:Kest} was essentially proven in \cite[Proposition 6.3.]{S15}, but we repeat the proof for the sake of completeness.
\subsection{Some preliminary estimates}
In this subsection we present some preliminary estimates that will be used in the proof of \Cref{th:reformulationcommie}. 

First we observe that the exponent of the H\"older continuity of $K$ can be chosen to be very small, namely
\begin{lemma}\label{hoelderkernelsmall}
Let $0<\alpha<\beta$ and $\Lambda > 0$ then there exists $\Lambda' > 0$ such that whenever $K \in \mathcal{C}(\beta, \Lambda)$ then $K \in \mathcal{C}(\alpha, \Lambda')$
\end{lemma}
This is an easy exercise which we leave to the reader.

Secondly we recall a quite useful application of the fundamental theorem of calculus.
\begin{lemma}\label{la:fundthm}
For any $r \in\R$, there exists a constant $C=C(r)$ such that the following holds. Let $a,b \in \R^n\backslash \{0\}$ with $|a-b|\aleq \min\{|a|,|b|\}$. 
%with $|a-b| \leq \frac{1}{2}|b|$ or $|a-b| \leq \frac{1}{2} |a|$. 
Then for any $\sigma \in [0,1]$ we have
\[
 \abs{ |a|^r - |b|^{r} } \leq C\, |a-b|^{\sigma}\, \min\left \{|a|^{r-\sigma},|b|^{r-\sigma}\right \}.
\]
\end{lemma}
\begin{proof}
We may assume that $r \neq 0$ otherwise the inequality it trivial.

If $|a-b|\aleq \min\{|a|,|b|\}$, 
%$|a-b| \leq \frac{1}{2}|b|$ or $|a-b| \leq \frac{1}{2} |a|$ 
then $|a| \aeq |b|$ (with a uniform constant), that is
\[
 \min\left \{|a|^{r-\sigma},|b|^{r-\sigma}\right \} \aeq |a|^{r-\sigma}.
\]
Also for any $\sigma \in [0,1]$ we have
\[
 |a-b| \aleq |a-b|^{\sigma} |a|^{1-\sigma}.
\]
Using the above inequality, to complete the proof it suffices to show that 
\[
 \abs{ |a|^r - |b|^{r} } \aleq |a-b|\, |b|^{r-1}.
\]
To that end, dividing by $|b|^{r}$, the above is equivalent to showing 
\[
 \abs{ \abs{\frac{a}{|b|}}^r - \abs{\frac{b}{|b|}}^{r} } \aleq  \abs{\frac{a}{|b|}-\frac{b}{|b|}}
\]
Observe that since $|a| \aeq |b|$, there are uniform constants $0<R_1 <1< R_2<\infty$ such that both $\frac{a}{|b|}$ and ${\frac{b}{|b|}}$ are in $A:= B_{R_2}(0) \backslash B_{R_1}(0)$. So, the problem is now reduced to showing 
\[
 \abs{ \abs{u}^r - \abs{v}^{r} } \leq C\, \abs{u-v} \quad \forall u,v \in A.
\]
% But this last inequality follows directly from the Mean Value Theorem, after noting that the function $f(u) = |u|^{r}$ is a smooth function on $A$ for any $r\in \R$.
Since $A$ is an annulus, for any $u,v \in A$ there exists a curve $\gamma \subset A$ with $\gamma(0) = u$, $\gamma(1) = v$, $|\gamma'| \aeq  |u-v|$ -- with constants depending only on $r_1$ and $r_2$ (and thus uniform). Set 
\[
\eta(t) := |\gamma(t)|^\alpha.
\]
Then, the fundamental theorem of calculus implies
\[
\abs{ \abs{u}^\alpha - \abs{v}^{\alpha} } \leq \sup_{t \in [0,1]} |\eta'(t)| \aleq |\gamma(t)|^{\alpha-1} |\gamma'(t)| \aleq |u-v|.
\]
\end{proof}
%\begin{remark}We make the observation that for fixed $\lambda$, $\alpha\leq \gamma$
%
% useful for the application of \Cref{la:Kest} and \Cref{la:Kest2} is that w.l.o.g. the H\"older coefficient of the kernel is small (which is helpful for technical reasons).
%\begin{lemma}\label{hoelderkernelsmall}
%Assume that $K: \R^n \times \R^n \to \R$ is $C^\alpha$-H\"older continuous for some $\alpha \in (0,1)$. That is, there exists $\Lambda > 0$ such that 
%\[
%\sup_{a,b} |K(a,b)| \leq \Lambda,
%\]
%and
%\[
%\sup_{c} |K(a,c)-K(b,c)|+ \sup_{c} |K(c,a)-K(c,b)| \leq \Lambda\, |a-b|^\alpha.
%\]
%Then $K$ is $C^{\tilde{\alpha}}$-H\"older continuous for all $\tilde{\alpha} \in (0,\alpha]$.
%\end{lemma}
%
%\end{remark}
The following Lemma was essentially proven in \cite[Proposition 6.3.]{S15}. 
\begin{lemma}\label{la:Kest}
Let $m \in (0,n)$, $\alpha \in (0,1)$, and $\lambda,\Lambda>0$ are given. 
%Then  
%Assume that $K$ is uniformly $\alpha$-H\"older continuous in the following sense: there exist $\Lambda > 0$ such that
%\[
%\sup_{a,b} |K(a,b)| \leq \Lambda,
%\]
%\[
%\sup_{c} |K(a,c)-K(b,c)|+ \sup_{c} |K(c,a)-K(c,b)| \leq \Lambda\, |a-b|^\alpha.
%\]
Then for any $\beta$ such that
\[
\alpha< \beta < \min\{m+\alpha,1\}
\]
and any $K\in \mathcal{C}(\alpha, \Lambda)$ we have
\[
 \left |K(x,y) - K(z,z) \right | \ \left | |x-z|^{m-n} -|y-z|^{m-n} \right | \aleq |x-y|^{\beta} \brac{|x-z|^{m+\alpha-\beta-n}+|y-z|^{m+\alpha-\beta-n}}.
\]
\end{lemma}
\begin{proof}
We first observe that we can estimate the difference $ \left |K(x,y) - K(z,z) \right |$ in three different ways \begin{equation}\label{eq:Kest}
 \left |K(x,y) - K(z,z) \right | \aleq \begin{cases}
                                           |x-z|^\alpha +|y-z|^\alpha\\
                                           |x-y|^\alpha +|x-z|^\alpha\\
                                           |x-y|^\alpha +|y-z|^\alpha.\\
                                          \end{cases}
\end{equation}
%We have three choices to estimate the first term involving $K$: choice 1 is
The first one can be obtained  by adding and subtracting $K(x,z):$ \[
 \left |K(x,y) - K(z,z) \right |  \leq \left |K(x,y) - K(x,z) \right |+ \left |K(x,z) - K(z,z) \right | \aleq |y-z|^\alpha + |x-z|^\alpha.
\]
The second and third forms are obtained in similar ways as  
\[
\begin{split}
 \left |K(x,y) - K(z,z) \right |  \leq& \left |K(x,y) - K(x,x) \right |+ \left |K(x,x) - K(x,z) \right |+\left| K(x,z) - K(z,z) \right |\\
  \aleq&|x-y|^\alpha+ 2|x-z|^\alpha,
 \end{split}
\]
and 
\[
\begin{split}
 \left |K(x,y) - K(z,z) \right |  \leq& \left |K(x,y) - K(y,y) \right |+ \left |K(y,y) - K(y,z) \right |+\left| K(y,z) - K(z,z) \right |\\
  \aleq&|x-y|^\alpha+ 2|y-z|^\alpha.
 \end{split}
\]
 The entire expression $\left|K(x,y) - K(z,z) \right | \ \left | |x-z|^{m-n} -|y-z|^{m-n}\right|$ can now be estimated by considering these three cases. To that end, 
first, if $|x-y| < \frac{1}{2}|x-z|$ or $|x-y| < \frac{1}{2}|y-z|$ then
\[
 |x-z| \aeq |y-z|,
\]
and thus by the mean value theorem, \Cref{la:fundthm},
\[
 \left | |x-z|^{m-n} -|y-z|^{m-n} \right | \aleq |x-y|\, |x-z|^{m-1-n}.
\]
So we take the first option in the estimate for $K$ \eqref{eq:Kest} and have under our assumptions on $x,y,z$ (since $\beta \leq 1$
\[
 \left |K(x,y) - K(z,z) \right |\, \left | |x-z|^{m-n} -|y-z|^{m-n} \right |\aleq |x-y| |x-z|^{m+\alpha-1-n} \aleq |x-y|^\beta |x-z|^{m+\alpha-\beta-n}
\]
Second, if $|x-y| \geq \frac{1}{2}|x-z|$ and $|x-y| \geq \frac{1}{2}|y-z|$ and $|x-z| < |y-z|$, we have
\[
 \left | |x-z|^{m-n} -|y-z|^{m-n} \right | \aeq |x-z|^{m-n}.
\]
In this case we choose the second estimate for the estimate of $K$ \eqref{eq:Kest} and obtain (since $\beta \in (\alpha,m+\alpha)$,
% yes. [{\color{blue} TM: Do you mean here: $\beta \in (\alpha,s+\alpha)$}]
\[
 \begin{split}
\left |K(x,y) - K(z,z) \right |\, \left | |x-z|^{m-n} -|y-z|^{m-n} \right |
 \aleq& |x-y|^{\alpha}\, |x-z|^{m-n}+ |x-z|^{\alpha+m-n} \\
 \aleq& |x-y|^{\beta}\, |x-z|^{\alpha+m-\beta -n}.
 \end{split}
\]
Finally, if $|x-y| \geq \frac{1}{2}|x-z|$ and $|x-y| \geq \frac{1}{2}|y-z|$ but $|x-z| \geq |y-z|$, we have by a symmetric argument
\[
 \begin{split}
\left |K(x,y) - K(z,z) \right |\, \left | |x-z|^{m-n} -|y-z|^{m-n} \right |
 \aleq& |x-y|^{\beta}\, |y-z|^{\alpha+m-\beta -n}.
 \end{split}
\]
\end{proof}

\begin{lemma}\label{la:Kest2}Let $\lambda, \Lambda >0$ be given. 
Suppose also that $s,t \in (0,1)$ with $2s-t \in (0,1)$ in the following form: assume that for some $\theta \in (0,1)$, 
\begin{equation}\label{eq:Kest2:st2smt}
 s \in (\theta,1-\theta),\ t \in (\theta,1-\theta),\ 2s-t \in (\theta,1-\theta). 
\end{equation}
Then there exists $\alpha_0=\alpha_0(\theta)$ such that for any $\alpha \in (0,\alpha_0)$, $\eps \in (0,\frac{\alpha}{3})$,  and $K\in \mathcal{C}(\alpha, \Lambda)$ the following holds.
%
%Assume that $K$ is uniformly $\alpha$-H\"older continuous in the following sense: there exist $\Lambda > 0$ such that
%\[
%\sup_{a,b} |K(a,b)| \leq \Lambda,
%\]
%\[
%\sup_{c} |K(a,c)-K(b,c)|+ \sup_{c} |K(c,a)-K(c,b)| \leq \Lambda\, |a-b|^\alpha.
%\]
For $i,j=1,2$ set 
\[
 M^{\eps}_{i,j}(z_1,z_2) = \int_{\R^n}\int_{\R^n} |K(x,y) - K(z_j,z_j)|\,   |\kappa^{\eps}_i(x,y,z_1,z_2)|\, dx\, dy.
\]
where
\[
 \kappa^{\epsilon}_1(x,y,z_1,z_2) := \frac{\brac{|x-z_1|^{t-\eps-n}-|y-z_1|^{t-\eps-n}}\, \brac{|x-z_2|^{2s-t-n}-|y-z_2|^{2s-t-n}}}{|x-y|^{n+2s}},
\]
\[
 \kappa^{\eps}_2(x,y,z_1,z_2) := \frac{\brac{|x-z_1|^{t-n}-|y-z_1|^{t-n}}\, \brac{|x-z_2|^{2s-t-\eps-n}-|y-z_2|^{2s-t-\eps-n}}}{|x-y|^{n+2s}}
\]
Then for any $f,g \in C_c^\infty(\R^n)$,
\[
 \int_{\R^n} \int_{\R^n} f(z_1)\, g(z_2) M^{\epsilon}_{i,j}(z_1,z_2) dz_1 dz_2 \leq C(\Lambda,\theta) \int_{\R^n}\lapms{\alpha-\eps} |f|(x)\, |g|(x)\, dx, \quad i,j=1,2.
 \]
\end{lemma}
\begin{proof}
We prove the lemma by taking 
\begin{equation}\label{eq:Kest2:alpha0}
 \alpha_0 := \frac{1}{10} \min\{\theta, 1-\theta\}.
\end{equation}
To that end, assume that $\alpha < \alpha_0$, $\eps < \frac{\alpha}{3}$ from now on. We will only consider the case of  $M^{\eps}_{12}$; the estimate of the other $M^{\eps}_{ij}$ is analogous. To simplify notation we write $\kappa^\epsilon := \kappa^{\epsilon}_1$ and $M^\eps:=M^{\eps}_{12}$. 

\newcommand{\Oo}{\mathcal{O}}
\newcommand{\Ot}{\mathcal{P}}
We begin writing $M^\epsilon (z_1, z_2)$ as 
\[
\begin{split}
M^\epsilon(z_1,z_2) &\leq \sum_{i, j=1}^{3}\iint_{\Oo_{i}\cap \Ot_{j}} |K(x,y) - K(z_2,z_2)|\,   |\kappa^\epsilon(x,y,z_1,z_2)|\, dx\, dy\\
 &=: \sum_{i,j=1}^{3} J_\epsilon^{(i, j)}(z_1, z_2), 
\end{split}
\]
where the regions of integration are given by
%\begin{itemize}
 \[
 \begin{split}
 \Oo_{1}&=\{(x, y): |x-y| \aleq \min\{|x-z_1|, |y-z_1|\}\} \\
 %(\text{and therefore $|x-z_1| \aeq |y-z_1|$})\\
\Oo_{2}&=\{(x, y): |x-z_1| \aleq \min\{|y-z_1| , |x-y|\}\}  \\
%(\text{and therefore $|y-z_1| \aeq |x-y|$}),\\
\Oo_{3}&=\{(x, y): |y-z_1| \aleq \min\{|x-z_1| , |x-y|\}\}\\
%  (\text{and therefore $|x-z_1| \aeq |x-y|$}),
\end{split}\]
%\end{itemize}
and 
\[
 \begin{split}
 \Ot_{1}&=\{(x, y): |x-y| \aleq \min\{|x-z_2|, |y-z_2|\}\}\\
 % (\text{and therefore $|x-z_2| \aeq |y-z_2|$})\\
  \Ot_{2}&=\{(x, y): |x-z_2| \aleq \min\{|y-z_2| , |x-y|\}\} \\
  % (\text{and therefore $|y-z_2| \aeq |x-y|$} )\\
  \Ot_{3}&= \{(x, y): |y-z_2| \aleq \min\{|x-z_2| , |x-y|\}\} \\
  % (\text{and therefore $|x-z_2| \aeq |x-y|$}). 
\end{split}
\]
Then we have 
\[
 \int_{\R^n} \int_{\R^n} f(z_1)\, g(z_2) M^{\epsilon}(z_1,z_2) dz_1 dz_2 = \sum_{i,j}  \int_{\R^n} \int_{\R^n} f(z_1)\, g(z_2) J_{\epsilon}^{i,j}(z_1,z_2) dz_1 dz_2. 
\]
We will estimate the integral that involves each of these terms. 
 
\underline{Estimating terms involving $J_\epsilon^{(1, 1)}, J_\epsilon^{(1, 2)}, J_\epsilon^{(1, 3)}$ and $J_\epsilon^{(2, 1)}$}:  

We begin by noting that for $(x, y)\in  \Oo_{1}$, from \Cref{la:fundthm} by taking $r=t-\epsilon-n$ for $\epsilon$ small, for any $0 \leq \sigma \leq 1$ and any $(x, y)$ 
\[
 \abs{|x-z_1|^{t-\eps-n}-|y-z_1|^{t-\eps-n}} \aleq |x-y|^{\sigma}\, \left(|x-z_1|^{t-\eps-\sigma-n} +|y-z_1|^{t-\eps-\sigma-n} \right).
\]
%provided $|x-y|\aleq \min\{|x-z_1|, |y-z_1| \}$.   
Moreover, from \Cref{la:Kest}   by taking $m=2s-t$, $\alpha<\alpha_0$, for any $\beta < 2s-t+\alpha < 1-{\frac{9\theta}{10}}<1$ and $(x, y)$  
\[
\begin{split}
 |K(x,y) - K(z_2,z_2)|\, &\brac{|x-z_2|^{2s-t-n}-|y-z_2|^{2s-t-n}}\\
  &\aleq |x-y|^\beta\, \left(|x-z_2|^{2s-t+\alpha-\beta-n}+|y-z_2|^{2s-t+\alpha-\beta-n}\right). 
\end{split}
\]
Combining the above two inequalities we obtain that for $0<\epsilon < \frac{\alpha}{3}$, for any $\beta < 2s-t+\alpha$, any $\sigma\in[0,1]$ and any $(x, y)\in  \Oo_{1}$
\begin{equation}\label{grand-inequality}
\begin{split}
|K(x,y) - K(z_2,z_2)|\,  & |\kappa_\epsilon(x,y,z_1,z_2)|\\
 &\aleq |x-y|^{-2s-n+\beta+\sigma}(|x-z_1|^{t-\epsilon-\sigma-n} + |y-z_1|^{t-\epsilon-\sigma-n}) \\
 &\quad\quad \quad\quad\times (|x-z_2|^{2s-t+\alpha-\beta-n} + |y-z_2|^{2s-t+\alpha-\beta-n} ) 
 \end{split}
\end{equation}

%For $(x, y)\in  \Oo_{1}$, we have $|x-z_1| \aeq |y-z_1|$. Therefore, applying \Cref{la:fundthm} for $r = t-\eps-n$, we can estimate that 
%for any $0 \leq \sigma \leq 1$
%\[
% \abs{|x-z_1|^{t-\eps-n}-|y-z_1|^{t-\eps-n}} \aleq |x-y|^{\sigma}\, |x-z_1|^{t-\eps-\sigma-n}.
%\]
%Moreover, from \Cref{la:Kest}, by taking $m=2s-t$, we have that for $\beta < 1$, $\beta < 2s-t+\alpha$
%\[
% |K(x,y) - K(z_2,z_2)|\, \brac{|x-z_2|^{2s-t-n}-|y-z_2|^{2s-t-n}} \aleq |x-y|^\beta\, |x-z_2|^{2s-t+\alpha-\beta-n}, 
%\]
%after noting that for $(x, y)\in  \Ot_{1}$, $|x-z_2| \aeq |y-z_2|$. 
Now for $(x, y)\in \Oo_{1}\cap \Ot_{1}$, \eqref{grand-inequality} reduces to 
%Combining the above estimates we get that 
\[
 |K(x,y) - K(z_2,z_2)|\,   |\kappa(x,y,z_1,z_2)| \aleq |x-y|^{\beta+\sigma-2s-n}\, |x-z_2|^{2s-t+\alpha-\beta-n}\, |x-z_1|^{t-\eps-\sigma-n}
\]
after noting that in this case $|x-z_1| \aeq |y-z_1|$ and $|x-z_2| \aeq |y-z_2|$. 
In view of \eqref{eq:Kest2:st2smt} and \eqref{eq:Kest2:alpha0}, we can choose $\beta$ slightly smaller than $2s-t+\alpha$ and $\sigma$ slightly smaller than $t-\eps$ and still ensure $\beta+\sigma-2s > 2s+\alpha -\eps> 2\theta >0$.
For each $x$, 
\[
 \int_{\{|x-y| \aleq \min\{|x-z_1|, |x-z_2|\}\}}  |x-y|^{\beta+\sigma-2s-n}\, dy \aleq |x-z_1|^{\sigma-t+\frac{\alpha}{2}}\, |x-z_2|^{\beta-2s+t-\frac{\alpha}{2}}
\]
and therefore, \[
 \iint_{\Oo_{1}\cap \Ot_{1}} |K(x,y) - K(z_2,z_2)|\,   |\kappa(x,y,z_1,z_2)| dx\, dy \aleq  \int_{\R^n} |x-z_2|^{\frac{\alpha}{2}-n}\, |x-z_1|^{\frac{\alpha}{2}-\eps-n}\, dx.
\]
From this we conclude that 
\[
\begin{split}
 \iint_{\R^n}  J^{(1,1)}_\eps(z_1,z_2)\, f(z_1)\, g(z_2) dz_1 dz_2 &\aleq \int_{\R^n}\lapms{\frac{\alpha}{2}-\eps} |f|(x)\, \lapms{\frac{\alpha}{2}} |g|(x)\, dx\\
 & =\int_{\R^n}\lapms{\alpha-\eps} |f|(x)\, |g|(x)\, dx, 
\end{split}
 \]
  where the last inequality follows by integration by parts.

For $(x, y)\in \Oo_{1}\cap \Ot_{2}$, \eqref{grand-inequality} reduces to 
\[
 |K(x,y) - K(z_2,z_2)|\,   |\kappa(x,y,z_1,z_2)| \aleq |x-y|^{\beta+\sigma-2s-n}\, |x-z_2|^{2s-t+\alpha-\beta-n}\, |x-z_1|^{t-\eps-\sigma-n}
\]
for our choice of $\beta < 1$ , $\sigma \in (0,t-\eps)$. In view of \eqref{eq:Kest2:st2smt} and \eqref{eq:Kest2:alpha0}, in fact we choose $\beta := 2s-t + \alpha/2 > \theta > \alpha$ to get the estimate that 
\[
 |K(x,y) - K(z_2,z_2)|\,   |\kappa(x,y,z_1,z_2)| \aleq |x-y|^{\sigma-t+\alpha/2-n}\, |x-z_2|^{\frac{\alpha}{2}-n}\, |x-z_1|^{t-\eps-\sigma-n}. 
\]
If $\sigma$ is close enough to $t-\eps$ and since $\eps < \alpha/2$, we can integrate
\[
 \int_{|x-y| \aleq |x-z_1|}  |x-y|^{\sigma-t+\alpha/2-n}\,|x-z_2|^{\frac{\alpha}{2}-n}\, |x-z_1|^{t-\eps-\sigma-n} dy \aleq 
 \,|x-z_2|^{\frac{\alpha}{2}-n}\, |x-z_1|^{\alpha/2-\eps-n}
\]
Arguing in the previous case, we obtain 
\[
\begin{split}
 \iint_{\R^n}  J^{(1,2)}_\eps(z_1,z_2)\, f(z_1)\, g(z_2) dz_1 dz_2 &\aleq \int_{\R^n}\lapms{\frac{\alpha}{2}-\eps} |f|(x)\, \lapms{\frac{\alpha}{2}} |g|(x)\, dx\\
 & =\int_{\R^n}\lapms{\alpha-\eps} |f|(x)\, |g|(x)\, dx, 
\end{split}
 \]
%Estimating $J_\eps^{(1, 2)}(z_1, z_2)$
For $(x,y)\in \Oo_1\cap \Ot_3$ or $(x, y)\in \Oo_2\cap \Ot_1$ and \eqref{grand-inequality} reduces to 
%use the estimate 
\[
\begin{split}
 |K(x,y) - K(z_2,z_2)|\,   |\kappa(x,y,z_1,z_2)| 
 %&\aleq |x-y|^{\beta+\sigma-2s-n}\, |x-z_2|^{2s-t+\alpha-\beta-n}\, |x-z_1|^{t-\eps-\sigma-n}\\
 &\aleq |x-y|^{\beta+\sigma-2s-n}\, |y-z_2|^{2s-t+\alpha-\beta-n}\, |y-z_1|^{t-\eps-\sigma-n}
\end{split}
\]
%For  (A2), (B1)  we use the estimate 
%\[
% |K(x,y) - K(z_2,z_2)|\,   |\kappa(x,y,z_1,z_2)| \aleq |x-y|^{\beta+\sigma-2s-n}\, |x-z_1|^{2s-t+\alpha-\beta-n}\, |x-z_2|^{t-\eps-\sigma-n}\\
%% &\aleq |x-y|^{\beta+\sigma-2s-n}\, |y-z_2|^{2s-t+\alpha-\beta-n}\, |y-z_1|^{t-\eps-\sigma-n}
%\]
As before we choose $\beta := 2s-t + \alpha/2$ (which is  greater than $\alpha$) to obtain that when $(x,y)\in \Oo_1\cap \Ot_3$
\[
 \int_{|x-y| \aleq |y-z_1|}  |x-y|^{\sigma-t+\alpha/2-n}\,|y-z_2|^{\frac{\alpha}{2}-n}\, |y-z_1|^{t-\eps-\sigma-n} dx \aleq 
 \,|y-z_2|^{\frac{\alpha}{2}-n}\, |y-z_1|^{\alpha/2-\eps-n}, 
\]
from which, we have 
\[
\iint_{\R^{2n} }J_\epsilon^{(1,3)}(z_1, z_2) f(z_1)\, g(z_2) dz_1 dz_2\aleq \int_{\R^{d}} I^{{\frac{\alpha}{2}}- \epsilon}(|f|)(y) I^{{\frac{\alpha}{2}}}(|g|)(y)dy = \int_{\R^{n}}I^{\alpha-\epsilon}|f|(y) |g|(y)dy.
\]
For $(x, y)\in \Oo_2\cap \Ot_1$, we have  
\[
 \int_{|x-y| \aleq |y-z_2|}  |x-y|^{\sigma-t+\alpha/2-n}\,|y-z_2|^{\frac{\alpha}{2}-n}\, |y-z_1|^{t-\eps-\sigma-n} dx \aleq 
 \,|y-z_2|^{\sigma - t+\alpha-n}\, |y-z_1|^{t-\epsilon-\sigma-n}
\]
From this it follows that 
\[
 \iint_{\R^{2n}} J^{(2,1)}_\epsilon(z_1,z_2) f(z_1) g(z_2) dz_1 dz_2 \aleq \int_{\R^n} \lapms{t-\eps-\sigma} |f|(y)\, \lapms{\sigma-t+\alpha} |g|(y)\, dy, 
\]
 and integration by parts leads to the same estimate.

\underline{Estimating $J^{(2, 2)}_{\eps}:$}

On the one hand, for $(x, y)\in \Oo_2\cap \Ot_2$, we have 
%\underline{Case (A2), (B2)}
%In this case we get 
(since $|x-z_2|^\alpha \aleq |y-z_2|^\alpha$),
\[
|K(x,y)-K(z_2,z_2)|\aleq |x-z_2|^\alpha + |y-z_2|^\alpha \aleq |y-z_2|^\alpha.
\]
On the other hand, $|y-z_1|^{-1} \aleq |x-z_1|^{-1}$ and $|y-z_2|^{-1} \aleq |x-z_2|^{-1}$, and thus
\[
 \kappa^\eps(x,y,z_1,z_2) \aleq \frac{|x-z_1|^{t-\eps-n}\, |x-z_2|^{2s-t-n}}{|x-y|^{n+2s}}.
\]
This leads to
\[
\begin{split}
 |K(x,y) - K(z_2,z_2)|\,  & |\kappa_\eps(x,y,z_1,z_2)| \\
 \aleq& |x-z_1|^{t-\eps-n}\, |x-z_2|^{2s-t-n}\, |y-z_2|^{\alpha}\, |x-y|^{-2s-n}\\
  \aleq&|x-z_1|^{t-\eps-n}\, |x-z_2|^{2s-t-n}\, |x-y|^{\alpha-2s-n},
\end{split}
 \]
where in the last step we used that $|x-y| \aeq |y-z_2|$.

In view of \eqref{eq:Kest2:st2smt} and \eqref{eq:Kest2:alpha0}, $\alpha-2s < \alpha - 2\theta < 0$, and we observe that
\[
\begin{split}
\int_{\{y: (x, y)\in \Oo_2\cap \Ot_2 \}} |x-y|^{\alpha-2s-n}\, dy \aleq& \int_{\{y: |x-y| \ageq \max\{|x-z_1|, |x-z_2|\}\}} |x-y|^{\alpha-2s-n}\, dy\\
 \aleq& \min\left \{ |x-z_1|^{\alpha -2s},\, |x-z_2|^{\alpha -2s}\right \}\\
 \aleq& |x-z_1|^{\frac{\alpha}{2}-t}\, |x-z_2|^{\frac{\alpha}{2}+t-2s}. 
 \end{split}
\]
As a consequence for each $x$
\[
 \int_{\{y: (x, y)\in \Oo_2\cap \Ot_2 \}} |K(x,y) - K(z_2,z_2)|\,   |\kappa_\eps(x,y,z_1,z_2)| dy \aleq |x-z_1|^{\frac{\alpha}{2}-\eps-n}\, |x-z_2|^{\frac{\alpha}{2}-n}
\]
That is, in this particular case
\[
\begin{split}
 \iint_{\R^{2n}} J_\eps^{(2,2)}(z_1,z_2) f(z_1) g(z_2) dz_1 dz_2  &\aleq \int_{\R^n}\lapms{\frac{\alpha}{2}-\eps} |f|(x)\, \lapms{\frac{\alpha}{2}-\eps} |g|(x)\, dx\\
 & = \int_{\R^n}\lapms{\alpha-\eps} |f|(x)\, |g|(x)\, dx.
 \end{split}
\]
\underline{Estimating $J_\eps^{(2,3)}:$}

%\underline{Case (A2), (B3)}
Since by \eqref{eq:Kest2:st2smt} and \eqref{eq:Kest2:alpha0} $\frac{\alpha}{3} < \theta - \frac{1}{10} \theta < t-\eps$,  $\frac{\alpha}{3} < \theta < 2s-t$ and  $\eps<\frac{\alpha}{3}$, we have for any $(x, y)\in  \Oo_2\cap \Ot_3$ that 
\[
\begin{split}
 |K(x,y) - K(z_2,z_2)|\,   |\kappa_\eps(x,y,z_1,z_2)| \aleq& |x-z_1|^{t-\eps-n}\, |y-z_2|^{2s-t-n}\, |x-z_2|^{\alpha}\, |x-y|^{-2s-n}\\
  \aleq& |x-z_1|^{\frac{\alpha}{3}-n}\, |y-z_2|^{\frac{\alpha}{3}-n}\, |x-y|^{2s-t-\frac{\alpha}{3} + t-\eps-\frac{\alpha}{3}+\alpha-2s-n}\\
  \aeq& |x-z_1|^{\frac{\alpha}{3}-n}\, |y-z_2|^{\frac{\alpha}{3}-n}\, |x-y|^{\frac{\alpha}{3}-\eps-n}\\
\end{split}
 \]
Thus in this case, we have that 
\[
 \iint_{\R^{2n}} J_\eps^{(2,3)}(z_1,z_2) f(z_1) g(z_2) dz_1 dz_2  \aleq \int_{\R^n} \lapms{\frac{\alpha}{3}} |f|(x) \lapms{\frac{\alpha}{3}-\eps}\brac{\lapms{\frac{\alpha}{3}} |g|}(x) = \int_{\R^n}\lapms{\alpha-\eps} |f|(x)\, |g|(x)\, dx,
\]
where we use the semigroup property of the Riesz potential. 

\underline{Estimating $J_\eps^{(3, 1)}:$}

%\underline{Case (A3), (B1)}
Here we get for any $\beta < 1$, $\beta < 2s-t+\alpha$ (in view of \eqref{eq:Kest2:st2smt} and \eqref{eq:Kest2:alpha0} $\alpha < \theta < 1-(2s-t)$), for any $(x, y)\in \Oo_3\cap \Ot_1$
\[
\begin{split}
 |K(x,y) - K(z_2,z_2)|\,   |\kappa_\eps(x,y,z_1,z_2)| \aleq& |y-z_1|^{t-\eps-n}\, |x-z_2|^{\alpha +2s-t-\beta-n} |x-y|^{\beta-2s-n}\\
  \aleq& |y-z_1|^{\frac{\alpha}{3}-n}\, |x-z_2|^{\alpha +2s-t-\beta-n} |x-y|^{\beta+t-\eps-\frac{\alpha}{3}-2s-n}
\end{split}
 \]
Taking $\beta:= 2s-t + \frac{2\alpha}{3}$ the above inequality simplifies to 
\[
\begin{split}
 |K(x,y) - K(z_2,z_2)|\,   |\kappa_\eps(x,y,z_1,z_2)| \aleq& |y-z_1|^{\frac{\alpha}{3}-n}\, |x-z_2|^{\frac{\alpha}{3} -n} |x-y|^{\frac{\alpha}{3}-\eps-n}.
\end{split}
 \]
Since $\eps < \frac{\alpha}{3}$, integrating we find that 
\[
 \iint_{\R^{2n}} J_\eps^{(3,1)}(z_1,z_2) f(z_1) g(z_2) dz_1 dz_2  \aleq \int_{\R^n} \lapms{\frac{\alpha}{3}-\eps}\lapms{\frac{\alpha}{3}} |f|(x) \lapms{\frac{\alpha}{3}}|g|(x) = \int_{\R^n}\lapms{\alpha-\eps} |f|(x)\, |g|(x)\, dx.
\]

\underline{Estimating $J^{(3,2)}_\eps:$}

%\underline{Case (A3), (B2)}
By \eqref{eq:Kest2:st2smt} and \eqref{eq:Kest2:alpha0}, $\frac{\alpha}{3} < \theta - \frac{1}{10} \theta < t-\eps$, and $\frac{\alpha}{3} < \theta < 2s-t$. Thus, for $(x, y)\in  \Oo_3\cap \Ot_2$
\[
\begin{split}
 |K(x,y) - K(z_2,z_2)|\,   |\kappa_\eps(x,y,z_1,z_2)| \aleq& |y-z_1|^{t-\eps-n}\, |x-z_2|^{2s-t-n} |y-z_2|^\alpha |x-y|^{-2s-n}\\
 \aeq& |y-z_1|^{t-\eps-n}\, |x-z_2|^{2s-t-n} |x-y|^{\alpha-2s-n}\\
 \aleq& |y-z_1|^{\frac{\alpha}{3}-n}\, |x-z_2|^{\frac{\alpha}{3}-n} |x-y|^{\alpha-2s-n+t-\eps-\frac{2\alpha}{3}+2s-t}\\
  =& |y-z_1|^{\frac{\alpha}{3}-n}\, |x-z_2|^{\frac{\alpha}{3}-n} |x-y|^{\frac{\alpha}{3}-\eps-n}.\\
\end{split}
 \]
As before, we can now estimate as 
\[
\int_{\R^{2n}} J_\epsilon^{(3, 2)}(z_1,z_2) f(z_1) g(z_2) dz_1 dz_2  \aleq \int_{\R^n} \lapms{\frac{\alpha}{3}-\eps}\lapms{\frac{\alpha}{3}} |f|(x) \lapms{\frac{\alpha}{3}}|g|(x) = \int_{\R^n}\lapms{\alpha-\eps} |f|(x)\, |g|(x)\, dx.
\]
Finally we estimate $J_{\eps}^{(3,3)}:$

%\underline{Case (A3), (B3)}
For $(x, y)\in \Oo_3\cap \Ot_3$, we have that 
\[
\begin{split}
 |K(x,y) - K(z_2,z_2)|\,   |\kappa(x,y,z_1,z_2)| \aleq& |y-z_1|^{t-\eps-n}\, |y-z_2|^{2s-t-n} |x-z_2|^\alpha |x-y|^{-2s-n}\\
 \aeq& |y-z_1|^{t-\eps-n}\, |y-z_2|^{2s-t-n}\, |x-y|^{\alpha-2s-n}\\
\end{split}
 \]
Observe that from \eqref{eq:Kest2:st2smt} and \eqref{eq:Kest2:alpha0}, we have $\alpha < 2\theta < 2s$, $\frac{\alpha}{2} < \theta < t$ and $\frac{\alpha}{2} < \theta < 2s-t$. Moreover, for any $y$
\[
\begin{split}
 \int_{\{x: (x, y)\in \Oo_3\cap \Ot_3\}} |x-y|^{\alpha-2s-n} dx \aleq& \int_{\{x: |x-y| \ageq \max\{|y-z_1|, |y-z_2|\}\}} |x-y|^{\alpha-2s-n} dx\\
 \aleq& \min\left \{|y-z_1|^{\alpha-2s},|y-z_2|^{\alpha-2s}\right \}\\
 \leq& |y-z_1|^{\frac{\alpha}{2}-t} |y-z_2|^{\frac{\alpha}{2}-2s+t}. 
\end{split}
 \]
Combining the previous two inequalities we have, 
\[
 \int_{\{x: (x, y)\in \Oo_3\cap \Ot_3\}} |K(x,y) - K(z_2,z_2)|\,   |\kappa_\eps(x,y,z_1,z_2)|\, dx \aleq |y-z_1|^{\frac{\alpha}{2}-\eps-n}\, |y-z_2|^{\frac{\alpha}{2}-n}.
\]
This implies in this case
\[
\iint_{\R^{2n}} J_\eps^{(33)}(z_1,z_2) f(z_1) g(z_2) dz_1 dz_2  \aleq \int_{\R^n} \lapms{\frac{\alpha}{2}-\eps} |f|(y)\, \lapms{\frac{\alpha}{2}} |g|(y)\, dy = \int_{\R^n}\lapms{\alpha-\eps} |f|(x)\, |g|(x)\, dx.
\]
This completes the proof of \Cref{la:Kest2}. 
\end{proof}

\begin{lemma}\label{la:integration}
Set for $s \in (0,1)$ and $t \in (0,2s)$ with $2s-t \in (0,1)$,
\[
 \kappa_{0}(x,y,z_1,z_2) := \frac{\brac{|x-z_1|^{t-n}-|y-z_1|^{t-n}}\, \brac{|x-z_2|^{2s-t-n}-|y-z_2|^{2s-t-n}}}{|x-y|^{n+2s}}
\]
then there exists a constant $c = c(s,t)$ such that
\[
 \int_{\R^n} f(z)\, g(z)\, dz = c\int_{\R^n} \int_{\R^n} \int_{\R^n} \int_{\R^n} f(z_1) g(z_2) \kappa_{0}(x,y,z_1,z_2)\, dz_1\, dz_2\, dx\, dy\, 
\]
holds for any $f \in L^p(\R^n)$, $p \in (1,\infty)$ and $g \in C^\infty(\R^n) \cap H^{1,p'}(\R^n)$.
\end{lemma}
\begin{proof}
Assume first that $f,g \in C_c^\infty(\R^n)$. Using the definitions of fractional Laplacian and Riesz potential via Fourier transform we have
\[
 \int_{\R^n} f(x)\, g(x)\, dx  = c\, \int_{\R^n} (\laps{2s} \lapms{t} f)(z)\, (\lapms{2s-t} g)(z)
\]
In view of \eqref{eq:fraclap} we thus find
\[
 \int_{\R^n} f(x)\, g(x)\, dx  = c\int_{\R^n}\int_{\R^n} \frac{\brac{\lapms{t} f(x)-\lapms{t} f(y)}\, \brac{\lapms{2s-t} g(x)-\lapms{2s-t} g(y)}}{|x-y|^{n+2s}}\, dx\, dy,
\]
which holds for any $f,g \in C_c^\infty(\R^n)$. By density, it also holds for $f \in L^p(\R^N)$ if e.g. $g \in H^{1,p'}(\R^n)$.

Recall that the Riesz potential is given by an explicit integral formula, and thus for almost every $x$ and $y$ in $\R^n$
\[
 \lapms{t} f(x)-\lapms{t} f(y) = C\int_{\R^n} \brac{|z_1-x|^{t-n}-|z_1-y|^{t-n}}\, f(z_1)\, dz_1,
\]
and
\[
 \lapms{2s-t} g(x)-\lapms{2s-t} g(y) = C\int_{\R^n} \brac{|z_2-x|^{2s-t-n}-|z_2-y|^{2s-t-n}}\, g(z_2)\, dz_2
\]
Again these formulas hold at first for $f,g \in C_c^\infty(\R^n)$ but by density they still hold for almost every $x$ and $y$ for our $f$ and $g$. This proves the above formula.
\end{proof}
\subsection{Proof of the commutator estimate}
We are now ready to present the proof of the commutator estimate given in \Cref{th:reformulationcommie}. 

\begin{proof}[Proof of \Cref{th:reformulationcommie}]
Assume first that $u, \varphi \in C_c^\infty(\R^n)$. Fix, $s\in (0, 1)$, and $t\in(0,1)$ such that $0<2s-t <1$. 
Using the inverse relationship between the fractional Laplacian and the Riesz potential, for every $x \in \R^n$,
 we have that 
\[
 u(x) = C\int_{\R^n}\, |x-z_1|^{t-n}\, \laps{t} u(z_1)\, dz_1, 
\]
and
\[
 \varphi(x) = C\int_{\R^n}\, |x-z_2|^{2s-t-n}\, \laps{2s-t} \varphi(z_2)\, dz_2.
\]
Plugging in these equations in $\langle \mathcal{L}^{s}_{\R^{n}} u, \varphi\rangle$ and interchanging the integrals we obtain that 
\[
\begin{split}
   \langle \mathcal{L}^{s}_{\R^{n}} u, \varphi\rangle= &\int_{\R^n}\int_{\R^n} K(x,y)\frac{ (u(x)-u(y))\, (\varphi(x)-\varphi(y))}{|x-y|^{n+2s}}\, dx\, dy \\
=&C^2 \int_{\R^n}\int_{\R^n}\int_{\R^n}\int_{\R^n} K(x,y)\, \laps{t} u(z_1)\, \laps{2s-t} \varphi(z_2) \kappa_0(x,y,z_1,z_2)\, dz_1 dz_2 dx\, dy \\
  \end{split}
\]
where 
\[
\kappa_0(x,y,z_1,z_2) := \frac{\brac{|x-z_1|^{t-n}-|y-z_1|^{t-n}}\, \brac{|x-z_2|^{2s-t-n}-|y-z_2|^{2s-t-n}}}{|x-y|^{n+2s}}
\]
is as defined in \Cref{la:integration}.  Notice that the constant $C^{2}$ depends only on $s, t, $ and $n$. 
Since $u,\varphi \in C_c^\infty(\R^n)$, $f(z_1) := K(z_1,z_1)\laps{t} u(z_1)$ and $g(z_2) := \laps{2s-t} \varphi(z_2)$ belong to $L^p(\R^n)$ for any $p \in [1,\infty]$, moreover $g$ belongs to $H^{1,p}(\R^n)$ for any $p \in (1,\infty)$. Consequently, by \Cref{la:integration}, 
\[
\begin{split}
\langle L^{s,t}_{diag}u, \varphi\rangle=&\int_{\R^n}K(z,z)\, \laps{t} u(z)\, \laps{2s-t} \varphi(z) dz \\
=&\int_{\R^n}\int_{\R^n}\int_{\R^n}\int_{\R^n} K(z_1,z_1)\, \laps{t} u(z_1)\, \laps{2s-t} \varphi(z_2) \kappa_0(x,y,z_1,z_2)\, dz_1 dz_2 dx\, dy .
\end{split}
\]
Thus for the choice of the constant $\Gamma = C^{2}$, we have 
\begin{equation}\label{eq:Dst1}
\begin{split}
 D_{s,t}(u,\varphi) &=    \langle \mathcal{L}^{s}_{\R^{n}} u, \varphi\rangle - \Gamma \langle L^{s,t}_{diag}u, \varphi\rangle\\
 & = \int_{\R^n}\int_{\R^n}\int_{\R^n}\int_{\R^n} \Phi(x, y, z_1,z_2)dz_1 dz_2\, dx\, dy
%&= \int_{\R^n}\int_{\R^n}\int_{\R^n}\int_{\R^n}.
\end{split}
\end{equation}
where 
\[
 \Phi(x, y, z_1,z_2):= \brac{K(x,y)-K(z_1,z_1)}\, \laps{t} u(z_1)\, \laps{2s-t} \varphi(z_2) \kappa_0(x,y,z_1,z_2).
\]
By the definition of the Riesz potential $\lapms{\sigma}$ and the fact that $\lapms{\sigma} = \brac{(-\lap)^{\frac{\sigma}{2}}}^{-1}$ for any $\sigma \in (0,n)$, we have for any $x,y \in \R^n$ and any $\eps < 2s-t$,
\[
\begin{split}
 &\int_{\R^n} \laps{2s-t} \varphi(z_2) \brac{|x-z_2|^{2s-t-n}-|y-z_2|^{2s-t-n}}\, dz_2\\
 =&c_1\brac{\lapms{2s-t}\laps{2s-t}\varphi(x)-\lapms{2s-t}\laps{2s-t}\varphi(y)}\\
 =&c_1\brac{\varphi(x)-\varphi(y)}\\
 =&c_1\brac{\lapms{2s-t-\eps}\laps{2s-t-\eps}\varphi(x)-\lapms{2s-t-\eps}\laps{2s-t}\varphi(y)}\\
 =&c_2\int_{\R^n} \laps{2s-t-\eps} \varphi(z_2) \brac{|x-z_2|^{2s-t+\eps-n}-|y-z_2|^{2s-t+\eps-n}}\, dz_2
\end{split}
 \]
 where $c_2$ will depend on $\epsilon.$
By Fubini's theorem we can thus rewrite the representation \eqref{eq:Dst1} for $D_{s,t}(u,\varphi)$ into
\[
 D_{s,t}(u,\varphi) = c 
 \int_{\R^n}\int_{\R^n}\int_{\R^n}\int_{\R^n}  \Phi_\eps(x, y, z_1,z_2)\, dx\, dy dz_1 dz_2,
\]
where $ \Phi_\eps(x, y, z_1,z_2)= \brac{K(x,y)-K(z_1,z_1)}\, \laps{t} u(z_1)\, \laps{2s-t-\eps} \varphi(z_2) \kappa_\epsilon(x,y,z_1,z_2)$ and 
\[
\kappa_\eps(x,y,z_1,z_2) := \frac{\brac{|x-z_1|^{t-n}-|y-z_1|^{t-n}}\, \brac{|x-z_2|^{2s-t-\eps-n}-|y-z_2|^{2s-t-\eps-n}}}{|x-y|^{n+2s}}.
\]
We can now estimate the latter to obtain that 
\[
 |D_{s,t}(u,\varphi)|  \aleq \int_{\R^n}\int_{\R^n}|\laps{t} u(z_1)|\, |\laps{2s-t-\eps} \varphi(z_2)|\, M^{\eps}(z_1,z_2) dz_1 dz_2.
\]
where
\[
 M^{\eps}(z_1,z_2) = \int_{\R^n}\int_{\R^n} |K(x,y) - K(z_1,z_1)|\,   |\kappa_\eps(x,y,z_1,z_2)|\, dx\, dy.
\]
Now in view of \Cref{la:Kest2} (when this $M^\eps$ correspond to $M^{\eps}_{1,2}$ of the lemma) we have for small enough $\alpha $, $\eps < \alpha/3$, and $K\in \mathcal{C}(\alpha, \Lambda)$
\[
 |D_{s,t}(u,\varphi)| \aleq \int_{\R^n} \lapms{\alpha-\eps} |\laps{t} u|(x)\, |\laps{2s-t-\eps} \varphi|(x)\, dx.
\]
The other estimate follows the same way by reversing the role of $u$ and $\varphi$ from the beginning and we conclude under the assumption that $u \in C_c^\infty(\R^n)$.

In the case that $u \in H^{t,p}(\R^n)$, but still $\varphi \in C_c^\infty(\R^n)$, take let $u_k \in C_c^\infty(\R^n)$ 
\[
 \|u_k-u\|_{H^{t,p}(\R^n)} \xrightarrow{k \to \infty} 0.
\]
Observe that since $\varphi \in C_c^\infty(\R^n)$ and $2s-1 < t$ we have 
\[
\begin{split}
 \lim_{k \to \infty} \langle \mathcal{L}^{s}_{\R^{n}} u_k, \varphi \rangle&= \int_{\R^n}\int_{\R^n} K(x,y)\frac{ (u_k(x)-u_k(y))\, (\varphi(x)-\varphi(y))}{|x-y|^{n+2s}}\, dx\, dy\\
=&\int_{\R^n}\int_{\R^n}K(x,y) \frac{ (u(x)-u(y))\, (\varphi(x)-\varphi(y))}{|x-y|^{n+2s}}\, dx\, dy
=\langle \mathcal{L}^{s}_{\R^{n}} u, \varphi \rangle. 
 \end{split}
\]
Similarly, 
\[
\begin{split}
 \lim_{k \to \infty} \langle {L}^{s,t}_{diag} u_k, \varphi \rangle &=\int_{\R^n}K(z,z)\, \laps{t} u_k(z)\, \laps{2s-t} \varphi(z) dz\\
 & = \int_{\R^n}K(z,z)\, \laps{t} u(z)\, \laps{2s-t} \varphi(z) dz=\langle {L}^{s,t}_{diag} u, \varphi \rangle. 
\end{split}
\]
Combining the above we see that 
\[
 \lim_{k \to \infty} D_{s,t}(u_k,\varphi) = D_{s,t}(u,\varphi). 
\]
Moreover, we have already shown that 
\[
\begin{split}
 |D_{s,t}(u_k,\varphi)| \aleq& \int_{\R^n} \lapms{\sigma-\eps} |\laps{t} u_k|(x)\, |\laps{2s-t-\eps} \varphi|(x)\, dx\\
\aeq & \int_{\R^n} |\laps{t} u_k|(x)\, \lapms{\sigma-\eps}|\laps{2s-t-\eps} \varphi|(x)\, dx\\
\end{split}
 \]
Again, from the $H^{t,p}$-convergence of $u_k$ (and using once again that $\varphi \in C_c^\infty(\R^n)$ is fixed so that \[\|\lapms{\sigma-\eps}|\laps{2s-t-\eps} \varphi|\|_{L^{p'}(\R^n)} < \infty,\] we find
\[
\limsup_{k \to \infty} |D_{s,t}(u_k,\varphi)| \aleq \int_{\R^n} |\laps{t} u|(x)\, \lapms{\sigma-\eps}|\laps{2s-t-\eps} \varphi|(x)\, dx\\
 \]
This concludes the proof of \Cref{th:reformulationcommie}.
\end{proof}

\section{Calderon-Zygmund theory for weighted fractional Laplace: Proof of Theorem~\ref{th:regularityKlapls:1}}\label{s:fraclap}
First, we prove the following intermediate result. 
\begin{proposition}\label{intermediate-regular-Klaps}
Let $s>0$ and $t \in (0,2s)$. Assume that for some $q \in (1,\infty)$, $\laps{t} u \in L^q(\R^n)$ is a distributional solution to
\[
 \int_{\R^n} \bar{K}(z) \laps{t} u\, \laps{2s-t} \varphi = \int_{\R^n} f_1\, \laps{2s-t} \varphi + \int_{\R^n} f_2\, \varphi\quad \forall \varphi \in C_c^\infty(\Omega).
\]
Here $\bar{K}: \R^n \to \R$ is a positive, measurable, and bounded from above and below, i.e.
\[
\Lambda^{-1} \leq \bar{K}(z) \leq  \Lambda \quad \text{a.e. }x \in \R^n.
\]
Then for any $\Omega_1 \subsubset \Omega_2 \subsubset \Omega\subset \R^n$, $p>q$, and $r\in (1, p)$ such that 
\[
\text{$r>\frac{n p}{n+(2s-t)p}$\,\, if $2s - t \leq 1$, and $r> \frac{n p}{n+p} $ if $2s-t \geq 1$}
\]

 if $f_1,f_2 \in  L^q(\R^n) \cap L^p(\Omega_2)$  then $\laps{t} u \in L^p(\Omega_1)$ with the estimate 
\begin{equation}\label{eq:asdk:goalLp}
 \|\laps{t} u\|_{L^p(\Omega_1)} \aleq \sum_{j=1}^{2}\left(\|f_j\|_{L^p(\Omega_2)} + \|f_j\|_{L^q(\R^n)}\right) + \|\laps{t} u\|_{L^r(\Omega_2)} +\|\laps{t} u\|_{L^q(\R^n)}.
\end{equation}
\end{proposition}
We delay the proof of the proposition. Rather we present two results that use the proposition. The first is \Cref{th:regularityKlapls:1} whose proof is given next. 
\begin{proof}[Proof of \Cref{th:regularityKlapls:1}]
If $p\leq q,$ there is nothing to prove.  So, we assume $p>q.$
We will use \Cref{intermediate-regular-Klaps} to iterate the estimate on successive subdomains.  Assume first that $2s-t < 1$. Let $\Omega_1 = \Omega',$ and $p_1=p$. We introduce successive subdomains 
\[
\Omega'= \Omega_1\subsubset \Omega_2 \subsubset \cdots \Omega_L \subsubset \Omega
\]
and successive positive numbers 
\[
\text{$p_1 = p$, $p_{i+1} \in [q,p_{i})$ with $p_{i+1} > \frac{n p_{i}}{n+(2s-t)p_i}$}
\]
in such a way that for some $L$, $p_L = q$. It is not difficult to see that such a finite $L$  exists depending on $p,q,n,s$  and $t$.
%Once we have \eqref{eq:asdk:goal} we can conclude by iteration. 
%Indeed (for the sake of presentation let us assume now $2s-t < 1$, obvious are needed adaptations if $2s - t \geq 1$).
%observe that by duality and uniform boundedness from below of $K$,
%\[
 %\|\laps{t} u\|_{L^p(\Omega_1)} \aleq \sup_{\psi \in C_c^\infty(\Omega_1), \|\psi\|_{L^{p'}(\R^n)} \leq 1} \int_{\R^n} K(z)\, \laps{t} u\, \psi.
%\]
%Thus, \eqref{eq:asdk:goal} implies for any $1 < r < p < \infty$ with $r > \frac{n p}{n+(2s-t)p}$,
%\[
 %\|\laps{t} u\|_{L^p(\Omega_1)} \aleq \|f_1\|_{L^p(\Omega_2)} + \|f_2\|_{L^p(\Omega_2)}+ \|f\|_{L^q(\R^n)} + \|\laps{t} u\|_{L^q(\R^n)}
 %+ \|\laps{t} u\|_{L^{r}(\Omega_2)}. 
%\]
%We can iterate this estimate on finitely many sets $\Omega_1 = \Omega' \subsubset \Omega_2 \subsubset \Omega_L \subsubset \Omega$, taking $p_1 = p$, $p_{i+1} \in [q,p_{i})$ with $p_{i+1} > \frac{n p_{i}}{n+(2s-t)p_i}$ and so that $p_L = q$ (it is easy to check that we can choose $L = L(p,q,n,s,t)$ a finite number).
By \Cref{intermediate-regular-Klaps}, in each step we have
\[
 \|\laps{t} u\|_{L^{p_i}(\Omega_i)} \aleq \sum_{j=1}^{2}\left(\|f_j\|_{L^p(\Omega)} + \|f_i\|_{L^q(\R^n)}\right) + \|\laps{t} u\|_{L^q(\R^n)}
 + \|\laps{t} u\|_{L^{p_{i+1}}(\Omega_{i+1})}.
\]
Iterating the above inequality $L$ number of times we get that  
\[
 \|\laps{t} u\|_{L^{p}(\Omega')} \aleq \sum_{j=1}^{2}\left(\|f_j\|_{L^p(\Omega)} + \|f_j\|_{L^q(\R^n)}\right) + \|\laps{t} u\|_{L^q(\R^n)}
 + \|\laps{t} u\|_{L^{q}(\Omega)}, 
\]
from which the desired inequality follows. If $2s-t\geq 1$, then an obvious modification of the above iteration lead to the inequality. 
\end{proof}
We can now prove \Cref{intermediate-regular-Klaps}. 
%  to obtain Calderon-Zygmund type estimates for the equation
%\begin{equation}\label{eq:Rieszfracgradientnormal}
% \int_{\R^n} K(z) \laps{t} u\, \laps{2s-t} \varphi \,dz= \int_{\R^n} f\, \laps{2s-t} \varphi \,dz,  \quad \forall \varphi \in C_c^\infty(\R^n), 
%\end{equation}
%where $K$ is a nonnegative kernel bounded from above and below. From the point of view of Calderon-Zygmund theory this is a manageable equation. 
\begin{proof}[Proof of \Cref{intermediate-regular-Klaps}] 
%eq:asdk:goalLp
%There is nothing to show if $p \leq q$, so we assume that $p > q$ from now on.
%
%We are going to show that for all $\Omega_1 \subsubset \Omega_2$, $\psi \in C_c^\infty(\Omega_1)$, and $1 < r < p < \infty$ that satisfy
%\[r > \begin{cases} \frac{n p}{n+(2s-t)p} \quad &\text{if $2s - t \leq 1$}\\
%        \frac{n p}{n+p} \quad &\text{if $2s-t \geq 1$}
%      \end{cases}
%\]
%we have
To prove \eqref{intermediate-regular-Klaps} we use a duality argument and show that 
\[
\sup_{\substack{0\neq \psi \in C_c^\infty(\Omega_1)\\\|\psi\|_{L^{p'}} \leq 1}} \int_{\R^{n}}  \laps{t} u\psi \,dx \leq \sum_{j=1}^{2}\left(\|f_j\|_{L^p(\Omega)} + \|f_j\|_{L^q(\R^n)}\right) + \|\laps{t} u\|_{L^q(\R^n)}
 + \|\laps{t} u\|_{L^{r}(\Omega_2)}. 
\]
Using the ellipticity of $\bar{K}$, it suffices to show that for any $ \psi \in C_c^\infty(\Omega_1),$
\begin{equation}\label{eq:asdk:goal}
\begin{split}
 &\int_{\R^n} \bar{K}(z)\, \laps{t} u\, \psi\,dz\\
 \aleq&\left(\sum_{j=1}^{2}\left(\|f_j\|_{L^p(\Omega)} + \|f_i\|_{L^q(\R^n)}\right) + \|\laps{t} u\|_{L^q(\R^n)}
 + \|\laps{t} u\|_{L^{r}(\Omega_2)}\right)\, \|\psi\|_{L^{p'}(\R^n)}.
 \end{split}
\end{equation}
%That is, once we have \eqref{eq:asdk:goal} we have proven the claim of \Cref{th:regularityKlapls:1}.
%
%So it remains to establish \eqref{eq:asdk:goal} (
To simplify notation, we will write $\Omega_2 = \Omega$. 

Let $\eta_1,\eta_2 \in C_c^\infty(\Omega)$, $\eta \equiv 1$ in a neighbourhood of $\Omega_1$ and $\eta_2 \equiv 1$ in a neighbourhood of $\supp \eta_1$. Set
\[
 \varphi := \brac{\eta_1 \lapms{2s-t} \psi},
\]
which  is now a good test function for the equation. 
Then using the inverse relationship between $\laps{2s-t}$ and $\lapms{2s-t}$, we have the  identity $$\psi = \laps{2s-t} \varphi + \eta_2 \laps{2s-t} (1-\eta_1 )\lapms{2s-t} \psi  + (1-\eta_2) \laps{2s-t} (1-\eta_1)\lapms{2s-t} \psi,$$
from which it follows that 
\[
\begin{split}
 \int_{\R^n} \bar{K}(z)\, \laps{t} u\, \psi \,dz
 %&= \int_{\R^n} \bar{K}(z)\, \laps{2s-t} \varphi \,dz+ \int_{\R^n} \bar{K}(z)\, \laps{t} u\, \laps{2s-t} \brac{(1-\eta) \lapms{2s-t} \psi}\,dz\\
 = I + II + III
\end{split}
\]
where
\[
\begin{split}
 I:=&\int_{\R^n}\bar{K}(z)\, \laps{t} u\, \laps{2s-t} \varphi\,dz, \quad\\
 II:= &\int_{\R^n} \bar{K}(z)\, \eta_2 \laps{t} u\, \laps{2s-t} \brac{(1-\eta) \lapms{2s-t} \psi}\,dz, \,\,\,\text{and} \\
 III:=& \int_{\R^n} \bar{K}(z)\, (1-\eta_2) \laps{t} u\, \laps{2s-t} \brac{(1-\eta_1) \lapms{2s-t} \psi}\,dz.
 \end{split}
\]
Now using the equation, since $\varphi$ is a valid test function, we have that 
\[
\begin{split}
 I = \int_{\R^n} \bar{K}(z)\,\laps{t} u\, \laps{2s-t} \varphi\,dz&= \int_{\R^n} f_1 \laps{2s-t} \varphi\,dz+\int_{\R^n} f_2 \varphi\,dz. 
 \end{split}
\]
The right  hand side can now be rewritten using the identity between $\varphi$ and $\psi$  as 
\[
 I = I_{1} +  I_2+I_3,
\]
where
\[
\begin{split}
 I_{1}:=&\int_{\R^n} f_1 \psi+f_2 \varphi\,dz\\
 I_2:=& \int_{\R^n} \eta_2 f \laps{2s-t} \brac{(1-\eta_1) \lapms{2s-t} \psi}\,dz\\
 I_3:=& \int_{\R^n} (1-\eta_2) f \laps{2s-t} \brac{(1-\eta_1) \lapms{2s-t} \psi} \,dz.
 \end{split}
\]
Clearly,
\[
 \int_{\R^n} f_1 \psi \,dz\aleq \|f_1\|_{L^p(\Omega)}\, \|\psi\|_{L^{p'}(\Omega)}.
\]
 Sobolev embedding, \Cref{pr:sob} together with the fact that $\psi$ is compactly supported we have,
\[
 \int_{\R^n} f_2 \varphi\,dz \aleq \|f_2\|_{L^p(\Omega)}\, \|\lapms{2s-t} \psi\|_{L^{p'}(\Omega)} \aleq \|f_2\|_{L^p(\Omega)}\, \|\psi\|_{L^{p'}(\Omega)}.
\]
That is,
\[
 |I_1| \aleq \brac{\|f_1\|_{L^p(\Omega)}+\|f_2\|_{L^p(\Omega)}}\, \|\psi\|_{L^{p'}(\Omega)}.
\]
Notice that by our choice of $r$,
\[
 r > \frac{n p}{n+(2s-t)p} \Leftrightarrow r' < \frac{n p'}{n-(2s-t)p'},
\]
and therefore,  \Cref{prop:disjoint-support} is applicable. 

To estimate $I_2$, we apply \Cref{prop:disjoint-support}, for $\tau=2s-t$, and $r=p'$ to obtain that 
\[
 |I_2| \aleq \|f\|_{L^p(\Omega)}\, \|\psi\|_{L^{p'}(\Omega)}.
\]
Moreover, again apply \Cref{prop:disjoint-support} for any $r > \frac{n p}{n-(2s-t)p}$, we can estimate $|||$ as
\[
 |II| \aleq \|\laps{t} u\|_{L^{r}(\Omega)} \|\psi\|_{L^{p'}(\R^n)}.
\]

For the remaining cases $III$ and $I_3$, we apply again \Cref{prop:disjoint-support} to estimate as 
\[
 |I_3| \aleq  \|f\|_{L^q(\R^n)} \|\psi\|_{L^{p'}(\R^n)},
\]
and
\[
|III| \aleq\|\laps{t} u\|_{L^q(\R^n)}\, \|\psi\|_{L^{p'}(\R^n)}.
\]
This was the last estimate needed for \eqref{eq:asdk:goal}, and we can conclude the proof.
\end{proof}
We finish the section by proving regularity result for weighted fractional elliptic equation when the coefficient $\bar{K}$ is H\"older continuous. In this case, we can ``differentiate the equation'', which leads to estimates of the following form.
\begin{proposition}\label{pr:regularityKlapls:2}
Let $s>0$ and $t \in [s,2s)$. Assume that for some $q \in (1,\infty)$ $\laps{t} u \in L^q(\R^n)$ is a distributional solution to
\[
 \int_{\R^n} \bar{K}(z) \laps{t} u\, \laps{2s-t} \varphi = \int_{\R^n} f_1\, \laps{2s-t} \varphi +  \int f_2\, \varphi\quad \forall \varphi \in C_c^\infty(\Omega).
\]
Assume that $K$ is positive, measurable, and bounded from above and below, i.e.
\[
\Lambda^{-1} \leq \bar{K}(x) \leq  \Lambda \quad \text{a.e. }x \in \R^n.
\]
and $\bar{K}$ is moreover uniformly H\"older continuous, i.e. for some $\gamma \in (0,1]$,
\[
 \sup_{x,y,\R^n} \frac{|\bar{K}(x)-\bar{K}(y)|}{|x-y|^\gamma} \leq \Lambda.
\]
Then for any $\beta < \min\{\gamma,2s-t\}$, and any $\Omega' \subsubset \Omega \subsubset \R^n$
\[
 \|\laps{t+\beta} u\|_{L^q(\Omega')} \leq C(\Omega,\Lambda,s,t,p,q)\, \brac{\|\laps{t} u\|_{L^q(\R^n)} +\|f_2\|_{L^{q}(\Omega)} + 
 \|\laps{\beta} f_1\|_{L^{q}(\R^n)}}.
\]
\end{proposition}

\begin{proof}[Proof of \Cref{pr:regularityKlapls:2}]
Let $\Omega' \subsubset \Omega_2 \subsubset \Omega$. To prove the proposition, we will show that for any  $\psi \in C_c^\infty(\Omega_2)$, 
%We are going to show that 
\small{\begin{equation}\label{eq:klaps2:goal}
 \int_{\R^n} \laps{t} u \laps{\beta} \psi \aleq \brac{\|\laps{t} u\|_{L^q(\R^n)} +\|f_2\|_{L^{q}(\Omega)} + 
 \|\laps{\beta} f_1\|_{L^{q}(\R^n)}
 } \|\psi\|_{L^{q'}(\R^n)}.
\end{equation}}
which by duality implies that $\laps{t+\beta} u \in L^q(\Omega')$, with
\[
 \|\laps{t+\beta} u\|_{L^q(\Omega_2)} \aleq \|\laps{t} u\|_{L^q(\R^n)} +\|f_2\|_{L^{q}(\Omega)} + 
 \|\laps{\beta} f_1\|_{L^{q}(\R^n)}.
\]
To establish \eqref{eq:klaps2:goal} observe
\[
\begin{split}
 \int_{\R^n} \laps{t} u \laps{\beta} \psi
 =\int_{\R^n} \laps{t} u\, \laps{\beta} \brac{\bar{K} \brac{\frac{1}{\bar{K}}\psi}} 
=I+II
\end{split}
 \]
 where 
 \[
 \begin{split}
  I :=& \int_{\R^n} \bar{K}\, \laps{t} u\, \laps{\beta}\brac{\frac{1}{\bar{K}}\psi}\,\,\text{and}\,\,
II :=\int_{\R^n} \laps{t} u\, [\laps{\beta}, \bar{K}] \brac{\frac{1}{\bar{K}}\psi}
\end{split}
 \]
where we used commutator notation
\[
 [\laps{\beta},f](g) = \laps{\beta}(fg)-f\, \laps{\beta}g.
\]
Now since $\bar{K}$ is $\gamma$-H\"older continuous we can apply  Coifman-McIntosh-Meyer estimate, e.g. as in \cite[Theorem 6.1.]{LS18}, combined with Sobolev inequality to obtain
\[
II\aleq  \|\laps{t} u\|_{L^q(\R^n)}\, [\bar{K}]_{C^\gamma}\, \|\frac{1}{\bar{K}} \psi\|_{L^{q'}(\R^n)} \aleq \|\laps{t} u\|_{L^q(\R^n)}\, \|\psi\|_{L^{q'}(\R^n)}.
\]
For $I$, we argue similar to the proof of \Cref{intermediate-regular-Klaps}. To that end, let $\eta \in C_c^\infty(\Omega)$, $\eta \equiv 1$ in a neighborhood of $\Omega_2$. Then, splitting $I$  using $\eta$ we get that, 
\[
\begin{split}
%   &\int_{\R^n} K(z)\, \laps{t} u\, \laps{\beta}\brac{\frac{1}{K}\psi}\\
I =&\int_{\R^n} \bar{K}(z)\, \laps{t} u\, \laps{2s-t}\brac{\eta \lapms{2s-t-\beta}\brac{\frac{1}{\bar{K}}\psi}}\\
&+\int_{\R^n} \bar{K}(z)\, \laps{t} u\, \laps{2s-t}\brac{(1-\eta) \lapms{2s-t-\beta}\brac{\frac{1}{\bar{K}}\psi}}\\
\end{split}
\]
We now use the equation and $\eta \lapms{2s-t-\beta}\brac{\frac{1}{\bar{K}}\psi}$ as a valid test function to conclude that 
\[
\begin{split}
I=&\int_{\R^n} f_1\, \laps{2s-t} \brac{\eta \lapms{2s-t-\beta}\brac{\frac{1}{\bar{K}}\psi}}+\int_{\R^n} f_2\, \brac{\eta \lapms{2s-t-\beta}\brac{\frac{1}{\bar{K}}\psi}}\\
&+\int_{\R^n} \bar{K}(z)\, \laps{t} u\, \laps{2s-t}\brac{(1-\eta) \lapms{2s-t-\beta}\brac{\frac{1}{\bar{K}}\psi}}\\
=&I_1 + I_2 +I_3
\end{split}
\]
where
\[\begin{split} 
I_1 :=& \int_{\R^n} \laps{\beta} f_1\, \laps{2s-t-\beta} \brac{\eta \lapms{2s-t-\beta}\brac{\frac{1}{\bar{K}}\psi}}\\
I_2 :=&\int_{\R^n} f_2\, \brac{\eta \lapms{2s-t-\beta}\brac{\frac{1}{\bar{K}}\psi}}\\
I_3 :=&\int_{\R^n} \bar{K}(z)\, \laps{t} u\, \laps{2s-t}\brac{(1-\eta) \lapms{2s-t-\beta}\brac{\frac{1}{\bar{K}}\psi}}\\
 \end{split}
\]
The term $I_1$ can be estimates using we can estimate with the help of \eqref{eq:alskd1} and \eqref{eq:alskd2}, in the same was we estimated $I$ of the proof of \Cref{intermediate-regular-Klaps}, which imply
\[
 |I_1| \aleq \|\laps{\beta} f_1\|_{L^{q}(\R^n)}\, \|\psi\|_{L^{q'}}.
\]
By Sobolev inequality, \Cref{pr:sob},
\[
 |I_2| \aleq \|f_2\|_{L^{q}(\Omega)}\, \|\psi\|_{L^{q'}}.
\]
As for $I_3$,
\[
 \begin{split}
  |I_3| \aleq & \|\laps{t} u\|_{L^q(\R^n)}\, \left \|\laps{2s-t}\brac{(1-\eta) \lapms{2s-t-\beta}\brac{\frac{1}{K}\psi}} \right \|_{L^{q'}(\R^n)}\\
 \end{split}
\]
Now observe that $1-\eta$ and $\psi$ have disjoint support, so that we can argue similarly to \eqref{eq:alskd2} to obtain
\[
\begin{split}
\left \|\laps{2s-t}\brac{(1-\eta) \lapms{2s-t-\beta}\brac{\frac{1}{K}\psi}} \right \|_{L^{q'}(\R^n)}\aleq \|\psi\|_{L^{q'}(\R^n)} + \left \|\brac{(1-\eta) \laps{\beta}\brac{\frac{1}{K}\psi}} \right \|_{L^{q'}(\R^n)}.
\end{split}
\]
Observe that $\psi \in C_c^\infty(\Omega_2)$ and $1-\eta \equiv 0$ in a neighborhood of $\Omega_2$. 
If $\beta = 0$ this implies 
\[
 (1-\eta)(x) \laps{\beta}\brac{\frac{1}{K}\psi}(x) \equiv 0.
\]
If $\beta > 0$ we use that for $y \in \Omega_2$ and $x \in \supp (1-\eta)$ we have $|y-x| \aeq 1+|x|$, and estimate
\[
 \left |\brac{(1-\eta) \laps{\beta}\brac{\frac{1}{K}\psi}}(x) \right | \aleq \int_{\R^n} (1+|x|)^{-n-\beta} \frac{1}{|K(y)|} |\psi(y)| \, dy,
\]
and thus
\[
 \left \|\brac{(1-\eta) \laps{\beta}\brac{\frac{1}{K}\psi}} \right \|_{L^{q'}(\R^n)} \aleq \Lambda\, \|\psi\|_{L^1(\R^n)} \aleq \|\psi\|_{L^q(\R^n)}
\]
We conclude that
\[
\left \|\laps{2s-t}\brac{(1-\eta) \lapms{2s-t-\beta}\brac{\frac{1}{K}\psi}} \right \|_{L^{q'}(\R^n)}\aleq \|\psi\|_{L^{q'}(\R^n)}.
\]
This establishes \eqref{eq:klaps2:goal} and that concludes the proof of the proposition.
\end{proof}

\section{Local to global equation}\label{s:localglobal}
The main idea for the proof of \Cref{th:main} is to use \Cref{th:reformulationcommie} to compare the equation of \Cref{th:main} with an easier equation to which we can apply \Cref{th:regularityKlapls:1} and \Cref{pr:regularityKlapls:2}. This works well on a local scale and the improvement of differentiability and integrability is each time incremental. So we apply this strategy repeatedly, which means that we repeatedly need to use cutoff arguments to restrict our equation to the set where we already have shown some improvement for differentiability and integrability. We describe this cutoff argument in this section. The next theorem states that if for a given $\Omega_1\subsubset \Omega$, $u$ solves the equation 
\[
 \mathcal{L}^{s}_{\Omega}u = F,\quad \text{in $\Omega_{1}$}, 
\]
then $u$ can be extended in $\R^{n}$ in a controlled way. Namely, the extension $v$ solves an equation of the form 
\[
 \mathcal{L}^{s}_{\R^{n}}v = G,\quad \text{in $\R^{n}$}
\]
and the norm of $v$ is controlled by $u$, and the norm of data $G$ is controlled by the norms $u$  and $F$. To be precise, we have the following. 
\begin{theorem}\label{th:reduction}
Let $\Omega_2 \subsubset \Omega_1 \subsubset \Omega \subseteq \R^n$ be open sets.
Take $s \in (0,1)$, $t \in [s,1)$ and $p,q \in [2,\infty)$, $r \in (0,1)$ (if $n=1$ additionally, $r \leq s$) satisfying the following conditions:
% \footnote{From the fact that 
% \[
% \int u \varphi
% \]
% belongs locally to $(W^{r,{q'}})^\ast$ if $u \in W^{t,p}$ we would get 
% \[
% r-\frac{n}{{q'}} \geq t-\frac{n}{p'}
% \]
% }
\begin{equation}\label{eq:red:rpq}
\frac{1}{q} \geq \frac{1}{p}-\frac{r}{n}, \quad \text{and} \quad \frac{1}{q} \geq \frac{1}{p}-\frac{t}{n} \quad \text{and} \quad \frac{1}{q} \geq \frac{1}{p}-\frac{1-2s+t}{n},
\end{equation}
\begin{equation}\label{eq:red:rpqt}
 \frac{1}{q} > \frac{1}{p}- \frac{r+1-2s+t}{n}
\end{equation}
and
% \footnote{this restriction again leads later to estimates at most of order $1-\eps$}
\begin{equation}\label{eq:red:sr}
 2s-1 < r.
\end{equation}

Suppose that $K\in L^{\infty}(\R^{n}\times \R^{n})$. 
%Let $K: \R^n \times \R^n \to \R$ be a bounded, measurable map, in particular for some $\Lambda > 0$
%\[
% \sup_{x,y \in \R^N} |K(x,y)|\leq \Lambda.
%\]
% 
For any $u \in H^{s,2}(\R^n)$ such that $\laps{t} u \in L^p(\Omega_1)$ satisfies for some $f_1,f_2 \in L^{q}(\R^n)$ the equation
\[
 %\int_{\Omega}\int_{\Omega} \frac{K(x,y) (u(x)-u(y))\, (\varphi(x)-\varphi(y))}{|x-y|^{n+2s}}\, dx\, dy 
\langle \mathcal{L}^{s}_{\Omega}u, \varphi\rangle= 
  \int_{\R^n} f_1\, \laps{r} \varphi + \int_{\R^n} f_2\, \varphi,\,\,\quad \forall \varphi \in C_c^\infty(\Omega_1).
\]
Then there exist $v \in H^{s,2}\cap H^{t,p}(\R^n)$, $\supp v \subset \Omega_1$, such that $u \equiv v$ in $\Omega_2$ and $g_1,g_2 \in L^{q}(\R^n)$ such that
\begin{equation}\label{eq:red:pdeg}
 %\int_{\R^n}\int_{\R^n} \frac{K(x,y) (v(x)-v(y))\, (\varphi(x)-\varphi(y))}{|x-y|^{n+2s}}\, dx\, dy 
 \langle\mathcal{L}^{s}_{\R^{n}} v, \varphi\rangle= \int_{\R^n} g_1 \laps{r} \varphi + \int_{\R^n} g_2 \varphi\quad \forall \varphi \in C_c^\infty(\R^n).
\end{equation}
Moreover,
\begin{equation}\label{eq:red:vest}
 %\|v\|_{L^2(\R^n)} + \|\laps{t} v\|_{L^2(\R^n)}+
\|v\|_{H^{2, t}(\R^n)} + \|\laps{t} v\|_{L^p(\R^n)} \aleq  
%\|u\|_{L^2(\R^n)}+\|\laps{s} u\|_{L^2(\R^n)}
\|u\|_{H^{2, s}(\R^n)}+\|\laps{t} u\|_{L^p(\Omega_1)},
\end{equation}
and
\begin{equation}\label{eq:red:g12est}
 \|g_1\|_{L^{q}(\R^n)} +\|g_2\|_{L^{q}(\R^n)} \aleq \|f_1\|_{L^{q}(\R^n)}+\|f_2\|_{L^{q}(\R^n)} +\|u\|_{H^{2, s}(\R^n)}+
 %+\|u\|_{L^2(\R^n)} + \|\laps{s} u\|_{L^2(\R^n)}+
 \|\laps{t} u\|_{L^p(\Omega_1)},
\end{equation}
Additionally, for any $\beta \in (0,1)$ such that 
\begin{equation}
\label{eq:red:beta} 
2s-1+\beta < r \quad \text{ and } \quad t+\frac{n}{q} > \frac{n}{p}- r-1+2s+\beta,
\end{equation} we have, whenever the right-hand side is finite,
\begin{equation}\label{eq:lapsbetag1est}
 \|\laps{\beta} g_1\|_{L^{q}(\R^n)} \aleq \|f_1\|_{H^{\beta,p}(\R^n)}+\|u\|_{H^{2, s}(\R^n)}
%  +
 %\|\laps{\beta} f_1\|_{L^{q}(\R^n)}+\|f_1\|_{L^{q}(\R^n)} + \|u\|_{L^2(\R^n)}+\|\laps{s} u\|_{L^2(\R^n)}
 +\|\laps{t} u\|_{L^p(\Omega_1)}.
\end{equation}
Above, $g_1$ and $g_2$ and $v$ are independent of $q$ and $\beta$, in the sense that if we apply the statement above to $f_1,f_2 \in L^{q_1}\cap L^{q_2}$ then there is one set of functions $g_1,g_2,v$ satisfying the equations and the estimates in $L^{q_1}$ and $L^{q_2}$.
\end{theorem}
We split the proof of \Cref{th:reduction} into several steps.

The first step is a cutoff argument, essentially replacing $u$ with $\eta u$ for a suitable cutoff function $\eta$.
\begin{lemma}\label{la:red:1}
Under the assumptions of  \Cref{th:reduction}, let $\Omega_2 \subsubset \tilde{\Omega} \subsubset \Omega_1$.
Then there exist $w \in H^{s,2}(\R^n) \cap H^{t,p}(\R^n)$ with $\supp w \subset \Omega_1$, $w \equiv u$ in a neighborhood of $\Omega_2$, and $g_1,g_2 \in L^q(\R^n)$ such that
\begin{equation}\label{eq:red:1:pde}
 %\int_{\Omega}\int_{\Omega} \frac{K(x,y) \big (w(x)-w(y) \big )\, (\varphi(x)-\varphi(y))}{|x-y|^{n+2s}}\, dx\, dy 
\langle \mathcal{L}^{s}_{\Omega}w, \varphi\rangle  = \int_{\R^n} g_1\, \laps{r} \varphi + \int_{\R^n} g_2\, \varphi \quad \forall \varphi \quad \forall \varphi \in C_c^\infty(\tilde{\Omega}).
\end{equation}
such that \eqref{eq:red:vest} (with $v$ replaced by $w$), \eqref{eq:red:g12est}, and \eqref{eq:lapsbetag1est} hold.
\end{lemma}
\begin{proof}
Let $\tilde{\Omega} \subsubset \tilde{\Omega}_{1,1} \subsubset \Omega_1$ with $\Omega_2 \subsubset \tilde{\Omega}$, and let $\eta \in C_c^\infty(\Omega_1)$, $\eta \equiv 1$ in $\tilde{\Omega}_{1,1}$. 

Set $w := \eta u$, From Poincar\'e inequality and Sobolev embedding, we find that \eqref{eq:red:vest} holds.
Moreover, for any $\varphi \in C_c^\infty(\tilde{\Omega})$, we have that 
\[
\begin{split}
\langle \mathcal{L}^{s}_{\Omega}w, \varphi\rangle =&\int_{\Omega}\int_{\Omega}K(x,y) \frac{ \big (w(x)-w(y) \big )\, (\varphi(x)-\varphi(y))}{|x-y|^{n+2s}}\, dx\, dy \\
 =&\int_{\Omega}\int_{\Omega}K(x,y) \frac{ (u(x)-u(y))\, (\varphi(x)-\varphi(y))}{|x-y|^{n+2s}}\, dx\, dy \\
 &+\int_{\Omega\backslash \tilde{\Omega}_{1,1}}\int_{\tilde{\Omega}} K(x,y)\frac{ (1-\eta(y))u(y)\, \varphi(x)}{|x-y|^{n+2s}}\, dx\, dy.  \\
 \end{split}
\]
Now, to show \eqref{eq:red:1:pde} holds,  we set $g_1 := f_1$ and $g_2 := f_2 + \tilde{g}_2$ where 
\[
 \tilde{g}_2 (x) := \chi_{\tilde{\Omega}}(x) \int_{\Omega\backslash \tilde{\Omega}_{1,1}}K(x,y) \frac{ (1-\eta(y))u(y)\, \varphi(x)}{|x-y|^{n+2s}}\, dy. 
\]
%then by setting  we have shown that \eqref{eq:red:1:pde} holds.

To obtain the estimate \eqref{eq:red:g12est} we only need to estimate $\tilde{g}_2$. Observe that for any $y \in \supp(1-\eta)$ and $x \in \tilde{\Omega}$ we have $|x-y| \ageq c +|y|$. Consequently,
\[
 \|\tilde{g}_2\|_{L^\infty(\R^n)} \aleq \sup_{x \in \R^n} \int_{\Omega} |u(y)|\, (c+|x-y|)^{-n-2s}\, dy \aleq \|u\|_{L^2(\R^n)}.
\]
Since $\tilde{\Omega}$ is bounded and $\supp g_2 \subset \tilde{\Omega}$ we find $g_2 \in L^1 \cap L^\infty(\R^n)$, in particular
\[
 \|\tilde{g}_2\|_{L^q(\R^n)} \aleq \|u\|_{L^2(\R^n)}.
\]
This concludes the proof of \Cref{la:red:1}
\end{proof}

In the second step we increase the domain of integration of \eqref{eq:red:1:pde} from $\Omega$ to $\R^n$. 
\begin{lemma}\label{la:red:2}
Under the assumption of \Cref{th:reduction}, let $\Omega_2 \subsubset \tilde{\Omega} \subsubset \Omega_1$. Then there exist $w \in H^{s,2}(\R^n) \cap H^{t,p}(\R^n)$ with $\supp w \subset \Omega_1$, $w \equiv u$ in a neighborhood of $\Omega_2$, and $h_1,h_2 \in L^q(\R^n)$ such that
\begin{equation}\label{eq:red:2:pde}
 %\int_{\R^n}\int_{\R^n} \frac{K(x,y) \big (w(x)-w(y) \big )\, (\varphi(x)-\varphi(y))}{|x-y|^{n+2s}}\, dx\, dy 
 \langle \mathcal{L}^{s}_{\R^{n}}w, \varphi\rangle= \int_{\R^n} h_1\, \laps{r} \varphi + \int_{\R^n} h_2\, \varphi \quad \forall \varphi \quad \forall \varphi \in C_c^\infty(\tilde{\Omega}).
\end{equation}
such that \eqref{eq:red:vest} holds with $v$ replaced by $w$. Moreover, \eqref{eq:red:g12est} and \eqref{eq:lapsbetag1est} with $h_1$, $h_2$ instead of $g_1$, $g_2$, respectively.
\end{lemma}
\begin{proof}
Take $w,g_1,g_2$ from \Cref{la:red:1} and let $\varphi \in C_c^\infty(\tilde{\Omega})$.
\begin{equation}\label{eq:splitint}
 \begin{split}
  \langle \mathcal{L}^{s}_{\R^{n}}w, \varphi\rangle=&\int_{\R^n}\int_{\R^n}K(x,y) \frac{ \big (w(x)-w(y) \big )\, (\varphi(x)-\varphi(y))}{|x-y|^{n+2s}}\, dy\, dx\\
 =& \int_{\Omega}\int_{\Omega} K(x,y)\frac{\big (w(x)-w(y) \big )\, (\varphi(x)-\varphi(y))}{|x-y|^{n+2s}}  dy\, dx\\
 &+2\int_{\R^n \backslash \Omega}\int_{\Omega} K(x,y)\frac{ \big (w(x)-w(y) \big )\, (\varphi(x)-\varphi(y))}{|x-y|^{n+2s}} dy\, dx\\
 &+ \int_{\R^n \backslash \Omega}\int_{\R^n \backslash \Omega} K(x,y)\frac{ \big (w(x)-w(y) \big )\, (\varphi(x)-\varphi(y))}{|x-y|^{n+2s}} dy\, dx.
 \end{split}
\end{equation}
The third term in right hand side of \eqref{eq:splitint} vanishes because of $\supp w \subset \Omega_1 \subsubset \Omega$. 
%\[
 %\int_{\R^n \backslash \Omega}\int_{\R^n \backslash \Omega} \frac{K(x,y) \big (w(x)-w(y) \big )\, (\varphi(x)-\varphi(y))}{|x-y|^{n+2s}} dy\, dx = 0.
%\]
Moreover, since $\supp w \subset \Omega_1 \subsubset \Omega$ and $\supp \varphi \subset \tilde{\Omega} \subsubset \Omega$, the second term in \eqref{eq:splitint} becomes
\[
\begin{split}
&\int_{\R^n \backslash \Omega}\int_{\Omega}K(x,y) \frac{ \big (w(x)-w(y) \big )\, (\varphi(x)-\varphi(y))}{|x-y|^{n+2s}} dy\, dx\\
 &=\int_{\R^n\backslash \Omega}\int_{\tilde{\Omega}} K(x,y)\frac{ w(y)\, \varphi(y)}{|x-y|^{n+2s}}\, dy\, dx\, 
 =\int_{\R^n} \varphi(y)  w(y)\, \chi_{\tilde{\Omega}}(y) \int_{\R^n\backslash \Omega} \frac{K(x,y) }{|x-y|^{n+2s}}\, dx\, dy. 
 \end{split}
\]
Now the conclusion of the lemma is satisfied if we set $h_1 := g_1$ and $h_2 := g_2 + \tilde{h}_2$, where
\[
 \tilde{h}_2(y) := w(y)\, \chi_{\tilde{\Omega}}(y) \int_{\R^n\backslash \Omega} \frac{K(x,y) }{|x-y|^{n+2s}}\, dx. 
\]
To see this, first, we obtain  \eqref{eq:red:2:pde}  from \eqref{eq:red:1:pde}, \eqref{eq:splitint} and the above observations. In addition, estimates \eqref{eq:red:vest} and \eqref{eq:lapsbetag1est} hold from \Cref{la:red:2} since $w$ did not change and $h_1 = g_1$. In order to prove \eqref{eq:red:g12est} for $g_1,g_2$ replaced by $h_1,h_2$ we only need an estimate for $\tilde{h}_2$, which we obtain by arguing in similar fashion as in the proof of \Cref{la:red:1}. Since $\dist(\tilde{\Omega},\R^n \backslash \Omega) > 0$, for all points $x \in \R^n\backslash \Omega$ and $y \in \tilde{\Omega}$ we have $|x-y| \ageq 1+|x|$, and thus
\[
\begin{split}
 |\tilde{h}_2(y)| \aleq& |w(y)|\, \chi_{\tilde{\Omega}}(y) \int_{\R^n\backslash \Omega} \frac{\Lambda}{1+|x|^{n+2s}}\, dx
 \aeq  |w(y)|\, \chi_{\tilde{\Omega}}(y).
\end{split}
 \]
Thus,
\begin{equation}\label{eq:red:2:1}
 \|\tilde{h}_2\|_{L^{q}(\R^n)} \aleq \|w\|_{L^{q}(\tilde{\Omega}}.
\end{equation}
Finally, since $w \in H^{t,p}(\R^n)$ with compact support, in view of \eqref{eq:red:rpq} and Sobolev inequality, \Cref{pr:sob}, we have 
\begin{equation}\label{eq:red:2:2}
 \|w\|_{L^{q}(\tilde{\Omega})} \aleq \|\laps{t} w\|_{L^p(\R^n)}.
\end{equation}
We conclude that $\tilde{h}_2$ satisfies the estimates \eqref{eq:red:g12est} with $g_2$ replaced by $\tilde{h}_2$ in view of \eqref{eq:red:2:1}, \eqref{eq:red:2:2} and \eqref{eq:red:vest}. 
This concludes the proof of \Cref{la:red:2}.
\end{proof}

In the last step of the proof, we increase the domain of the testfunctions in \eqref{eq:red:2:pde} from $\tilde{\Omega}$ to $\R^n$. This is where the the main influence of the conditions on $p,q,r$ etc. come into play.
\begin{proof}[Proof of \Cref{th:reduction}]
Take $w,h_1,h_2$ from \Cref{la:red:2}, so that \eqref{eq:red:2:pde} holds.

Let $\eta \in C_c^\infty(\tilde{\Omega})$, $\eta \equiv 1$ in $\Omega_2$ and set $v := \eta w$. 
Since we know from from \Cref{la:red:2} that $w$ satisfies the estimates \eqref{eq:red:vest}  (with $v$ replaced by $w$), consequently in view of Poincar\'e and Sobolev inequality, so does $v$.

Fix any $\psi \in C_c^\infty(\R^n)$. Observe that
\[
\begin{split}
 \big (v(x)-v(y) \big)\, \big (\psi(x)-\psi(y) \big )  =&\big (\eta(x) w(x)-\eta(y) w(y) \big)\, \big (\psi(x)-\psi(y) \big ) \\
% =&\eta(x)\, \big (w(x)-w(y) \big )\, \big (\psi(x)-\psi(y) \big ) +\big (\eta(x)-\eta(y) \big )\, w(y)\, \big (\psi(x)-\psi(y) \big ) \\
 =&\big (w(x)-w(y) \big )\, \big (\eta(x)\psi(x)-\eta(y)\psi(y) \big ) \\
  &+\big (w(x)-w(y) \big )\, \big (\eta(x)-\eta(y) \big )\, \psi(y)) \\
 &+\big (\eta(x)-\eta(y) \big )\, w(y)\, \big (\psi(x)-\psi(y) \big ).\\
\end{split}
\]
We can now use the map $\eta \psi \in C_c^\infty(\tilde{\Omega})$ as a test function for \eqref{eq:red:2:pde}, and  obtain
\begin{equation}\label{eq:red:finalpde1}
\begin{split}
%\int_{\R^n}\int_{\R^n} \frac{K(x,y) (v(x)-v(y))\, \big (\psi(x)-\psi(y) \big )}{|x-y|^{n+2s}}\, dx\, dy \\
\langle\mathcal{L}^{s}_{\R^{n}} v, \varphi\rangle=& I + II + III
\end{split}
\end{equation}
where
\[
\begin{split}
I:=&\int_{\R^n} h_1\, \laps{r}(\eta \psi) \,dx+ \int_{\R^n} h_2\, \eta \psi \,dx\\
II:=&\int_{\R^n}\int_{\R^n} K(x,y)\frac{ \big (w(x)-w(y) \big )\, \big (\eta(x)-\eta(y) \big ) \psi(y))}{|x-y|^{n+2s}}\, dx\, dy \\
III:=&\int_{\R^n}\int_{\R^n}K(x,y) \frac{ \big (\eta(x)-\eta(y) \big )w(y)\, \big (\psi(x)-\psi(y) \big )}{|x-y|^{n+2s}}\, dx\, dy. 
\end{split}
\]
Using the commutator notation $[T,m](g) = T(mg)-mTg$, we can rewrite  the first term  of $I$ as 
\[
\begin{split}
 \int_{\R^n} h_1\, \laps{r}(\eta \psi) \,dx=& \int_{\R^n} h_1\, \eta \laps{r} \psi \,dx+ \int_{\R^n} h_1\, [\laps{r},\eta](\psi)\,dx\\
 =& \int_{\R^n} \eta h_1\, \laps{r}\psi\,dx - \int_{\R^n} [\laps{r},\eta] (h_1)\, \psi \,dx.
 \end{split}
\]
In the last step we used an integration by parts, we can justify by approximation as follows: since $r \in (0,1)$ we can use the Coifman–McIntosh–Meyer commutator estimate, e.g. in the formulation in \cite[Theorem 6.1.]{LS18}, and have
\[
  \|[\laps{r},\eta] (h_1)\|_{L^{q}(\R^n)} \aleq \|\eta\|_{\lip}\, \|h_1\|_{L^{q}(\R^n)}.
\]
Also, by Leibniz formula (or Sobolev embedding) for any $\beta > 0$,
\[
 \|\laps{\beta} (\eta h_1)\|_{L^{q}(\R^n)} \aleq \|h_1\|_{H^{\beta,q}(\R^n)},
\]
% \itad{Is there any assumption that $h_1 \in H^{\beta,q}(\R^n)$?}
whenever the right-hand side is finite.
So if we set 
\[
g_1^1 := \eta h_1 \quad \text{and} \quad g_2^1 := - [\laps{r},\eta] (h_1) \quad \text{and} \quad g_2^2 := \eta h_2
\]
we have shown that 
\[
I = \int_{\R^n} g_1^1 \laps{r}\psi + \int_{\R^n} (g_2^1 + g_2^2) \psi,
\]
and $g_1^1$, $g_2^1, g_2^2$ satisfy \eqref{eq:lapsbetag1est}, \eqref{eq:red:g12est} because $h_1, h_2$ satisfies those equations.

Similar to the argument in \eqref{eq:splitint}, by the support of $w$ and $\eta$, we have for the remaining terms of \eqref{eq:red:finalpde1}
\[
\begin{split}
II + III=&\int_{\R^n}\int_{\R^n} K(x,y)\frac{ \big (w(x)-w(y) \big )\, \big (\eta(x)-\eta(y) \big ) \psi(y)}{|x-y|^{n+2s}}\, dx\, dy \\
&+\int_{\R^n}\int_{\R^n} K(x,y)\frac{ \big (\eta(x)-\eta(y) \big )w(y)\, \big (\psi(x)-\psi(y) \big )}{|x-y|^{n+2s}}\, dx\, dy \\
=&\int_{\Omega_1}\int_{\Omega_1} K(x,y)\frac{ \big (w(x)-w(y) \big )\, \big (\eta(x)-\eta(y) \big ) \psi(y))}{|x-y|^{n+2s}}\, dx\, dy \\
&+\int_{\Omega_1}\int_{\Omega_1} K(x,y)\frac{ \big (\eta(x)-\eta(y) \big )w(y)\, \big (\psi(x)-\psi(y) \big )}{|x-y|^{n+2s}}\, dx\, dy \\
&+\int_{\tilde{\Omega}}\int_{\R^n\backslash \Omega_1} (K(x,y)+K(y,x)) \frac{w(y)\, \eta(y) \psi(y)}{|x-y|^{n+2s}}\, dx\, dy \\
&+\int_{\tilde{\Omega}}\int_{\R^n\backslash \Omega_1} (K(y, x)-K(x,y))\frac{\, \eta(y)\, w(y)\, \big (\psi(x)-\psi(y) \big )}{|x-y|^{n+2s}}\, dx\, dy.
\end{split}
\]
We set 
\[\begin{split}
 g_2^3(y) :=& \chi_{\tilde{\Omega}} w(y)\, \eta(y) \int_{\R^n\backslash \Omega_1} \frac{K(x,y)+K(y,x) }{|x-y|^{n+2s}}\, dx\\
 g_2^4(x) :=& -\chi_{\R^n \backslash \Omega_1}(x) \int_{\tilde{\Omega}}\frac{K(x,y) \eta(y)\, w(y)}{|x-y|^{n+2s}}\, dy
 +\chi_{\R^n \backslash \Omega_1}(x) \int_{\tilde{\Omega}}\frac{K(y,x) \eta(y)\, w(y)}{|x-y|^{n+2s}}\, dy\\
 g_2^5(y) :=& -2\chi_{\tilde{\Omega}}(y)\eta(y)\, w(y)\, \int_{\R^n\backslash \Omega_1} \frac{K(x,y) }{|x-y|^{n+2s}}\, dx 
 +2\chi_{\tilde{\Omega}}(y)\eta(y)\, w(y)\, \int_{\R^n\backslash \Omega_1} \frac{K(y,x) }{|x-y|^{n+2s}}\, dx
 \end{split}
\]Then
\[
 \begin{split}
II+III&=II_1 + III_1+\int_{\R^n}\psi(y) g_2^3(y)\, dy
+\int_{\R^n} \psi(x)\, g_2^4(x) dx 
+\int_{\R^n}\psi(y)\, g_2^5(y) dy 
\end{split}
\]
where 
\[
\begin{split}
 II_1 := &\int_{\Omega_1}\int_{\Omega_1} K(x,y)\frac{ \big (w(x)-w(y) \big )\, \big (\eta(x)-\eta(y) \big ) \psi(y))}{|x-y|^{n+2s}}\, dx\, dy,\,\text{and}  \\
 III_1 := &\int_{\Omega_1}\int_{\Omega_1} K(x,y)\frac{ \big (\eta(x)-\eta(y) \big )w(y)\, \big (\psi(x)-\psi(y) \big )}{|x-y|^{n+2s}}\, dx\, dy. 
 \end{split}
\]
As in the steps before,
\[
 \|g_2^3\|_{L^{q}(\R^n)} + \|g_2^5\|_{L^{q}(\R^n)}\aleq \|\laps{s} w\|_{L^p(\R^n)},
\]
As for $g_2^4$, by the distance of $x \in \R^n \backslash \Omega_1$ and $y \in \tilde{\Omega}$ we have $|x-y| \ageq c+|x|$, and thus
\[
 |g_2^4(x)|\aleq \frac{1}{1+|x|^{n+2s}}\, \|w\|_{L^1(\tilde{\Omega})} \aleq \frac{1}{1+|x|^{n+2s}}\, \|w\|_{L^p(\R^n)}.
\]
Since $\frac{1}{1+|x|^{n+2s}}$ is integrable to any power, we find that 
\[
 \|g_2^4\|_{L^{q}(\R^n)} \aleq \|w\|_{L^p(\R^n)} \aleq \|\laps{s} w\|_{L^p(\R^n)} + \|w\|_{L^2(\R^n)}.
\]
That is $g_2^3, g_2^4, g_2^5$ satisfy \eqref{eq:red:g12est} because $w$ satisfies \eqref{eq:red:vest}.

Next we estimate $II_1$.
\[
\int_{\Omega_1}\int_{\Omega_1}K(x,y) \frac{ \big (w(x)-w(y) \big )\, \big (\eta(x)-\eta(y) \big ) \psi(y)}{|x-y|^{n+2s}}\, dx\, dy
 =\int_{\R^n} \psi(y)\, g_2^6(y)\, dy 
\]
for 
\[
 g_2^6(y) := \chi_{\Omega_1}(y)\,\int_{\Omega_1} K(x,y)\frac{ \big (w(x)-w(y) \big )\, \big (\eta(x)-\eta(y) \big ) }{|x-y|^{n+2s}}\, dx.
\]
Now we have, see e.g. \cite[Proposition 6.6.]{S18Arma}, for any $\alpha < 1$,
\[
 |w(x)-w(y)| \aleq |x-y|^{\alpha}\, \brac{\mathcal{M} \laps{\alpha} w(x)+\mathcal{M} \laps{\alpha} w(y)},
\]
% \itad{What is the roughest $w$ can be for this inequality to hold? I suspect somehow $\laps{\alpha} w$ has to make sense. }
where $\mathcal{M}$ denotes the Hardy-Littlewood maximal function. Using this, the Lipschitz continuity of $\eta$ and the definition of the Riesz potential $\lapms{\alpha+1-2s}$, we find for any $\alpha \in (2s-1,1)$
\[
 |g_2^6| \aleq \chi_{\Omega_1} \brac{\mathcal{M} \laps{\alpha} w + \chi_{\Omega_1} \lapms{\alpha+1-2s} \brac{\chi_{\Omega_1} \mathcal{M} \laps{\alpha}w}}
\]
Observe that $t \geq s > 2s-1$. In particular in view of \eqref{eq:red:rpq} we can choose $\alpha \leq  t$ such that 
\[
 t-\frac{n}{p} \geq \alpha - \frac{n}{q},
\]
and from Sobolev embedding (observe that $\Omega_1$ is bounded) we obtain
\[
 \|g_2^6\|_{L^{q}(\R^n)} \aleq \|\laps{t} w\|_{L^p(\R^n)}.
\]
That is, we have shown that 
\[
 II_1 = \int_{\R^n} g_2^6 \psi,
\]
and $g_2^6$ satisfies \eqref{eq:red:g12est} because $w$ satisfies \eqref{eq:red:vest}.

The last term it remains to estimate is $III_1$. Set 
\[
 T[\psi]:=\int_{\Omega_1}\int_{\Omega_1} \frac{K(x,y) \big (\eta(x)-\eta(y) \big )w(y)\, \big (\psi(x)-\psi(y) \big )}{|x-y|^{n+2s}}\, dx\, dy \\
\]
Clearly $T$ is a linear operator acting on $\psi \in C_c^\infty(\R^n)$. Moreover, as above, for any $\alpha \in (2s-1,1)$,
\[
 |T[\psi]| \aleq \int_{\Omega_1}|w|\, \brac{\mathcal{M}\laps{\alpha}\psi+
 \lapms{\alpha+1-2s} \brac{\chi_{\Omega_1} \laps{\alpha} \psi}}\,dx. 
\]
Under the assumption \eqref{eq:red:sr} we can take $\alpha < r$, and have
\[
 \|\laps{\alpha} \psi\|_{L^{\frac{n{q'}}{n-(r-\alpha){q'}}}(\R^n)} \aleq \|\laps{r} \psi\|_{L^{{q'}}(\R^n)}.
\]
We repeat this argument for $T[\laps{\beta} \psi]$. If $2s-1+\beta < r$, we can choose $\alpha \in (2s-1,1)$, $\alpha > 0$, such that $\alpha + \beta < r$, (observe that since ${q'} \leq 2$, $r-\max\{2s-1,0\}-\beta < \frac{n}{{q'}}$ is certainly satisfied if $n \geq 2$, $r \in (0,1)$. If $n=1$, the condition $r \leq s$ implies $r-\max\{2s-1,0\} \leq \frac{1}{2}$ as well),
\[
 |T[\laps{\beta}\psi]|  \aleq \|w\|_{L^{\frac{nq}{n+(r-\alpha-\beta)q}}(\Omega)}\, \|\laps{\alpha+\beta} \psi\|_{L^{\frac{n{q'}}{n-(r-\alpha-\beta){q'}}}(\R^n)}  
\]
If for $\beta \geq 0$ \eqref{eq:red:beta} is satisfied, then
\[
 \|w\|_{L^{\frac{nq}{n+(r-\alpha-\beta)q}}(\Omega)} \aleq \|w\|_{H^{t,p}(\R^n)}
\]
In particular for $\beta =0$, in view of \eqref{eq:red:rpqt},
\begin{equation}\label{eq:local:Tpsi}
 |T[\psi]| \aleq \|w\|_{H^{t,p}(\R^n)}\, \|\laps{r} \psi\|_{L^{{q'}}(\R^n)}.
\end{equation}
and if \eqref{eq:red:beta} is satisfied we also have
\begin{equation}\label{eq:local:Tpsi2}
 |T[\laps{\beta} \psi]| \aleq \|w\|_{H^{t,p}(\R^n)}\, \|\laps{r} \psi\|_{L^{{q'}}(\R^n)}.
\end{equation}

\eqref{eq:local:Tpsi} implies that $T$ is a linear bounded operator on $\dot{H}^{r,{q'}}(\R^n)$. By the characterization of dual spaces, \Cref{pr:dualclass} we find $g_1^7 \in L^{q}(\R^n)$ such that 
\[
III_1 = T[\psi] = \int_{\R^n} g_1^7 \laps{r} \psi\,dx
\]
and
\[
 \|g_1^7\|_{L^{q}(\R^n)} \aleq \|w\|_{H^{t,p}(\R^n)}.
\]
If \eqref{eq:red:beta} is satisfied, \eqref{eq:local:Tpsi2} implies that $\psi \mapsto T[\laps{\beta}\psi]$ is a still linear bounded operator on $\dot{H}^{r,{q'}}(\R^n)$. From the characterization of dual spaces, \Cref{pr:dualclass} we thus find $g_{7,\beta} \in L^{q}(\R^n)$ such that 
\[
\int_{\R^n} g_1^7 \laps{r+\beta} \psi\, T[\laps{\beta}\psi] = \int_{\R^n} g_{7,\beta} \laps{r} \psi.
\]
and
\[
 \|g_{7,\beta}\|_{L^{q}(\R^n)} \aleq \|w\|_{H^{t,p}(\R^n)}.
\]
This implies that $\laps{\beta} g_1^7 = g_{7,\beta}$, and we have consequently the estimate needed for \eqref{eq:lapsbetag1est}
\[
 \|\laps{\beta} g_1^7\|_{L^{q}(\R^n)} \aleq \|w\|_{H^{t,p}(\R^n)}.
\]
That is, $g_1^7$ satisfies \eqref{eq:red:g12est} and \eqref{eq:lapsbetag1est}.

In view \eqref{eq:red:finalpde1} for 
\[
 g_1 := g_1^1 + g_1^7
\]
and 
\[
 g_2 := g_2^1+g_2^2 + g_2^3+g_2^4+g_2^5+g_2^6
\]
we have shown \eqref{eq:red:pdeg} holds, and $g_1,g_2$ satisfy the estimate \eqref{eq:red:g12est} and \eqref{eq:lapsbetag1est}. We have already observed that $w$ and $v$ satisfy the estimate \eqref{eq:red:vest}, so the proof of \Cref{th:reduction} is completed.
\end{proof}

\section{The Regularity theory: Proof of Theorem~\ref{th:main}}\label{s:proofmain}
In this section we prove the main result of the paper, Theorem~\ref{th:main}. The argument of the proof is based on iterating the following incremental  higher integrability result for a priori known smooth enough solution.  
\begin{theorem}\label{th:slighincrease}
Fix $s \in (0,1)$, $t \in [s,2s)$, $t <1$. For  given $\gamma\in (0, 1)$, $\lambda,\Lambda >0$, let $K\in \mathcal{K}(\gamma,\lambda, \Lambda)$. Suppose also that for any $2 \leq p, q < \infty$,  $u \in H^{s,2}(\R^n) \cap H^{t,p}(\R^n) \cap H^{t,2}(\R^n)$ with $\supp u \subset \Omega \subsubset \R^n$ is a solution to
\begin{equation}\label{eq:slight:1}
 %\int_{\R^n}\int_{\R^n} \frac{K(x,y) (u(x)-u(y))\, (\varphi(x)-\varphi(y))}{|x-y|^{n+2s}}\, dx\, dy 
\langle \mathcal{L}^{s}_{\R^{n}} u, \varphi\rangle= \int_{\R^{n}} f_1\, \laps{2s-t} \varphi 
 + \int_{\R^{n}} f_2\, \varphi\quad \forall \varphi \in C_c^\infty(\R^n).
\end{equation}
%Assume that $K: \R^n \times \R^n \to \R$ is a kernel which
%\begin{itemize}
% \item is measurable in $x$ and $y$
% \item bounded i.e. $|K(x,y)| \leq \Lambda$ for some $\Lambda > 0$ and a.e. $x,y \in \R^n$
% \item bounded below on the diagonal, i.e. $K(z,z) > \Lambda^{-1}$ for some $\Lambda > 0$ and a.e. $z \in \R^n$
% \item uniformly H\"older continuous in the following sense, there exists $\gamma \in (0,1]$ and $\Lambda > 0$
%\[
%|K(x,z)-K(y,z)| +  |K(z,x)-K(z,y)| \leq \Lambda\, |x-y|^\gamma \quad \text{a.e. $x,y,z \in \R^n$}
%\]
%\end{itemize}
Then there exists $\eps > 0$ such that if $r \in [p,p+\eps)$ and $f_1,f_2 \in L^r\cap L^p(\R^n),$ then
\[
 \|\laps{t} u\|_{L^r(\Omega)} \aleq \sum_{i=1}^2\|f_i\|_{L^r(\R^n)} + \|f_i\|_{L^p(\R^n)}+ \|\laps{t} u\|_{L^p(\R^n)}.
\]
In addition, if $\beta \in [0,\eps]$, $\laps{\beta} f_1 \in L^p(\R^n)$,  and $f_1,f_2 \in L^p(\R^n)$,  then $\laps{t+\beta} u \in L^p_{loc}(\R^n)$ and for any $\Omega \subsubset \R^n$ we have the estimate
\[
 \|\laps{t+\beta} u\|_{L^p(\Omega)} \aleq \|\laps{\beta}f_1\|_{L^p(\R^n)} + \|f_1\|_{L^p(\R^n)}+\|f_2\|_{L^p(\R^n)}+ \|\laps{t} u\|_{L^p(\R^n)}.
\]
Here, $\eps > 0$ is uniform in the following sense: $\eps$ depends only $\gamma$ and the number $\theta \in (0,1)$  which is such that
\[
 \theta < s,t,2s-t < 1-\theta,\,\,\text{and   \,\,$ 2 \leq p,q < \frac{1}{\theta}.$}
\]
%and
%\[
% 2 \leq p,q < \frac{1}{\theta}.
%\]
\end{theorem}
\begin{proof}
First we observe that in view of \Cref{th:reformulationcommie} and \eqref{eq:slight:1} we have for any $\varphi \in C_c^\infty(\R^n)$
{\small \begin{equation}\label{eq:Kzz22}
\begin{split}
 \int_{\R^n}K(z,z) \laps{t} u\, \laps{2s-t} \varphi \,dx=& \int_{\R^n} f_1\, \laps{2s-t} \varphi \,dx + \int_{\R^n} f_2 \varphi \,dx- D_{s,t}(u,\varphi)\\
 =& \int_{\R^{n}} \laps{\beta} f_1\, \laps{2s-t-\beta} \varphi\,dx  + \int_{\R^n} f_2 \varphi\,dx - D_{s,t}(u,\varphi),
 \end{split}
\end{equation}}
where $D_{s,t}(u,\varphi)$ is as defined in \eqref{defn-Lcommutator} and where we have taken without loss of generality that the constant $\Gamma=1$ in  \Cref{th:reformulationcommie}. 
% =& \int \laps{\beta} f\, \laps{2s-t-\beta} \varphi  - D_{s,t}(u,\varphi), 
Now we observe that the map $T$ defined as
\[
 T[\varphi]:=D_{s,t}(u,\varphi)
\]
is linear in $\varphi \in C_c^\infty(\R^n)$.  Moreover, from \Cref{th:reformulationcommie} and Sobolev embedding, \Cref{pr:sob}, we have the estimate for any $\beta \in [0,\eps]$
\[
\begin{split}
T[\varphi] \aleq&  \int_{\R^n} |\laps{t} u|(x)\,  \lapms{\sigma-\eps}|\laps{2s-t-\eps} \varphi|(x)\, dx\\
\aleq& \|\laps{t} u\|_{L^{p}(\R^n)}\, \|\lapms{\sigma-\eps}|\laps{2s-t-\eps} \varphi|\|_{L^{p'}(\R^n)}\\
\aleq& \|\laps{t} u\|_{L^{p}(\R^n)}\, \|\laps{2s-t-\eps} \varphi\|_{L^{\frac{n p'}{n+(\sigma-\eps)p'}}(\R^n)}\\
\aleq& \|\laps{t} u\|_{L^{p}(\R^n)}\, \|\laps{2s-t-\beta} \varphi\|_{L^{\frac{n p'}{n+(\sigma +\beta-\eps)p'}}(\R^n)}.
\end{split}
\]
Here $\sigma$ and $\eps$ can be chosen to depend only on $\theta$, and since $p < \frac{1}{\theta}$ we can make that choice so that $\frac{n p'}{n+(\sigma-2\eps)p'} > 1$ and Sobolev embedding is applicable with a uniform constant.

That is, $T$ belongs to $\brac{\dot{H}^{2s-t-\beta,\frac{np'}{n+(\sigma +\beta-\eps)p'}}(\R^n)}^\ast$ for any $\beta \in [0,\eps]$. By classification of the dual spaces, \Cref{pr:dualclass}, and since $\brac{\frac{np'}{n+(\sigma +\beta-\eps)p'}}'= \frac{np}{n-(\sigma +\beta-\eps)p} $ we find $g_\beta \in L^{\frac{np}{n-(\sigma +\beta-\eps)p}}(\R^n)$
\[
 \|g_\beta \|_{L^{\frac{np}{n-(\sigma +\beta-\eps)p}}(\R^n)} \aleq \|\laps{t} u\|_{L^{p}(\R^n)},
\]
and
\[
 T[\varphi] = \int_{\R^n} g_\beta\, \laps{2s-t-\beta} \varphi.
\]
That is, \eqref{eq:Kzz22} becomes for any $\beta \in [0,\eps]$ 
\[
 \int_{\R^n}K(z,z) \laps{t} u\, \laps{2s-t} \varphi = \int_{\R^n} \brac{\laps{\beta} f_1+g_\beta}\, \laps{2s-t-\beta} \varphi\,dx +\int_{\R^n} f_2 \varphi\,dx
\]
for all $ \varphi \in C_c^\infty(\R^n).$

For $\beta= 0$ we obtain from \Cref{th:regularityKlapls:1} that for any $\Omega \subsubset \R^n$, $r \in \left [p,\frac{np}{n-(\sigma -\eps)p}\right ]$, we have
\[
 \|\laps{t} u\|_{L^r(\Omega)} \aleq \|f\|_{L^r(\R^n)} + \|f\|_{L^p(\R^n)} + \|\laps{t} u\|_{L^p(\R^n)}.
\]
For $\beta \in [0,\eps]$ from \Cref{pr:regularityKlapls:2} for any $r \in [p,\frac{np}{n-(\sigma +\beta-\eps)p}]$,
\[
 \|\laps{t+\beta} u\|_{L^p(\Omega)} \aleq \|\laps{\beta} f\|_{L^{r}(\R^n)} \|\laps{t} u\|_{L^r(\R^n)}.
\]

\end{proof}

Iterating \Cref{th:slighincrease} and \Cref{th:reduction} leads to the proof of \Cref{th:main}, namely
\begin{theorem}\label{th:main2}
Fix $s \in (0,1)$, $t \in [s,2s)$, $t <1$. For  given $\gamma\in (0, 1)$, $\lambda,\Lambda >0$, let $K\in \mathcal{K}(\gamma,\lambda, \Lambda)$.
%Assume $K$ is a H\"older continuous, uniformly elliptic kernel, i.e. assume that $\Lambda^{-1} \leq K(z,z) \leq \Lambda$ and for some $\gamma  \in (0,1]$,
%\[
%\sup_{c} |K(a,c)-K(b,c)| + \sup_{c} |K(c,a)-K(c,b)| \leq \Lambda\, |a-b|^\gamma.
%\]
Let $\Omega' \subsubset \Omega'' \subsubset \Omega \subseteq \R^n$ be two open sets.  
%and take $t \in [s,2s)$.
Assume that $u \in W^{s,2}(\Omega)$ satisfies  the equation
\begin{equation}\label{eq:m:pde}
 %\int_{\Omega}\int_{\Omega} \frac{K(x,y) (u(x)-u(y))\, (\varphi(x)-\varphi(y))}{|x-y|^{n+2s}}\, dx\, dy 
 \langle\mathcal{L}^{s}_{\Omega}u, \varphi\rangle= \int_{\R^n} f_1 \laps{2s-t} \varphi + \int f_2 \varphi \quad \forall \varphi \in C_c^\infty(\Omega'').
\end{equation}
If $f_1,f_2 \in L^q(\Omega)\cap L^2(\R^n)$, $q \in [2,\infty)$, then $\laps{t} u \in L^q(\Omega')$ and we have 
\[
 \|\laps{t} u\|_{L^q(\Omega')} \leq C(\Omega,\Omega',\Omega'',s,t,p,q) \brac{\|u\|_{H^{s,2}(\Omega)} + \sum_{i=1}^2 \|f_i\|_{L^q(\Omega)}+\|f_i\|_{L^2(\R^n)}}.
\]
\end{theorem}
\begin{proof}
Fix $\theta \in (0,1)$ such that
\begin{equation}\label{eq:maint:thetachoice}
 t < 1-10\theta,\ 10\theta < s < 1-10\theta,\ 10\theta < 2s-t < 1-10\theta, 2 \leq q < \frac{1}{10\theta}.
\end{equation}

We also fix $\eps = \eps(\theta,\gamma)$ from \Cref{th:slighincrease}, and w.l.o.g. $\eps < \frac{1}{10} \frac{\theta}{n}$.

\underline{Step 0: Rewriting the equation}
Take some cutoff function $\eta \in C_c^\infty(\Omega)$ with $\eta \equiv 1$ in a neighbourhood of $\Omega''$
\[
\begin{split}
 \int_{\R^{n}} f_1 \laps{2s-t} \varphi =& \int_{\R^{n}} \eta f_1 \laps{2s-t} \varphi \,dx+ \int_{\R^{n}} (1-\eta) f_1 \laps{2s-t} \varphi \,dx\\
  =&\int_{\R^{n}} \eta f_1 \laps{2s-t} \varphi \,dx+ \int_{\R^n} \chi_{\Omega''} \laps{2s-t}\brac{(1-\eta) f_1} \varphi \,dx.
\end{split}
 \]
Now observe that by the disjoint support of $\chi_{\Omega''}$ and $1-\eta$ we have
\[
 \|\chi_{\Omega''} \laps{2s-t}\brac{(1-\eta) f_1}\|_{L^\infty} \aleq \|f_1\|_{L^2(\R^n)}
\]
For $\sigma \in [s,t]$ we set
\[
 \tilde{f}_{1,\sigma} := \lapms{t-\sigma}(\eta f_1) 
\]
and
\[
 \tilde{f}_2 := \chi_{\Omega''} f_2 + \chi_{\Omega''} \laps{2s-t}\brac{(1-\eta) f_1}
\]
then we have for all $\varphi \in C_c^\infty(\Omega''),$
\[
 \int_{\Omega}\int_{\Omega} K(x,y)\frac{ (u(x)-u(y))\, (\varphi(x)-\varphi(y))}{|x-y|^{n+2s}}\, dx\, dy = \int_{\R^n} \tilde{f}_{1,\sigma} \laps{2s-\sigma} \varphi \,dx+ \int \tilde{f}_2 \varphi \,dx.
\]
Moreover $\tilde{f}_2 \in L^q(\R^n) \cap L^2(\R^n)$ and since $t-s \leq 1-s < 1-\theta$ we have by Sobolev embedding, \Cref{pr:sob},
\[
 \|\lapms{t-s}(\eta f_1) \|_{L^q(\R^n)} \leq C(\theta) \|\eta f_1\|_{L^{\frac{nq}{n+(t-s) q}}} \aleq \|\eta f_1\|_{L^q(\R^n)}
\]
If $n\geq 2$ we also have
\begin{equation}\label{eq:22:L2est}
 \|\lapms{t-s}(\eta f_1) \|_{L^2(\R^n)} \leq C(\theta) \|\eta f_1\|_{L^{\frac{2n}{n+(t-s) 2}}} \aleq \|\eta f_1\|_{L^2(\R^n)}
\end{equation}
so that for $n\geq 2$ we have found $\tilde{f}_{1,\sigma}, \tilde{f}_2 \in L^q\cap L^2(\R^n)$ such that \eqref{eq:m:pde} holds for $t$ replaced with $\sigma$ and $f_1$, $f_2$ replaced with $\tilde{f}_{1,\sigma}, \tilde{f}_2$.

If $n=1$ we need a slight adaptation to have \eqref{eq:22:L2est} (if $t$ is close to one and $s$ is close to zero): Let $\eta_2 \in C_c^\infty(\R^n)$ with $\eta_2 \equiv 1$ in a neighborhood of $\Omega''$. Then we set 
\[
 \tilde{\tilde{f}}_{1,\sigma} := \eta_2 \lapms{t-\sigma}(\eta f_1) 
\]
and 
\[
 \tilde{\tilde{f}}_{2,\sigma} := \tilde{f}_{2} + \chi_{\Omega''} \brac{\laps{2s-t} \brac{(1-\eta_2) \lapms{t-\sigma}(\eta f_1) }}.
\]
By the disjoint support we then get the same estimates as before.

In conclusion, for any $\sigma \in [s,t]$ we have $f_{1,\sigma},f_{2,\sigma} \in L^2\cap L^q(\R^n)$, such that for any $ \varphi \in C_c^\infty(\Omega'')$, 
{\small \begin{equation}\label{eq:m:pdesigma}
 \int_{\Omega}\int_{\Omega} K(x,y)\frac{ (u(x)-u(y))\, (\varphi(x)-\varphi(y))}{|x-y|^{n+2s}}\, dx\, dy = \int_{\R^n} f_{1,\sigma} \laps{2s-\sigma} \varphi\,dx + \int f_{2,\sigma} \varphi\,dx .
\end{equation} }

Let $L \in \N$ a number which we shall define later, and choose nested open sets 
\begin{equation}\label{eq:omeganest} \Omega' := \Omega_{2L} \subsubset \ldots \subsubset \Omega_1 \subsubset \Omega'' \subsubset \Omega.\end{equation} 
\underline{Step 1: First improvement}
\begin{equation}\label{eq:m:q1}
 \frac{1}{q_1} := \max\left \{\frac{1}{2} -\frac{\theta}{n},\theta \right \}
\end{equation}
then \eqref{eq:red:rpq} and \eqref{eq:red:rpqt} are satisfied in view of \eqref{eq:maint:thetachoice}.
\begin{equation}\label{eq:m:b1}
 \beta_1 := \min\{\frac{1}{2}\eps,s-t\}
\end{equation}
\begin{equation}\label{eq:m:p1}
 p_1 := \min\{2+\frac{1}{2}\eps,q_1\}
\end{equation}
We claim that
\begin{equation}\label{eq:m:s1}
\begin{split}
 \|\laps{s+\beta_1} u\|_{L^{p_1}(\Omega_2)}  &+ \|\laps{s} u\|_{L^{p_1}(\Omega_2)} \\
 &\aleq \sum_{i=1}^2 \brac{\|f_i\|_{L^2(\R^n)} + \|f_i\|_{L^q(\R^n)}} + \|u\|_{H^{s,2}(\R^n)}.
\end{split}
\end{equation}

We apply \Cref{th:reduction} for $\tilde{s} = \tilde{t} = \tilde{r} = s$, $\tilde{p} =2$ and the equation to \eqref{eq:m:pdesigma} with $\sigma = s$. Then \eqref{eq:red:sr} is satisfied since $s < 1$. 
We also choose $\tilde{q} := q_1 \in (2,q]$ then \eqref{eq:red:rpq} and \eqref{eq:red:rpqt} are satisfied in view of \eqref{eq:maint:thetachoice}.

Observe that $f_{1,\sigma},f_{2,\sigma} \in L^q \cap L^2(\R^n) \subset L^{q_1}(\R^n)$, so from \Cref{th:reduction} we obtain $v_1 \in H^{s,2}(\R^n)$, $\supp v_1 \subsubset \Omega''$
\[
 v_1 \equiv u \text{\quad in a neighborhood of $\Omega_1$}
\]
and for any $ \varphi \in C_c^\infty(\R^n)
$
\begin{equation}\label{eq:m:pdev1}
\begin{split}
 \int_{\R^n}\int_{\R^n}K(x,y) &\frac{(v_1(x)-v_1(y))\, (\varphi(x)-\varphi(y))}{|x-y|^{n+2s}}\, dx\, dy \\
 &= \int_{\R^n} g_{1,s} \laps{s} \varphi \,dx+ \int_{\R^n} g_{2,s} \varphi \,dx,
 \end{split}
 \end{equation}
for some $g_1, g_2 \in L^{q_1}(\R^n)$ with the estimate
\[
 \|g_{1,s}\|_{L^{q_1}(\R^n)} + \|g_{2,s}\|_{L^{q_1}(\R^n)} \aleq \sum_{i=1}^2\brac{\|f_i\|_{L^2(\R^n)} + \|f_i\|_{L^q(\R^n)}} + \|u\|_{H^{s,2}(\R^n)}
\]
and
\[
 \|\laps{s} v_1\|_{L^2(\R^n)} \aleq \|u\|_{H^{s,2}(\R^n)}
\]
and in view of \eqref{eq:lapsbetag1est}, for any $0 \leq \alpha \leq \min\{\theta,t-s\}$ we have (for $\tilde{\beta} := \alpha$) that \eqref{eq:red:beta} is satisfied 
\[
 \|\laps{\alpha} g_{1,\sigma}\|_{L^2(\R^n)} \aleq \|f_1\|_{L^2(\R^n)}+\|u\|_{H^{s,2}(\R^n)}.
\]
In view of \Cref{th:slighincrease} (applied to $\tilde{t} := s$ and the equation \eqref{eq:m:pdev1})  we have the estimate
\[
\begin{split}
 \|\laps{s} v_1\|_{L^{p_1}(\Omega'')} \aleq& \sum_{i=1}^2 \brac{\|g_i\|_{L^{p_1}(\R^n)} + \|g_i\|_{L^{2}(\R^n)}} + \|\laps{s} v_1\|_{L^2(\R^n)}\\
 \aleq& \sum_{i=1}^2 \brac{\|f_i\|_{L^2(\R^n)} + \|f_i\|_{L^q(\R^n)}} + \|u\|_{H^{s,2}(\R^n)}.
 \end{split}
\]
Moreover, since we applied \Cref{th:slighincrease} to the equation \eqref{eq:m:pdev1}, we have 
\[
\begin{split}
 \|\laps{s+\beta_1} v_1\|_{L^{2}(\Omega'')} \aleq&\sum_{i=1}^2 \brac{\|f_i\|_{L^2(\R^n)} + \|f_i\|_{L^q(\R^n)}} + \|u\|_{H^{s,2}(\R^n)}.
 \end{split}
\]
Since $u \equiv v_1$ in a neighborhood of $\Omega_1$, by \Cref{la:cutoffarg}, we find that this implies
\begin{equation}\label{eq:s1:lapsulp1}
 \|\laps{s} u\|_{L^{p_1}(\Omega_1)} \aleq \sum_{i=1}^2 \brac{\|f_i\|_{L^2(\R^n)} + \|f_i\|_{L^q(\R^n)}} + \|u\|_{H^{s,2}(\R^n)}.
\end{equation}
and
\begin{equation}\label{eq:s1:lapsbetaul2}
\begin{split}
 \|\laps{s+\beta_1} u\|_{L^{2}(\Omega_1)} \aleq&\sum_{i=1}^2 \brac{\|f_i\|_{L^2(\R^n)} + \|f_i\|_{L^q(\R^n)}} + \|u\|_{H^{s,2}(\R^n)}.
 \end{split}
\end{equation}

In order to obtain \eqref{eq:m:s1} we need to have an $L^{p_1}$-estimate in \eqref{eq:s1:lapsbetaul2}. For this we repeat this argument for the equation \eqref{eq:m:pdesigma} with $\sigma = s+\beta_1$ (this is only necessary if $t > s$, otherwise $\beta_1 = 0$ and we are done with \eqref{eq:s1:lapsulp1}).

We apply \Cref{th:reduction} for $\tilde{s} = s$, $\tilde{t} = s+\beta_1$ and $\tilde{r} = s-\beta_1$, $\tilde{p} =2$, $\tilde{q} := q_1$ to \eqref{eq:m:pdesigma} with $\sigma = s+\beta_1$.  Again \eqref{eq:red:sr} is satisfied, since $s+\beta_1 < s+t-s = t < 1$. \eqref{eq:red:rpq}, and \eqref{eq:red:rpqt} are satisfied in view of \eqref{eq:maint:thetachoice}. Then \Cref{th:reduction} implies the existence of $v_2 \in H^{s+\beta_1,2}(\R^n)$, $\supp v_2 \subsubset \Omega''$
\[
 v_2 \equiv u \text{\quad in a neighborhood of $\Omega_2$}
\]
and for all $\varphi \in C_c^\infty(\R^n)$
\begin{equation}\label{eq:m:pdev2}
 \begin{split}
 \int_{\R^n}\int_{\R^n}K(x,y)&\frac{ (v_2(x)-v_2(y))\, (\varphi(x)-\varphi(y))}{|x-y|^{n+2s}}\, dx\, dy \\
 &= \int_{\R^n} g_{1,s+\beta_1} \laps{s+\beta_1} \varphi \,dx + \int_{\R^n} g_{2,s+\beta_1} \varphi \,dx,
 \end{split}
\end{equation} 
for some $g_{1,s+\beta_1}, g_{2,s+\beta_1} \in L^{q_1}(\R^n)$ with the estimate
\[
 \|g_{1,s+\beta_1}\|_{L^{q_1}(\R^n)} + \|g_{2,s+\beta_1}\|_{L^{q_1}(\R^n)} \aleq \sum_{i=1}^2\brac{\|f_i\|_{L^2(\R^n)} + \|f_i\|_{L^q(\R^n)}} + \|u\|_{H^{s,2}(\R^n)}
\]
and (with the additional help of \eqref{eq:s1:lapsbetaul2}),
\[
 \|\laps{s+\beta} v_2\|_{L^2(\R^n)} \aleq \sum_{i=1}^2 \brac{\|f_i\|_{L^2(\R^n)} + \|f_i\|_{L^q(\R^n)}} + \|u\|_{H^{s,2}(\R^n)}.
\]
Applying \Cref{th:slighincrease} (for $\tilde{t} := s+\beta_1$ and the equation \eqref{eq:m:pdev2}, observe that $\eps$ does not change) we have
\[
 \|\laps{s+\beta} v_2\|_{L^{p_1}(\Omega_2)} \aleq \sum_{i=1}^2 \brac{\|f_i\|_{L^2(\R^n)} + \|f_i\|_{L^q(\R^n)}} + \|u\|_{H^{s,2}(\R^n)}.
\]
By \Cref{la:cutoffarg}, since $u \equiv v_2$ in a neighborhood of $\Omega_2$ we find 
\begin{equation}\label{eq:s1:lapsbetaul2p1}
\begin{split}
 \|\laps{s+\beta_1} u\|_{L^{p_1}(\Omega_2)} \aleq&\sum_{i=1}^2 \brac{\|f_i\|_{L^2(\R^n)} + \|f_i\|_{L^q(\R^n)}} + \|u\|_{H^{s,2}(\R^n)}.
 \end{split}
\end{equation}
Together, \eqref{eq:s1:lapsbetaul2p1} and \eqref{eq:s1:lapsulp1} imply \eqref{eq:m:s1}.

\underline{Step 2: Iteration}
We define for $k \in \N$, 
\begin{equation}\label{eq:m:qkp1}
 \frac{1}{q_{k+1}} := \max\left \{\frac{1}{p_k} -\frac{\theta}{n},\frac{1}{q}\right \},
\end{equation}
\begin{equation}\label{eq:m:pkp1}
 p_{k+1} := \min\left \{p_k+\frac{1}{2}\eps,q_{k+1} \right \},
\end{equation}
\begin{equation}\label{eq:m:bkp1}
 \beta_{k+1} := \beta_{k}+\min\left \{\frac{1}{2}\eps,s-t-\beta_k \right \}.
\end{equation}
starting from $q_1,p_1,\beta_1$ as in \eqref{eq:m:q1}, \eqref{eq:m:p1}, \eqref{eq:m:b1}, respectively.

Our goal is to show that for any $k \in \N$, 
\begin{equation}\label{eq:m:sk}
\begin{split}
 \|\laps{s+\beta_k} u\|_{L^{p_k}(\Omega_{2k})}  &+\|\laps{s} u\|_{L^{p_k}(\Omega_{2k})} \\
 &\aleq \sum_{i=1}^2 \brac{\|f_i\|_{L^2(\R^n)} + \|f_i\|_{L^q(\R^n)}} + \|u\|_{H^{s,2}(\R^n)}.
 \end{split}
\end{equation}
We prove this by induction  induction.  
We already have shown \eqref{eq:m:sk} to hold for $k=1$, \eqref{eq:m:s1}.

So assume as induction hypothesis that for some $k \in \N$ \eqref{eq:m:sk} holds. We need to show 
\begin{equation}\label{eq:m:skp1}
 \|\laps{s} u\|_{L^{p_{k+1}}(\Omega_{2k+2})} \aleq \sum_{i=1}^2 \brac{\|f_i\|_{L^2(\R^n)} + \|f_i\|_{L^q(\R^n)}} + \|u\|_{H^{s,2}(\R^n)}.
\end{equation}
and
\begin{equation}\label{eq:m:skbetap1}
\|\laps{s+\beta_{k+1}} u\|_{L^{p_{k+1}}(\Omega_{2k+2})} \aleq \sum_{i=1}^2 \brac{\|f_i\|_{L^2(\R^n)} + \|f_i\|_{L^q(\R^n)}} + \|u\|_{H^{s,2}(\R^n)}.
\end{equation}

\underline{First we treat \eqref{eq:m:skp1}}. If $q_{k} = q$ there is nothing to show, because then $p_k = p_{k+1} = q$. If not, we apply \Cref{th:reduction} for $\tilde{s} = \tilde{t} = \tilde{r} = s$, $\tilde{p} =p_{k}$, $\tilde{q} := q_{k}$ to \eqref{eq:m:pdesigma} with $\sigma = s$. Again \eqref{eq:red:sr} is satisfied since $s < 1$. \eqref{eq:red:rpq} and \eqref{eq:red:rpqt} are satisfied in view of \eqref{eq:maint:thetachoice} and the fact that since $\eps < \frac{1}{10} \frac{\theta}{n}$ we have that $\left |\frac{1}{p_k} - \frac{1}{p_{k-1}} \right| \leq \frac{\theta}{n}$. 

Then \Cref{th:reduction} implies the existence of $v_{1} \in H^{s,p_{k}}(\R^n)$, $\supp v_1 \subsubset \Omega''$
\[
 v_1 \equiv u \text{\quad in a neighborhood of $\Omega_{2k+1}$}
\]
and for all $\varphi \in C_c^\infty(\R^n)$, 
\begin{equation}\label{eq:m:pdev2k}
\begin{split}
 \int_{\R^n}\int_{\R^n}K(x,y)&\frac{ (v_1(x)-v_1(y))\, (\varphi(x)-\varphi(y))}{|x-y|^{n+2s}}\, dx\, dy \\
 &= \int_{\R^n} g_{1,s} \laps{s} \varphi \,dx+ \int_{\R^n} g_{2,s} \varphi \,dx,
 \end{split}
\end{equation}
for some $g_{1,s}, g_{2,s} \in L^{q_k}(\R^n)$ with the estimate
\[
 \|g_{1,s}\|_{L^{q_k}(\R^n)} + \|g_{2,s}\|_{L^{q_k}(\R^n)} \aleq \sum_{i=1}^2\brac{\|f_i\|_{L^2(\R^n)} + \|f_i\|_{L^q(\R^n)}} + \|u\|_{H^{s,2}(\R^n)}.
\]
and additionally (using the induction hypothesis \eqref{eq:m:sk})
\[
 \|\laps{s} v_1\|_{L^{p_k}(\R^n)} \aleq \sum_{i=1}^2 \brac{\|f_i\|_{L^2(\R^n)} + \|f_i\|_{L^q(\R^n)}} + \|u\|_{H^{s,2}(\R^n)}.
\]
Applying \Cref{th:slighincrease} for $\tilde{t} := s$ and the equation \eqref{eq:m:pdev2k} we obtain
\[
\begin{split}
 \|\laps{s} v_1\|_{L^{p_{k+1}}(\Omega'')} \aleq& \sum_{i=1}^2 \brac{\|g_i\|_{L^{p_{k+1}}(\R^n)} + \|g_i\|_{L^{2}(\R^n)}} + \|\laps{s} v_1\|_{L^{p_k}(\R^n)}\\
 \aleq& \sum_{i=1}^2 \brac{\|f_i\|_{L^2(\R^n)} + \|f_i\|_{L^q(\R^n)}} + \|u\|_{H^{s,2}(\R^n)}.
 \end{split}
\]
Since $u \equiv v_1$ in a neighborhood of $\Omega_{2k+1}$ this implies \eqref{eq:m:skp1}.

\underline{Now we treat \eqref{eq:m:skbetap1}}.
We apply \Cref{th:reduction} to $\tilde{s} = s$, $\tilde{t} = s+\beta_{k}$, $\tilde{p} = p_{k}$, $\tilde{q} = q_k$, $\tilde{r} = s-\beta_k$ and to the equation \eqref{eq:m:pdesigma} with $\sigma = s+\beta_{k}$. \eqref{eq:red:sr} is satisfied since $s+\beta_k \leq t < 1$. As before, \eqref{eq:red:rpq}, \eqref{eq:red:rpqt} are satisfied in view of the choice of $\theta$, $q_k$, $p_k$, $\beta_{k+1}$. Since we have by assumption \eqref{eq:m:sk}, we find $v_2 \in H^{s,2}\cap H^{s+\beta_k}(\R^n)$, $v_2 \equiv u$ in a neighborhood of $\Omega_{2k+1}$, $g_1,g_2 \in L^{q_k}(\R^n)$ such that for all $ \varphi \in C_c^\infty(\R^n)
$
\[
 \int_{\R^n}\int_{\R^n}K(x,y)\frac{ (v_2(x)-v_2(y))\, (\varphi(x)-\varphi(y))}{|x-y|^{n+2s}}\, dx\, dy = \int_{\R^n} g_{1} \laps{s-\beta_k} \varphi\,dx + \int_{\R^n} g_{2} \varphi\,dx. \]
We apply \Cref{th:slighincrease} for $\tilde{t} = s+\beta_{k}$ to this equation, and find that 
\[
 \|\laps{s+\beta_{k+1}} v_2\|_{L^{p_k}(\Omega_{2k})} \aleq \sum_{i=1}^2 \brac{\|f_i\|_{L^2(\R^n)} + \|f_i\|_{L^q(\R^n)}} + \|u\|_{H^{s,2}(\R^n)}.
\]
Since $v_2 \equiv u$ in a neighborhood of $\Omega_{2k+1}$ we conclude that 
\begin{equation}\label{eq:m:est:2454}
 \|\laps{s+\beta_{k+1}} u\|_{L^{p_k}(\Omega_{2k+1})} \aleq \sum_{i=1}^2 \brac{\|f_i\|_{L^2(\R^n)} + \|f_i\|_{L^q(\R^n)}} + \|u\|_{H^{s,2}(\R^n)}.
\end{equation}
If $p_{k+1} = p_k$ we have \eqref{eq:m:skbetap1}. Otherwise, we need to apply this chain of arguments one more time:
This time, we apply \Cref{th:reduction} to $\tilde{s} = s$, $\tilde{t} = s+\beta_{k+1}$, $\tilde{p} = p_{k}$, $\tilde{q} = q_k$, $\tilde{r} = s-\beta_k$ and to the equation \eqref{eq:m:pdesigma} with $\sigma = s+\beta_{k+1}$. Again, \eqref{eq:red:sr} is satisfied since $s+\beta_{k+1} \leq t < 1$, and \eqref{eq:red:rpqt}, \eqref{eq:red:rpq} are satisfied in view of the choice of $\theta$, $q_k$, $p_k$, $\beta_{k+1}$. Since we have \eqref{eq:m:est:2454}, we obtain from \Cref{th:reduction}  $v_3 \in H^{s,2}\cap H^{s+\beta_{k+1}}(\R^n)$, $v_2 \equiv u$ in a neighborhood of $\Omega_{2k+2}$, $g_1,g_2 \in L^{q_k}(\R^n)$ such that for all $\varphi \in C_c^\infty(\R^n)$
\[
 \int_{\R^n}\int_{\R^n}K(x,y)\frac{ (v_2(x)-v_2(y))\, (\varphi(x)-\varphi(y))}{|x-y|^{n+2s}}\, dx\, dy = \int_{\R^n} g_{1} \laps{s-\beta_{k+1}} \varphi\,dx + \int_{\R^n} g_{2} \varphi \,dx.
\]
We apply \Cref{th:slighincrease} for $\tilde{t} = s+\beta_{k+1}$ to this equation, and find that 
\[
 \|\laps{s+\beta_{k+1}} v_3\|_{L^{p_{k+1}}(\Omega_{2k+1})} \aleq \sum_{i=1}^2 \brac{\|f_i\|_{L^2(\R^n)} + \|f_i\|_{L^q(\R^n)}} + \|u\|_{H^{s,2}(\R^n)}.
\]
Since $v_3 \equiv u$ in a neighborhood of $\Omega_{2k+2}$, we finally conclude \eqref{eq:m:skbetap1}.

\underline{Conclusion:}
From the definition of $p_{k+1}$, $q_{k+1}$, $\beta_{k+1}$ as in \eqref{eq:m:pkp1}, \eqref{eq:m:qkp1}, \eqref{eq:m:bkp1} starting from $p_1,q_1,\beta_1$ as in \eqref{eq:m:p1}, \eqref{eq:m:q1}, \eqref{eq:m:b1} we see that there is a large number (depending on $\eps$ and $\theta$, $s$, $t$, and $q$ -- all of which are fixed numbers in this proof) there is a finite number $L \in \N$ such that $p_{L} = q_{L} = q$, $\beta_{L} = t-s$. Thus, from we have \eqref{eq:m:sk} we obtain
\[
 \|\laps{t} u\|_{L^{p}(\Omega_{2L})}  +\|\laps{s} u\|_{L^{p}(\Omega_{2L})} \aleq \sum_{i=1}^2 \brac{\|f_i\|_{L^2(\R^n)} + \|f_i\|_{L^q(\R^n)}} + \|u\|_{H^{s,2}(\R^n)}.
\]
Since $\Omega' \subset \Omega_{2L}$ (see \eqref{eq:omeganest}), and taking into account the arguments from Step 0 of this proof, we conclude.
\end{proof}

\section{Proof of the corollaries of Theorem~\ref{th:main}}\label{s:corollaries}

\Cref{co:main2} is an immediate consequence of \Cref{th:main} and its $H^{t,q}_{loc}$-estimates.

\begin{proof}[Proof of \Cref{co:maindual}]
Let $\Omega'' \subset \subset \Omega$ with $\Omega' \subset \subset \Omega''$. Let $\eta \in C_c^\infty(\Omega)$ with $\eta \equiv 1$ in $\Omega''$. Since $f\in \brac{H^{2s-t,q'}(\Omega)}^\ast$ we have that $\tilde{f}=\eta f \in \brac{H^{2s-t,q'}(\R^n)}^\ast$, since for any $\varphi \in C_c^\infty(\R^n)$,
\[
 \langle \tilde{f},\varphi \rangle := \langle f, \eta \varphi \rangle
\]
Then $u$ is a solution of 
\[
\langle \mathcal{L}^s_{\Omega} u,\varphi\rangle = \langle \tilde{f}, \varphi \rangle \quad \forall \varphi \in C_c^\infty(\Omega'').
\]
Observe that 
\[
\langle \tilde{f},\varphi \rangle  \aleq \|f\|_{\brac{H^{2s-t,q'}(\Omega)}^\ast}\, \|\eta \varphi\|_{H^{2s-t,q'}(\R^n)}
\]
By the fractional Leibniz rule, we also have 
\[
 \|\eta \varphi\|_{H^{2s-t,q'}(\R^n)} \aleq \|\varphi\|_{H^{2s-t,q'}(\R^n)}
\]
Moreover since $q' \leq 2$ and $\eta$ has compact support,
\[
 \|\eta \varphi\|_{H^{2s-t,2}(\R^n)} \aleq \|\varphi\|_{H^{2s-t,2}(\R^n)}.
\]
In view of \Cref{pr:dualclass} we find $f_1,f_2 \in L^{q}\cap L^2(\R^n)$ such that
\[
 \langle \tilde{f}, \varphi \rangle = \int_{\R^{n}}  f_1\laps{2s-t} \varphi\, dx  + \int_{\R^{n}} f_2 \varphi\, dx \quad \forall \varphi \in C_c^\infty(\R^n),
\]
and 
\[
\|f_1\|_{L^q(\R^n)} +  \|f_2\|_{L^q(\R^n)} + \|f_1\|_{L^2(\R^n)} +  \|f_2\|_{L^2(\R^n)}\aleq \|\tilde{f}\|_{\brac{H^{2s-t,q'}(\R^n)}^\ast} \aleq \|f\|_{\brac{H^{2s-t,q'}(\Omega)}^\ast}.
\]
Thus, $u$ is a solution of 
\[
\langle \mathcal{L}^s_{\Omega} u,\varphi\rangle = \int_{\R^{n}}  f_1\laps{2s-t} \varphi\, dx  + \int_{\R^{n}} f_2 \varphi \quad \forall \varphi \in C_c^\infty(\Omega'').
\]
Applying \Cref{th:main} to this equation in $\Omega' \subsubset \Omega''$ we obtain the claim.

\end{proof}

Lastly, we show the following corollary of \Cref{th:main} for equations of the type $\mathcal{L}^{s}_{\Omega} u = \div_{s,\Omega} F$, where $\div_{s}$ denotes a fractional divergence as treated e.g. in \cite{DEG13,MS18}. Observe that \Cref{co:main3} is a direct consequence of \Cref{co:main4} if we set $F(x,y):= \frac{f(x)-f(y)}{|x-y|^s}$.
\begin{corollary}\label{co:main4}
Let  $s\in(0, 1)$ and $p\geq 2$. Let $\Omega \subsubset \R^n$ be a smoothly bounded set, and let $\Omega_1 \subsubset \Omega$ be open. Assume that $u \in W^{s,2}(\Omega)$ satisfies 
\[
 %\int_{\Omega}\int_{\Omega} \frac{K(x,y) (u(x)-u(y))\, (\varphi(x)-\varphi(y))}{|x-y|^{n+2s}}\, dx\, dy = 
 \langle \mathcal{L}^{s}_{\Omega} u, \varphi\rangle = 
 \int_{\Omega}\int_{\Omega} \frac{F(x,y)\, (\varphi(x)-\varphi(y))}{|x-y|^{n+s}}\, dx\, dy 
 \]
 for any $\varphi\in C_c^{\infty}(\Omega)$, where $\mathcal{L}^{s}_{\Omega}$ corresponds to $K\in \mathcal{K}(\alpha, \lambda, \Lambda)$ for some given $\alpha\in (0, 1)$ and $\lambda, \Lambda>0$. Then if for any $t > 0$ we have
 \[
  \int_{\Omega}\int_{\Omega} \frac{|F(x,y)|^p}{|x-y|^{n+tp}}\, dx\, dy < \infty
 \]
then for any $r \in [s,s+t) $ we have $u\in W^{t, p}_{loc}(\Omega)$, and for any $\Omega_1 \subset \Omega$ we have the estimate
\[
 [u]_{W^{r,p}(\Omega_1)} \leq C\, \brac{\brac{\int_{\Omega}\int_{\Omega} \frac{|F(x,y)|^p}{|x-y|^{n+tp}}\, dx\, dy}^{\frac{1}{p}} + \|u\|_{L^2(\Omega)} + [u]_{W^{s,2}(\Omega)}}.
\]
\end{corollary}
\begin{proof}
Set 
\[
 \Lambda := \brac{\int_{\Omega}\int_{\Omega} \frac{|F(x,y)|^p}{|x-y|^{n+tp}}\, dx\, dy}^{\frac{1}{p}}
\]
Observe that since $\Omega$ is bounded, we have for any $\tilde{t} \in [0,t]$,
\[
 \brac{\int_{\Omega}\int_{\Omega} \frac{|F(x,y)|^p}{|x-y|^{n+\tilde{t}p}}}^{\frac{1}{p}} \aleq \Lambda.
\]

Let $\Omega_2 \subsubset \Omega_3 \subset \R^n$ be an open set such that $\Omega_1 \subsubset \Omega_2 \subsubset \Omega_3 \subsubset \Omega$. Take $\eta \in C_c^\infty(\Omega)$ such that $\eta \equiv 1$ in a neighborhood of $\Omega_3$. Then for any $\varphi \in C_c^\infty(\Omega_3)$,
\[
 \begin{split}
  &\int_{\Omega}\int_{\Omega} \frac{F(x,y)\, (\varphi(x)-\varphi(y))}{|x-y|^{n+s}}\, dx\, dy \\
  =& \int_{\Omega}\int_{\Omega} \frac{F(x,y)\, (\eta(x)\varphi(x)-\eta(y)\varphi(y))}{|x-y|^{n+s}}\, dx\, dy.
 \end{split}
\]
Moreover we have for any $\varphi \in C_c^\infty(\R^n)$, and any $\tilde{t} \in [0,t]$,
\[
\begin{split}
 T[\varphi] :=& \int_{\Omega}\int_{\Omega} \frac{F(x,y)\, (\eta(x)\varphi(x)-\eta(y)\varphi(y))}{|x-y|^{n+s}}\, dx\, dy\\
 \aleq&\Lambda [\varphi]_{W^{s-\tilde{t},p'}(\R^n)}.
 \end{split}
\]
By Sobolev embedding, for any $r > s-\tilde{t}$,
\[
 [\varphi]_{W^{s-\tilde{t},p'}(\R^n)} \aleq \|\varphi\|_{H^{r,p'}(\R^n)}.
\]
That is, $T$ is an element of $(H^{r,p'}(\R^n))^\ast$, and by \Cref{pr:dualclass} we find $f_1, f_2 \in L^{p}(\R^n)$ such that 
\[
 T[\varphi] = \int_{\R^n} f_1 \laps{s} \varphi\,dx + \int_{\R^n} f_2 \varphi\,dx,
\]
with
\[
 \|f_1\|_{L^p(\R^n)}+\|f_2\|_{L^p(\R^n)} \aleq \Lambda.
\]
In particular we have for any $\varphi \in C_c^\infty(\Omega_2)$
\[
 \langle \mathcal{L}^{s}_{\Omega} u, \varphi\rangle = \int_{\R^n} f_1 \laps{s} \varphi\,dx + \int_{\R^n} f_2 \varphi\,dx,
\]
and from \Cref{th:main} we conclude that for any $W^{s,2}$-extension $\tilde{u}: \R^n \to \R$ of $u \Big |_{\Omega}$ we have
\[
 \|\laps{r} \tilde{u}\|_{L^p(\Omega_2)} \aleq \Lambda + [\tilde{u}]_{W^{s,2}(\R^n)}+\|\tilde{u}\|_{L^2(\R^n)}
\]
Again from Sobolev embedding this implies for any $0<\tilde{r} < r$
\[
 [u]_{W^{\tilde{r},p}(\Omega_1)} \aleq \Lambda + [\tilde{u}]_{W^{s,2}(\R^n)}+\|\tilde{u}\|_{L^2(\R^n)},
\]
Since $\Omega$ is an extension domain we can find an extension $\tilde{u}$ such that 
\[
  [\tilde{u}]_{W^{s,2}(\R^n)}+\|\tilde{u}\|_{L^2(\R^n)} \aleq  [u]_{W^{s,2}(\Omega)}+\|u\|_{L^2(\Omega)},
\]
and conclude the theorem.
\end{proof}

\bibliographystyle{abbrv}
\bibliography{bib}

\end{document}